\documentclass[11pt]{article}

\usepackage[pagewise]{lineno}
\usepackage{amssymb}
\usepackage{latexsym,amsfonts,amssymb,amsmath,amsthm}
\usepackage{graphicx}
\usepackage{cite}
\usepackage{amsthm}
\usepackage{amsfonts}
\usepackage{xr}
\usepackage{graphicx}
\usepackage{amsmath}
\usepackage{color}
\usepackage{amssymb}
\usepackage{mdwlist}

\newtheorem{theorem}{Theorem}[section]
\newtheorem{definition}{Definition}[section]

\newtheorem{lemma}[theorem]{Lemma}
\newtheorem{rem}[theorem]{Remark}

\numberwithin{equation}{section}

\makeatletter % `@' now normal "letter"
\@addtoreset{equation}{section}
\makeatother  % `@' is restored as "non-letter"

\newcommand\norm[1]{\lVert#1\rVert}
\newcommand\abs[1]{\lvert#1\rvert}

\begin{document}

\title{On the persistence of $k$-exponential separation of linear cocycles under a small perturbation}
\author{Lirui Feng\thanks{School of Mathematical Sciences, University of Science and Technology of China, Hefei, Anhui, 230026, People's Republic of China (ruilif@ustc.edu.cn,\,ruilif@163.com), Supported by NSF of China No.12331006 and the Chinese Academy of Sciences through grant: XDB0900100 (Sub-project: XDB0900101).}}
\date{}
\maketitle
\begin{abstract}
In this paper, the concept about a $k$-exponential separation of a linear cocycle $(\tilde{F},\mathcal{G})$ on $\tilde{K}\times X$ is extended for a general linear coclye whose base space $\tilde{K}$ and fibre-map-value map $\mathcal{G}$ become a nonempty set and a continuous map from $\tilde{K}$ to $L(X)$ respectvely, by removing the prior assumptions in the classical sense that $\tilde{K}$ is a compact set contained in the Banach space $X$, and $\mathcal{G}(x)$ is compact for any $x\in K$.  We prove that $(\tilde{F},\mathcal{G})$ on $\tilde{K}\times X$ admits a $k$-exponential separation if the cocycle $(\tilde{F},\mathcal{G})$ is generated from a linear cocycle $(F,\mathcal{T})$ on $K\times X$ adimtting a $k$-exponential separation with $K$ being compact in the classical sense via a small perturbation. We also obtain some consequent results with their needed concepts spinning off from the one of a $k$-exponential separation of $(\tilde{F},\mathcal{G})$ on $\tilde{K}\times X$, as well as the unified terminology system around $k$-exponential separation is normalized. We apply our results to analyze the linearized structure near by an invariant set of a system generated from a dissipative system via a small perturbation, where the small perturbation is without the restriction of compactness.
\end{abstract}

\section{Introduction} A $k$-exponential separation of a linear cocycle is a very important property for dynamical analysis of differential equations and dynamical systems (See e.g. \cite{C-H,F-Wu,H-P-M,L-R-S,M-S-jde03,M-S-tams-13,M-S-jmaa-13,P-1,Pu,Pu-S,P-T,S-Y,Tere}).  The linear cocycle $(\tilde{F},\mathcal{G})$ on the product space $\tilde{K}\times X$ is a map with the cocycle property (more precisely, see Def. \ref{Def-cocyle}), which can be extracted from the variational equation or the  difference equation of various nonlinear evolution equations. Roughly speaking, a $k$-exponential separation of the linear cocycle $(\tilde{F},\mathcal{G})$ is that the product space $\tilde{K}\times X$ be decomposed as a whitney sum of an invariant $k$-dimensional continuous subbundle $\tilde{K}\times (\tilde{P}_x)$ and a  positive invariant $k$-codimensional continuous subbundle $\tilde{K}\times (\tilde{Q}_x)$ with respect to $(\tilde{F},\mathcal{G})$ such that for any $x\in \tilde{K}$, the proportion of the evolution rate of vectors originating from $\tilde{P}_x\setminus\{0\}$ to the one of vectors originating from $\tilde{Q}_x\setminus\{0\}$ under the action of $(\tilde{F},\mathcal{G})$ is controlled by an exponential function (more precisely, see Def.\,\ref{k-ES}). 

In the process of deveopment of theory of $k$-exponential separation of a linear cocyle with infinite dimensional fibre spaces,  Pol\'{a}$\check{c}$ik and Tere$\check{s}\check{c}\acute{a}$k (See \cite{P-T}) proved that a linear cocyle on the certain product space admits a $1$-exponential separation in the classical sense that the base space and fibre maps of the linear cocycle are compact by utilizing strong positivity of fibre maps with respect to a solid convex cone (see Def.\,2.4 and the concepts related to a $k$-cone; see Def.\,\ref{pos-opera} for positivity and strong positivity with respect to a cone). Later on, Tere$\check{s}\check{c}\acute{a}$k (See \cite{Tere}) proved that a linear cocyle on the certain product space admits a $k$-exponential separation in the classical sense for a general positive integer $k$ under the assumption that the fibre maps of the linear cocycle are strongly positive with respect to a complemented and $k$-solid cone (see Def.\,\ref{k-cone} and the concepts related to a $k$-cone). In brief, strong positivity of a linear bounded operator $T$ with respect to a cone $C_{alg}$ is a property that $T(C_{alg}\setminus\{0\})\subset {\rm Int} C_{alg}$. Naturally, to investigate $k$-exponential separation of a linear cocyle based on the strong positivity of fibre maps of the linear cocyle is a branch of deveopment of theory of $k$-exponential separation of a linear cocycle with infinite dimensional fibre spaces. Strong positivity of these fibre maps with respect to a cone is a very important feasible condition that can be verified in various systems generated by evolution equations, but it is easy to be destroyed by a small perturbation in general. The fact draws our attention to investigating the persistence of $k$-exponential separation of a linear cocycle in itself rather than the attempt to constrain the fibre maps of the linear cocycle to keep strong positivity with respect to $k$-cones under a small perturbation with surprised additional restrictions.

Clearly, the linear cocycle $(\tilde{F},\mathcal{G})$ on $\tilde{K}\times X$ can be formed from a linear cocycle on its product space via a small perturbation on three aspects including the force-map, the base space and the fibre map-value map. In order to avoid lossing the persistence of the continuity of positive invariant bundles with respect to the linear cocyle, we assume that after a small perturbation, the force-map $\tilde{F}$ is a homoemorphism on the nonempty base space $\tilde{K}$, and both of $\tilde{F}$ and the fibre-map-value map $\mathcal{G}$ are continuous on a neighborhood of $\tilde{K}$, as well as we extend the concpet about a $k$-exponential separation for such a kind of linear cocycle (see Def.\,\ref{k-ES}) by removing the compactness assumption on $\tilde{K}$ and $\mathcal{G}(x),\,x\in\tilde{K}$ in the classical sense. Actually, the continuity of $\tilde{F}$ and $\mathcal{G}$ can not be omited in step 6, 7, 8 of our proof for the main theorem, and these steps are devoted to prove the continuity and uniqueness of the positive invariant bundles with respect to $(\tilde{F}, \mathcal{G})$ mentioned in the decomposition property referring to $k$-exponential separation (see Def.\,\ref{k-ES}). However, without the continuity of $\tilde{F}$ and $\mathcal{G}$, we prove that all properties of the linear cocycle $(\tilde{F},\,\mathcal{G})$ mentioned in $k$-exponential separation hold except the continuity of the positive invariant bundles whose whitney sum is the product space $\tilde{K}\times X$ (see Def.\,\ref{k-ES}) in step 1,2,3,4,5 of our proof for the main theorem. For convience to describe these results, we extend the concept of a linear cocycle and give the concept of $k$-exponential separation type property for a more general $(\tilde{F},\mathcal{G})$ on $\tilde{K}\times X$ that $\tilde{F}$ is just a map such that $\tilde{F}\tilde{K}=\tilde{K}$, and $\mathcal{G}$ is just a map from $\tilde{K}$ to $L(X)$ (see Remark \ref{k-ES-tp}).

Our main result in this paper is that a linear cocycle $(\tilde{F},\mathcal{G})$ on $\tilde{K}\times X$ keeps a $k$-exponential separation if it is formed from a linear cocycle $(F,\mathcal{T})$ on $K\times X$ admitting a $k$-exponential separation with a nonempty compact base space $K$ via a sufficiently small perturbation. Furthermore, with additional assumption on the compactness of $\tilde{K}$, $(\tilde{F},\mathcal{G})$ on $\tilde{K}\times X$ keeps to have the unique $k$-exponential separation (see Def.\,\ref{k-ES}). Roughly speaking, the sufficiently small perturbation means that $\tilde{K}$, $\tilde{F}$ and $\mathcal{G}$ are sufficiently closed to $K$, $F$ and $\mathcal{T}$ respectively. When we admit the maps $\tilde{F}$, $\mathcal{G}$ occurring in a wider area that $\tilde{F}$ and $\mathcal{G}$ are just maps on $\tilde{K}$ such that $\tilde{F}\tilde{K}=\tilde{K}$, we obtain the consequent results listed in the main corollary including the existence of a $k$-exponential separation type property of $(\tilde{F},\mathcal{G})$ with the its certain base space. For sake of distinguishing the allowed area of the perturbation for the force-map and fibre-map-value map in the main corollary and the one in the main theorem, the force-map and the fibre-map-value map are denoted by $\tilde{F}_{nnh}$ and $\tilde{\mathcal{G}}$ in the main corollary respectively. 

One key of our approach in the proof for the main theorem is to construct invariant cone fields (see Def. \ref{cone-field}) with respect to the linear cocycle $(\tilde{F}^n, \mathcal{G}^n)$ for any integer $n\in[N^{\prime},\,2N^{\prime}-1]$ with a $N^{\prime}\in\mathbb{N}^+$ such that each fibre map of $(\tilde{F}^n, \mathcal{G}^n)$ at $x\in\tilde{K}$, $G^n_x$ is strongly focusing (see Def. \ref{SF-operator}) on the cone of the certain invariant cone field at $x$ with respect to the cone of the certain invariant cone field at $\tilde{F}^n(x)$, where $(\tilde{F}^n, \mathcal{G}^n)$ is the $n$-iteration of the linear cocycle $(\tilde{F}, \mathcal{G})$. Roughly speaking, the image of $C$ of a strongly focusing operator on $C$ with respect to $\tilde{C}$ is a subset of $\tilde{C}$ such that unite vectors contained in it are uniformly separated from the boundary of $\tilde{C}$. A strongly focusing operator originated from Krasnosel\'skij et.al \cite{K-L-S} to prove a Krein-Rutman type theorem with respect to $k$-cones for a single operator, and also from Lian and Wang \cite{L-W} to investigate the relationship between Multiplicative Ergodic Theorem and Krein-Rutman type Theorem for random linear dynamical systems. An important property of a strongly focusing operator is that angle-numbers of two vectors or two linear subspaces (see Def. \ref{angles}) will be contracted under the action of this kind of operators, that implies the certain distance of two vectors or two linear subspaces being contracted under its action. Without injective assumption on a bounded linear operator $T$ on $X$, we give and prove precise properties of angle-numbers themselves and the variation of angle-numbers under the action of $T$ in Lemma \ref{P-alp-0}-\ref{SF-angles}, which are matched to our construction of the cone fields and used to prove the existence of the unique invariant $k$-dimensional bundle with respect to the linear cocycle $(\tilde{F}^n,\,\mathcal{G}^n)$ located in the certain cone field in step 1 and step 2 of the proof for the main theorem respectively. By an argument on the unique invariant $k$-dimensional bundle with respect to different $n$-iteration of $(\tilde{F},\mathcal{G})$ located in the certain cone field, we obtain that there is a unique invariant $k$-dimensional bundle $\tilde{K}\times (\tilde{P}_x)$ with respect to $(\tilde{F},\mathcal{G})$ located in the cone field $\tilde{K}\times (C_x)$. The existence of the positive invariant $k$-codimensional bundle $\tilde{K}\times (\tilde{Q}_x)$ with respect to $(\tilde{F},\mathcal{G})$ is proved in step 4 of the proof for the main theorem based on the construction of the correction maps in step 3 of the proof for the main theorem which is motivated from \cite{L-W} and need complex estimates to deal with the situation that the strong focusing operators are taked from fibre maps of $N^{\prime}$-iteration of the linear cocycle $(\tilde{F},\mathcal{G})$. The separation property of the linear cocycle $(\tilde{F},\mathcal{G})$ proved in the step 5 is also based on the estimates for the construction of the correction maps in step 3. By the previous five steps of the proof for the main theorem, we actually prove that $(\tilde{F},\mathcal{G})$ on $\tilde{K}\times X$ has a $k$-exponential separation type property without the assumption that $\tilde{F}$ is a homoemorphism on $\tilde{K}$ and $\mathcal{G}$ is continuous on a neighberhood of $\tilde{K}$. This result becomes one of the bases for our main corollary.

To prove the continuity of the invariant $k$-dimensional bundle $\tilde{K}\times (\tilde{P}_x)$ and the positive invariant $k$-codimensional bundle $\tilde{K}\times (\tilde{Q}_x)$ is one of dificulties in this paper. The cone fields we construct in step 1 of the proof for the main theorem are based on the $k$-dimensional bundle $\tilde{K}\times (P_{h(x)})$ and the $k$-codimensional bundle $\tilde{K}\times (Q_{h(x)})$, which are shifted from the invariant $k$-dimensional continuous bundle $K\times (P_y)$ and the positive invariant $k$-codimensional continuous bundle $K\times (Q_y)$ with respect to the oringinal linear cocycle $(F,\mathcal{T})$ on $K\times X$. But the shift-map $h$ is just a map to describe $\tilde{K}$ being slightly perturbated from $K$. The irregularity of the slight shift $h$ between the base spaces $K$ and $\tilde{K}$ leads the huge difficulty to construct a $k$-dimensional continuous bundle contained in the cone fields constructed by us in step 1 of the proof for the main theorem as the initial point of a cauchy sequence of $k$-dimensional continuous bundles on $\tilde{K}$ generated by the force of $Nn$-th iteration $(\tilde{F}^{Nn},\,\mathcal{G}^{Nn})$ of $(\tilde{F},\mathcal{G})$ for any integer $n\in\mathbb{N}^+$ and $N\in [N^{\prime}, 2N^{\prime}-1]$ with a $N^{\prime}\in \mathbb{N}^+$ in step 2 of the proof for the main theorem. We discard the way through a cauchy sequence of $k$-dimensional continuous bundles generated by the force of $Nn$-th iteration of $(\tilde{F},\mathcal{G})$ to obtain the invariant $k$-dimensional continuous bundle with respect to $(\tilde{F},\mathcal{G})$. We turn to utilize the $k$-exponential separation type property of the linear cocycle $(\tilde{F},\mathcal{G})$ proved in the previous five steps to explore the approach to prove the continuity of the invariant $k$-dimensional bundle $\tilde{K}\times (\tilde{P}_x)$ and the positive invariant $k$-codimensional bundle $\tilde{K}\times(\tilde{Q}_x)$ with respect to the linear cocycle $(\tilde{F},\mathcal{G})$. For this purpose, a key we grasp from all information is that a sequence of fibre $\tilde{P}_{x_m}$ of the invariant $k$-dimensional bundle $\tilde{K}\times (\tilde{P}_x)$ at the point $x_m$ with $m\in\mathbb{N}^+$ such that all fibres $\tilde{P}_{x_m},\,m\in\mathbb{N}^+$ are uniformly far away from $\tilde{P}_x$ and $x_m$ converges to the supposed discontinuous point $x$ of $\tilde{K}\times  (\tilde{P}_x)$ will be turned to approaching the fibre $\tilde{Q}_{\tilde{F}^{-n}(x)}$ of the positive invariant $k$-codimensional bundle $\tilde{K}\times (\tilde{Q}_x)$ at the point $\tilde{F}^{-n}(x)$ uniformly with respect to sufficiently large $m\in\mathbb{N}^+$ by the action of the inverse on $\tilde{P}_{\tilde{F}^{-n}(x_m)}$ of the fibre map $G_{\tilde{F}^{-n}(x_m)}^n$ of the $n$-iteration $(\tilde{F}^n,\,\mathcal{G}^n)$, as well as a similiar phenonmenon for a sequence of fibre $\tilde{Q}_{x_m}$ of the positive invariant $k$-codimensional bundle $\tilde{K}\times (\tilde{Q}_x)$ exists. Another key to prove the continuity of $\tilde{K}\times (\tilde{P}_x)$ and $\tilde{K}\times (\tilde{Q}_x)$ is the law of variation of the proportion of the projection of a vector in fibre space on the fibre $Q_{h(x)}$ to the one on the fibre $P_{h(x)}$ as $x$ varies in its small neighborhood, which is revealed by a containing relationship (\ref{Is-cf-cndelta}) among the parameterised cone fields we constructed in step 6 of the proof for the main theorem. This law implies that for any point $x^{\prime}$ in a neighborhood of the preimage $\tilde{F}^{-n}(x)$ of the supposed discontinuous point $x\in \tilde{K}$ small enough,  the proportion of the projection of a unit vector in $\tilde{P}_{x^{\prime}}$ on $Q_{h(\tilde{F}^{-n}(x))}$ to the one on $P_{h(\tilde{F}^{-n}(x))}$ is controlled by a upper bound number $\frac{13}{17}$. Together with $\tilde{P}_{\tilde{F}^{-n}(x_m)}$ and $\tilde{F}^{-n}(x_m)$ approaching $\tilde{Q}_{\tilde{F}^{-n}(x)}$ and $x$ respectively, it furthermore implies that there exists at least one unit vector in $\tilde{Q}_{\tilde{F}^{-n}(x)}$ such that the proportion of its projection on $Q_{h(\tilde{F}^{-n}(x))}$ to the one on $P_{h(\tilde{F}^{-n}(x))}$ is with a upper bound number sufficiently closed to $\frac{13}{17}$, i.e., the proportion of its projection on $P_{h(\tilde{F}^{-n}(x))}$ to the one on $Q_{h(\tilde{F}^{-n}(x))}$ is with a lower bound number sufficiently closed to than $\frac{17}{13}$. We know the following three facts: (i) the correction map $\Psi_{\tilde{F}^{-n}(x)}$ for $\tilde{F}^{-n}(x)\in \tilde{K}$ maps vectors in $Q_{h(\tilde{F}^{-n}(x))}$ to $\tilde{P}_{\tilde{F}^{-n}(x)}$ with upper bound $M_4>0$ of $\Psi_{\tilde{F}^{-n}(x)}$ (see (\ref{Covergence-norm-Psi})); (ii) $\tilde{Q}_{\tilde{F}^{-n}(x)}=\{u-\Psi_{\tilde{F}^{-n}(x)}u:\,u\in Q_{h(\tilde{F}^{-n}(x))}\}$ (see (\ref{Def-tiQ})); (iii) the proportion of the projection of a unit vector in $\tilde{P}_{\tilde{F}^{-n}(x)}$ on $Q_{h(\tilde{F}^{-n}(x))}$ to the one on $P_{h(\tilde{F}^{-n}(x))}$ is less than 1 (see (\ref{def-Cx}) and (\ref{tP-inner})). But it is very difficult to get a contradiction on the proportion of the projection of a unit vector in $\tilde{Q}_{\tilde{F}^{-n}(x)}$ on $P_{h(\tilde{F}^{-n}(x))}$ to the one on $Q_{h(\tilde{F}^{-n}(x))}$ by improving the estimates on the upper bound $M_4$. A crucial property $\tilde{Q}_{x^{\prime}}\cap C_{x^{\prime}}=\{0\}$ for any $x^{\prime}\in\tilde{K}$ (see (\ref{loc-tilide-Q})) is dug out by us to connect segments of logical deduction for proving the continuity of $\tilde{K}\times(\tilde{P}_x)$. By virtue of this property, the proportion of the projection of a unit vector in $\tilde{P}_{\tilde{F}^{-n}(x_m)}$ on $Q_{h(\tilde{F}^{-n}(x))}$ to the one on $P_{h(\tilde{F}^{-n}(x))}$ exceeds $\frac{13}{17}$ for any $x_m$ sufficiently closed to $x$, a contradiction to the proportion no more than $\frac{13}{17}$. It follows that $\tilde{K}\times (\tilde{P}_x)$ is continuous. The outline of proof for the continuity of $\tilde{K}\times (\tilde{Q}_x)$ in step 7 is similiar except we need to make clear several properties on the connections among the continuity of $k$-dimensional bundle $\tilde{K}\times (L_x)\subset \tilde{K}\times X^*$ such that ${\rm Ker}\{L_x\}=\tilde{Q}_x$ and the distances among the $k$-codimensional fibre spaces $\tilde{Q}_x$ (see the assertions (As-1)-(As-3) in step 7). The property $\tilde{Q}_{x^{\prime}}\cap C_{x^{\prime}}=\{0\}$ for any $x^{\prime}\in\tilde{K}$ also plays the important role to connect segments of logical deduction in step 7 of the proof for the main theorem. The proof in step 8 is from another paper of ours \cite{F-Wang}, and the results in main corollary are generated from the applications of the main theorem in different cases that are endowed different special invariant sets w.r.t the dynamical system $(\tilde{F}_{nnh},\,\tilde{K})$ as the base spaces. In this paper, we succeed in applying our main theorem and corollary to dissipative systems under a small perturbation, and obtain the persistence of $k$-exponential separation for dissipative systems. 

This paper is organized as follows.  In section 2, we give some definitions and notations. In section 3, we present our main theorem and corollary. In section 4, we prove the main theorem. In section 5, we prove the main corollary. In section 6, we apply our main results to dissipative systems under a small perturbation.

\section{Definitions and Notations}
Let $X$ be a Banach space equipped with the norm $\norm{\cdot}$, and $X^*$ be its dual space with the norm $\norm{\cdot}_*$, where $\norm{l}_*=\sup\limits_{v\in X,\norm{v}=1}\{\abs{l(v)}\}$ for any $l\in X^*$. Denoted by $L(X)$ the space of all linear bounded operators on $X$, equipped with the usual operators' norm $\norm{\cdot}_{L(X)}$. Let $G(k,X)$ be {\it the Grassmanian of $k$-dimensional linear subspaces of $X$}, which consists of all $k$-dimensional linear subspaces in $X$. $G(k, X)$ is a completed metric space by endowing {\it the gap metric} $d$ (see \cite{K, L-L}). More precisely, the gap metric $d$ is defined as

\begin{equation*} d(H_1,H_2)=\max\left\{\sup_{v\in H_1\cap S}\inf_{u\in H_2\cap S}\norm{v-u}, \sup_{v\in H_2\cap S}\inf_{u\in H_1\cap S}\norm{v-u}\right\},
 \end{equation*} for any nontrivial
closed subspaces $H_1, H_2\subset X$. By the same way, we define the gap metric $d^*$ for the Grassmanian $G(k, X^*)$ of $k$-dimensional linear subspaces of $X^*$. Denoted by $S=\{v\in X:\norm{v}=1\}$ and $S^*=\{l\in X^*:\norm{l}_*=1\} $ the unit spheres of $X$ and $X^*$ respectively. 

Let $\tilde{K}$ be a nonempty subset of $X$ in the rest of this paper. Let $\{\tilde{P}_x\}_{x\in \tilde{K}}$ be a set of $k$-dimensional linear subspaces of $X$, and $\{\tilde{Q}_x\}_{x\in \tilde{K}}$ be a set of $k$-codimensional closed linear subspaces of $X$.

\begin{definition}
{\rm (i)} $\tilde{K}\times (\tilde{P}_x)$ is called as {\it a $k$-dimensional vector bundle} on $\tilde{K}$ {\rm(}for short, $k$-dimensional bundle{\rm)} if $\tilde{P}_x\in G(k,\,X)$ for any $x\in \tilde{K}$, and it is called as {\it a $k$-dimensional continuous vector bundle} on $\tilde{K}$ {\rm(}for short, $k$-dimensional continuous bundle{\rm)} if the map $\tilde{K}\mapsto G(k,\,X): x\mapsto \tilde{P}_x$ is continuous in additional;

{\rm (ii)} $\tilde{K}\times (\tilde{Q}_x)$ is called as {\it a $k$-codimensional vector bundle on $\tilde{K}$} {\rm(}for short, $k$-codimensional bundle{\rm)} if there is a $k$-dimensional vector bundle $\tilde{K}\times (L_x)\subset\tilde{K}\times X^*$ such that the kernel ${\rm Ker}(L_x)=\tilde{Q}_x$ for each $x\in \tilde{K}$, and it is called as {\it a $k$-codimensional continuous vector bundle on $\tilde{K}$} {\rm(}for short, $k$-codimensional continuous bundle{\rm)} if the $k$-dimensional vector bundle $\tilde{K}\times (L_x)\subset \tilde{K}\times X^*$ is continuous such that ${\rm Ker}(L_x)=\tilde{Q}_x$ for each $x\in \tilde{K}$. Hereafter, the continuity of $x:\mapsto \tilde{Q}_x$ refers to $x:\mapsto L_x$ being continous from $\tilde{K}$ to $G(k,X^*)$.
\end{definition} 
 
Let $\tilde{K}\times (\tilde{P}_x)$ be a $k$-dimensional continuous bundle on $\tilde{K}$, and $\tilde{K}\times (\tilde{Q}_x)$ be a $k$-codimensional continuous bundle
on $\tilde{K}$ such that $X=\tilde{P}_x\oplus \tilde{Q}_x$ for all  $x\in \tilde{K}$. $\tilde{P}_x$ (resp. $\tilde{Q}_x$) is called {\it the fibre of $\tilde{K}\times(\tilde{P}_x)$ {\rm(}resp. $\tilde{K}\times (\tilde{Q}_x)${\rm)} at $x\in \tilde{K}$}. Let $\mathring{C}$ be a set containing at least a $k$-dimensional linear subpace of $X$ and denote by $G_k(\mathring{C})$ the set consisting of all $k$-dimensional linear subspaces contained in $\mathring{C}$. Denoted by $C(\tilde{K},\,G_k(\mathring{C}))$ the set of all k-dimensional continuous bundle on $\tilde{K}$ such that its fibres belong to $G_k(\mathring{C})$. For each $x\in \tilde{K}$, $X=\tilde{P}_x\oplus \tilde{Q}_x$ implies that the natural projection of $X$ onto $\tilde{P}_x$ along $\tilde{Q}_x$, denoted by $\Pi^{\tilde{P}_x}_{\tilde{Q}}$, is well-defined. Write $\Pi^{\tilde{Q}_x}_{\tilde{P}}={\rm Id}-\Pi^{\tilde{P}_x}_{\tilde{Q}}$ for each $x\in \tilde{K}$, where ${\rm Id}$ is the identity map on $X$. Clearly, $\Pi^{\tilde{Q}_x}_{\tilde{P}}$ is the natural projection of $X$ onto $\tilde{Q}_x$ along $\tilde{P}_x$.

Let $\tilde{F}$ be a homeomorphism on $\tilde{K}$. Clearly, $\tilde{F}\overline{\tilde{K}}=\overline{\tilde{K}}$, where $\overline{\cdot}$ represents to take the closure of a subset of $X$. Let $\mathcal{G}$ be a continuous map from $\mathcal{D}(\mathcal{G})$ to $L(X)$ such that $\tilde{K}\subset \mathcal{D}(\mathcal{G})$ and $\mathcal{G}(x)=G_x$ for any $x\in \mathcal{D}(\mathcal{G})$. Let $\tilde{F}^0={\rm Id}$ and $\tilde{F}^n=\tilde{F}\circ \tilde{F}^{n-1},\,n\in\mathbb{N}^+$; and more, the inverse of $\tilde{F}^n$ is denoted by $\tilde{F}^{-n}$ for any $n\in\mathbb{N}^+$. Denoted by $G^0_x={\rm Id}$ and $G^n_x=G_{\tilde{F}^{n-1}(x)}\circ G^{n-1}_x$ for any $n\in\mathbb{N}^+$ such that $\tilde{F}^i(x)\in\mathcal{D}(\mathcal{G})$ with $i\in\{0,\,1,\,2,\cdots,n-1\}$. For any $n\in \mathbb{N}^+$, define the map $\mathcal{G}^n$ from $\mathcal{D}(\mathcal{G}^n)$ to $L(X)$ by $\mathcal{G}^n(x)=G^n_x$ for any $x\in \mathcal{D}(\mathcal{G}^n)$, where $ \mathcal{D}(\mathcal{G}^n)=\{x\in \mathcal{D}(\mathcal{G}):\,\,{\rm such\,\, that}\,\, \tilde{F}^i(x)\in\mathcal{D}(\mathcal{G})\,\,{\rm for\,\, any}\,\,i\in\{0,\,1,\,2,\cdots,n-1\}\} $.

\begin{definition}\label{Def-cocyle} The linear cocycle of $\tilde{F}$ and $\mathcal{G}$ on $\tilde{K}\times X$, denoted by $(\tilde{F},\mathcal{G})$, is the map

$$\tilde{K}\times X\rightarrow \tilde{K}\times X: \,\,(x,v)\mapsto (\tilde{F}(x), G_x v).$$ Here, $\tilde{F}$ {\rm(}resp. $\mathcal{G}${\rm)} are called as the force-map {\rm(}resp. the fibre-map-value map{\rm)} of $(\tilde{F},\mathcal{G})$; $\mathcal{G}(x)$ is called as the fibre map of $(\tilde{F},\mathcal{G})$ {\rm(}at $x\in \tilde{K}${\rm)}; $\tilde{K}$ {\rm(}resp. $X${\rm)} are called the base {\rm(}resp. fibre{\rm)} space of the linear cocycle $(\tilde{F},\mathcal{G})$ on $K\times X$. It is clear that the $n$-iteration of $(\tilde{F},\mathcal{G})$ is also a linear cocycle on $\tilde{K}\times X$, which equals to $(\tilde{F}^n, \mathcal{G}^n)$ for any $n\in\mathbb{N}^+$.

For sake of roughly speaking, we also say a linear cocycle as a map with the cocycle property.
\end{definition}

\begin{definition}\label{k-ES} $(\tilde{F},\mathcal{G})$ on $\tilde{K}\times X$ admits a $k$-exponential separation {\rm(}abbr. $k$-ES{\rm)} {\rm(}or, a $k$-exponential separation of $(\tilde{F},\mathcal{G})$ on 
$\tilde{K}\times X$ {\rm)} if {\rm(}or, is a property such that{\rm)}

 {\rm (i)} {\bf (Decomposition)} there are a $k$-dimensional continuous bundle $\tilde{K}\times (\tilde{P}_x)$ and a $k$-codimensional continuous bundle $\tilde{K}\times (\tilde{Q}_x)$ such that $X=\tilde{P}_x\oplus \tilde{Q}_x$ for any $x\in \tilde{K}$;

{\rm (ii)} {\bf (Invariance)} $G_x \tilde{P}_x=\tilde{P}_{\tilde{F}(x)}$ and $G_x \tilde{Q}_x\subset \tilde{Q}_{\tilde{F}(x)}$ for any $x\in \tilde{K}$;

{\rm (iii)} {\bf (Separation)} there are constants $\tilde{M}>0$, $\tilde{\gamma}\in (0,1)$ such that 

\begin{equation}\label{Separation}\norm{G^n_x w}\leq \tilde{M} \tilde{\gamma}^n \norm{G^n_x v}\end{equation} for any $v\in \tilde{P}_x\cap S$, $w\in \tilde{Q}_x\cap S$, $x\in \tilde{K}$ and $n\in \mathbb{N}$. 

In additional, if $\tilde{K}\times (\tilde{P}_x)$ and $\tilde{K}\times (\tilde{Q}_x)$ are the unique pair of bundles satisfying Def. {\rm\ref{k-ES}(i)-(iii)}, it is called as $(\tilde{F},\mathcal{G})$ on $\tilde{K}\times X$ has the unique $k$-exponential separation {\rm(}or, the unique $k$-exponential separation of $(\tilde{F},\mathcal{G})$ on $\tilde{K}\times X$ {\rm)}.
\end{definition} For short, $(\tilde{F},\mathcal{G})$ on $\tilde{K}\times X$ in Def. \ref{k-ES} is written as $(\tilde{F},\mathcal{G})$.

\begin{rem}In the sense of Def. \ref{k-ES}, $\tilde{K}$ is just a nonempty subset of $X$. However, in the classical sense, a linear cocycle $(\tilde{F},\mathcal{G})$ on $\tilde{K}\times X$ admitting a $k$-exponential separation has more restricted assumptions as follows: 

{\rm (i)} $\tilde{K}$ needs to be restricted as a nonempty compact subset of $X$;

{\rm (ii)} $\mathcal{G}(x)$ is compact for any $x\in K$. 
\end{rem}

\begin{rem}\label{k-ES-tp} Definition \ref{Def-cocyle} and \ref{k-ES} are also applicable for a more general case that the force-map $\tilde{F}$ is a map on $\tilde{K}$ such that $\tilde{F}\tilde{K}=\tilde{K}$ and the fibre-map-value map $\mathcal{G}$ is a map from $\mathcal{D}(\mathcal{G})(\supset \tilde{K})$ to $L(X)$ without the continuous assumptions on $\tilde{F}$ and $\mathcal{G}$. In this case, we have concepts spinning off from a $k$-exponential separation as follows:

{\rm (i)} $(\tilde{F},\mathcal{G})$ on $\tilde{K}\times X$ has a $k$-exponential separation type property if $\tilde{K}\times (\tilde{P}_x)$ and $\tilde{K}\times (\tilde{Q}_x)$ satisfy Def. {\rm\ref{k-ES}(i)-\ref{k-ES}(iii)} except $\tilde{K}\times (\tilde{P}_x)$ and $\tilde{K}\times (\tilde{Q}_x)$ are continuous.

{\rm (ii)} $(\tilde{F},\mathcal{G})$ on $\tilde{K}\times X$ has the unique $k$-exponential separation type property if $\tilde{K}\times (\tilde{P}_x)$ and $\tilde{K}\times (\tilde{Q}_x)$ are the unique pair of bundles such that $(\tilde{F},\mathcal{G})$ on $\tilde{K}\times X$ has a $k$-exponential separation type property. 

Here, $(\tilde{F},\mathcal{G})$ on $\tilde{K}\times X$ is also written as $(\tilde{F},\mathcal{G})$ for short.
\end{rem}

\begin{rem} In the general case the force-map $\tilde{F}$ being a map on $\tilde{K}$ such that $\tilde{F}\tilde{K}=\tilde{K}$ and the fibre-map-value map $\mathcal{G}$ being a map from $\mathcal{D}(\mathcal{G})(\supset \tilde{K})$ to $L(X)$ without the continuous assumptions on $\tilde{F}$ and $\mathcal{G}$, we have the following terminologies:

{\rm (i)} A nonempty subset $\tilde{\tilde{K}}\subset \tilde{K}$ such that $\tilde{F}\tilde{\tilde{K}}=\tilde{\tilde{K}}$ {\rm(}resp. $\tilde{F}\tilde{\tilde{K}}\subset\tilde{\tilde{K}}${\rm )} is called as invariant {\rm(}resp. positive invariant{\rm)} with respect to $\tilde{F}$. A $k$-dimensional bundle $\tilde{K}\times (\tilde{P}_x)$ {\rm (resp.} $k$-codimensional bundle $\tilde{K}\times (\tilde{Q}_x)${\rm)} is called as {\it positive invariant with respect to $(\tilde{F},\mathcal{G})$} if $G_x\tilde{P}_x\subset \tilde{P}_{\tilde{F}(x)}$ {\rm(resp.} $G_x\tilde{Q}_x\subset \tilde{Q}_{\tilde{F}(x)}${\rm)} for any $x\in \tilde{K}$; and more,  $\tilde{K}\times (\tilde{P}_x)$ {\rm(resp.} $\tilde{K}\times (\tilde{Q}_x)${\rm)} is called as invariant with respect to $(\tilde{F},\mathcal{G})$ if $G_x\tilde{P}_x=G_x\tilde{P}_{\tilde{F}(x)}\,\, ({\rm resp.} \,\, G_x\tilde{Q}_x=\tilde{Q}_{\tilde{F}(x)})$ for any $x\in \tilde{K}$. 

\vskip 3mm
{\rm (ii)} $(\tilde{F},\tilde{K})$ is called as {\it a dynamical system of $\tilde{F}$ and $\tilde{K}$}. The set $O^+_{(\tilde{F},\tilde{K})}(x)=\{\tilde{F}^i(x):\,\,i\in\mathbb{N}\}$ is called as {\it the positive semiorbit of $(\tilde{F},\tilde{K})$ with initial point $x\in\tilde{K}$.} A set $O^-_{n,(\tilde{F},\tilde{K})}(x)=\{y_i\in \tilde{K}:\,\, {\rm such\,\, that}\,\, y_0=x \,\,{\rm and}\,\, \tilde{F}(y_{i+1})=y_i\,\,{\rm for\,\, each}\,\, i\in\mathbb{N}\}$, is called as {\it a negative semiorbit of $(\tilde{F},\tilde{K})$ with initial point $x\in\tilde{K}$;} and the set $O_{f,(\tilde{F},\tilde{K})}(x)=O^+_{(\tilde{F},\tilde{K})}(x)\cup O^-_{n,(\tilde{F},\tilde{K})}(x)$ is called as {\it a full-orbit of $(\tilde{F},\tilde{K})$ with initial point $x\in\tilde{K}$.} Since $\tilde{F}\tilde{K}=\tilde{K}$, there is at least one negative semiorbit of $(\tilde{F},\tilde{K})$ with initial point $x\in\tilde{K}$ for each $x\in \tilde{K}$. When there is a unique negative semiorbit of $(\tilde{F},\tilde{K})$ with initial point $x\in \tilde{K}$, we denote it as $O^-_{(\tilde{F},\tilde{K})}(x)$; it is then clear that there is a unique full-orbit of $(\tilde{F},\tilde{K})$ with initial point $x\in \tilde{K}$ and we denote it as $O_{(\tilde{F},\tilde{K})}(x)$.

\vskip 3mm
{\rm (iii)} A nonempty set $\tilde{\tilde{K}}\subset \tilde{K}$ is called {\it minimal with respect to $(\tilde{F},\tilde{K})$} if both of $\tilde{\tilde{K}}$ and $\tilde{K}$ are closed, $\tilde{\tilde{K}}$ is invariant w.r.t. $\tilde{F}$, and there is no nonempty closed invariant proper subset of $\tilde{\tilde{K}}$ w.r.t. $\tilde{F}$.

For a positive invariant set $X_0\subset X$ w.r.t. $\tilde{F}$, $(\tilde{F},X_0)$ is called as {\it a dynamical system of $\tilde{F}$ and $X_0$}. The concepts of semiorbits, full-orbits and minimal sets for $(\tilde{F},X_0)$ are defined by the same way for the ones for $(\tilde{F},\tilde{K})$, and the corresponding descriptions are from the ones for $(\tilde{F},\tilde{K})$ by replacing $\tilde{K}$ as $X_0$.
\end{rem}

\begin{definition}\label{k-cone} A closed subset $C\subset X$ is called as {\it a cone of rank $k$ \rm{(}abbr. $k$-cone\rm{)}} if:

{\rm (i)} $\mathbb{R}\cdot C\subset C$; 

{\rm (ii)} $\max\{\text{dim}\,H: \,\,H\subset C\,\, \text{is a linear subspace} \}=k$. 
\end{definition} A $k$-cone $C$ is called as {\it solid} if its interior $\text{Int}C\neq \emptyset$, and is called as {\it k-solid} if there is a $k$-dimensional linear subspace $H$ such that $H\setminus\{0\}\subset \text{Int}C$. Moreover, it is called as {\it complemented} if there is a closed $k$-codimensional linear subspace $H^c$ such that $H^c\setminus \{0\}\subset X\setminus C$. We recall here that a {\it convex cone} $X^+$ in $X$ is a nonempty closed subset of $X$ such that (i) $\mathbb{R}^+\cdot X^+\subset X^+$, (ii) $x+y\in X^+$ for any $x,y\in X^+$, (iii) $X^+\cap (-X^+)=\{0\}$. Furthermore, $X^+\cup (-X^+)$ is a 1-cone in $X$.

Given a $k$-cone $C$. Denoted by $G_k({\rm Int} C)$ the set of $k$-dimensional linear subspaces inside ${\rm Int} C\cup \{0\}$ for $C$ being $k$-solid, i.e.,

$$G_k({\rm Int} C)=\{H\in G(k,X):\,H\setminus\{0\}\subset {\rm Int} C\}.$$ For $C$ being a complemented and $k$-solid cone, denoted by

$$G_k(C^*)=\{L\in G(k,\,X^*):\, {\rm Ker}(L)\cap C=\{0\}\}.$$ 

\begin{definition}\label{angles} Let $C$ be a complemented and $k$-solid cone. Define the following angle-numbers

{\rm (i)} $\alpha^C_0(u,v)=\inf\{\alpha\geq 0:\,\beta v-u\in {\rm Int}C\,\,\text{for all}\,\,\beta\geq \alpha\}$ for any $u,v\in{\rm Int} C$.

{\rm (ii)} $\alpha^C(u,v)=\alpha^C_0(u,v)\alpha^C_0(v,u)$ for any $u,v\in{\rm Int} C$.

{\rm (iii)} $ \alpha^C(P_1,P_2)=\sup\{\alpha^C(u,v):\,u\in P_1\setminus \{0\},\,\,v\in P_2\setminus \{0\}\}$ for any $P_1,\,P_2\in G_k({\rm Int}\,C)$.
\end{definition}

Clearly, $\alpha^C_0(u,u)=1$ and $\alpha^C(u,v)=\alpha^C(v,u)=\alpha^C(\frac{u}{\norm{u}},\frac{v}{\norm{v}})$ for any $u, v\in {\rm Int} C$. 

\begin{definition}\label{pos-opera} A bounded linear operator $T\in L(X)$ is called {\rm(}resp. strongly{\rm)} positive with respect to a cone $C_{alg}$ if {\rm(}resp. $T(C_{alg}\setminus\{0\})\subset {\rm Int} C_{alg}${\rm)} $TC_{alg}\subset C_{alg}$, where $C_{alg}$ is chosen as a convex cone $X^+$ or a $k$-cone $C$.
\end{definition}

Denoted by $\mathcal{SP}(C,\tilde{C})\subset L(X)$ the set of all linear operators such that $T(C\setminus\{0\})\subset {\rm Int} \tilde{C}$, where $C,\tilde{C}$ are two complemented and $k$-solid cones. 

\begin{definition}\label{SF-operator} Let $C,\tilde{C}$ are two complemented and $k$-solid cones. $T$ is called strongly focusing on $C$ with respect to $\tilde{C}$ if $T\in  \mathcal{SP}(C,\tilde{C})$ and 

$$\kappa=\inf\limits_{v\in TC\cap S}\{\inf\limits_{u\in X\setminus\tilde{C} }\{\norm{v-u}\}\}>0.$$ Here, $\kappa$ is called as the separation index of $T$ on $C$ with respect to $\tilde{C}$ {\rm(}for short, the separation index of $T${\rm)}.
\end{definition} Clearly, the separation index of a strongly focusing operator $T$ on $C$ with respect to $\tilde{C}$ is in $(0,1]$. Denoted by $\mathcal{SF}(C,\tilde{C})$ the set of all strongly focusing operators on $C$ w.r.t. $\tilde{C}$.

\begin{definition}\label{cone-field}Given a $k$-cone $C_x$ for each $x\in \tilde{K}$, $\tilde{K}\times (C_x)$ is called as {\it a cone field on $\tilde{K}$ of rank $k$ (for short, a cone field)}; and more, is called as {\it a positive invariant cone field with respect to $(\tilde{F},\mathcal{G})$ on $\tilde{K}\times X$} of rank $k$ (for short, an invariant cone field) if $G_x C_x\subset C_{\tilde{F}(x)}$ for any $x\in \tilde{K}$, where $\tilde{F}$ is of the general case that $\tilde{F}$ is a map such that $\tilde{F}\tilde{K}=\tilde{K}$. Here, $C_x$ is called as {\it the cone of $\tilde{K}\times (C_x)$ at $x\in \tilde{K}$}.
\end{definition}

\section{Main results}
Throughout this paper, let $K$ be a nonempty compact set of $X$ and $B_{c}(K)=\{x\in X: \inf\limits_{y\in K}\norm{x-y}\leq c\}$ for any $c\geq 0$. Let $F$ be a self-map on $X$, which is a homeomorphism on $K$ and a continuous map on $B_1(K)$. Let $\mathcal{T}$ be a continuous map from $B_1(K)$ to $L(X)$ defined by $\mathcal{T}(x)=T_x\in L(X)$ for each $x\in B_1(K)$. Then, the compactness of $K$ impllies that 

$$\sup\limits_{x\in K}\norm{T_x}_{L(X)}<+\infty.$$ Denoted by $T^0_x={\rm Id}$ and $T^n_x=T_{F^{n-1}(x)}\circ T^{n-1}_x$ for any $n\in\mathbb{N}^+$ and $x\in K$, where $F^0={\rm Id}$ and $F^n=F\circ F^{n-1},\,n\in\mathbb{N}^+$.

In the rest of the paper, we have the following setting: assume that $\mathcal{G}$ is continuous on $\mathcal{D}(\mathcal{G})$ such that $\tilde{K}\subset B_1(K)\subset \mathcal{D}(\mathcal{G})$; let ${\it h}$ be a map from $\tilde{K}$ to $K$; assume that $\tilde{F}$ is continuous on $B_1(K)$ and a homeomorphism on $\tilde{K}$; assume that 

\begin{equation}\label{Pertu-K-F}\sup\limits_{x\in \tilde{K}}\{\norm{h(x)-x}\},\,\sup\limits_{x\in B_{1}(K)}\{\norm{\tilde{F}(x)-F(x)}\}\leq \delta.\end{equation} Here, we call the map $h$ as {\it a shift-map between $\tilde{K}$ and $K$}. Let $C(B_1(K),L(X))$ consist of all continuous map on $B_1(K)$ with value in $L(X)$. The norm $\norm{\cdot}_{C_{KL(X)}}$ is defined by $\norm{\mathcal{\tilde{G}}}_{C_{KL(X)}}=\sup\limits_{x\in B_1(K)}\{\norm{\mathcal{\tilde{G}}(x)}_{L(x)}\}$ for any map $\mathcal{\tilde{G}}$ on $B_1(K)$ with value in $L(X)$. Denote

\begin{equation}\label{Pertu-fibre-map}\varepsilon=\norm{\mathcal{T}-\mathcal{G}}_{C_{KL(X)}}.\end{equation} We always assume that:

\vskip 3mm
 {\rm {\bf(H)}} The linear cocycle $(F,\mathcal{T})$ on $K\times X$ admits a $k$-exponential separation. 

\vskip 3mm
\noindent{\bf Main Theorem:} $(\tilde{F},\mathcal{G})$ on $\tilde{K}\times X$ admits a $k$-exponential separation for $\varepsilon\geq0$ and $\delta\geq0$ small enough. If $\tilde{K}$ is compact in additional, then $(\tilde{F},\mathcal{G})$ on $\tilde{K}\times X$ has the unique $k$-exponential separation for $\varepsilon\geq0$ and $\delta\geq0$ small enough. 

\vskip 5mm
Let $\tilde{F}_{nnh}$ be a map on $B_1(K)$ such that $\tilde{F}_{nnh}\tilde{K}=\tilde{K}$. Let $\tilde{\mathcal{G}}$ be a map from $B_1(K)$ to $L(X)$ defined by $\tilde{\mathcal{G}}(x)=\tilde{G}_x$ for any $x\in B_1(K)$. Denoted by $\tilde{G}^0_x={\rm Id}$ and $\tilde{G}^n_x=\tilde{G}_{\tilde{F}_{nnh}^{n-1}(x)}\circ \tilde{G}^{n-1}_x$ for any $n\in\mathbb{N}^+$ and $x\in B_1(K)$, where $\tilde{F}_{nnh}^0={\rm Id}$ and $\tilde{F}_{nnh}^n=\tilde{F}_{nnh}\circ \tilde{F}_{nnh}^{n-1},\,n\in\mathbb{N}^+$. Let $\mathcal{A}_x$ be the index set such that $\{O_{f,(\tilde{F}_{nnh},\tilde{K})}(x)_{\eta}\}_{\eta\in \mathcal{A}_x}$ consists of all different full-orbit of $(\tilde{F}_{nnh}, \tilde{K})$ with initial point $x\in\tilde{K}$. By $\tilde{F}_{nnh}\tilde{K}=\tilde{K}$, one has $\mathcal{A}_x\neq\emptyset$ for any $x\in \tilde{K}$. Suppose that 

$$\sup\limits_{x\in B_{1}(K)}\{\norm{\tilde{F}_{nnh}(x)-F(x)}\}\leq \delta\quad {\rm and}\quad \tilde{\varepsilon}=\norm{\mathcal{T}-\tilde{\mathcal{G}}}_{C_{KL(X)}}.$$ We say that $\tilde{F}_{nnh}$ possesses (${\bf H}^{\prime}$) property on a nonempty set $\tilde{\tilde{K}}\subset \tilde{K}$ if the following holds:

$\tilde{F}_{nnh}$ is of $C^1$ on $\tilde{\tilde{K}}$ such that the inverse of its derivative $D_y\tilde{F}_{nnh}$, denoted by $(D_y \tilde{F}_{nnh})^{-1}$, exists and is bounded for any $y\in \tilde{\tilde{K}}$.

\vskip 3mm
\noindent{\bf Main Corollary:} For $\tilde{\varepsilon}\geq 0$ and $\delta\geq 0$ small enough, the following hold: 

{\rm (i)} $(\tilde{F}_{nnh},\tilde{\mathcal{G}})$ on $O_{f,(\tilde{F}_{nnh},\tilde{K})}(x)_{\eta}\times X$ has a $k$-exponential separation type property for each $x\in \tilde{K}$ and $\eta\in\mathcal{A}_x$. 

{\rm (ii)} If $\tilde{F}_{nnh}$ possesses (${\bf H}^{\prime}$) property on $O_{f,(\tilde{F}_{nnh},\tilde{K})}(x)_{\eta}$ for the certain $x\in \tilde{K}$ and $\eta\in\mathcal{A}_x$, then $(\tilde{F}_{nnh},\tilde{\mathcal{G}})$ on $O_{f,(\tilde{F}_{nnh},\tilde{K})}(x)_{\eta}\times X$ adimits a $k$-exponential separation.

{\rm (iii)} If $\tilde{F}_{nnh}$ possesses (${\bf H}^{\prime}$) property on a compact minimal set $\tilde{\tilde{K}}$ w.r.t. $(\tilde{F}_{nnh},\tilde{K})$, then $(\tilde{F}_{nnh},\tilde{\mathcal{G}})$ on $\tilde{\tilde{K}}\times X$ has the unique $k$-exponential separation.

\vskip 3mm

\section{Proof of the Main Theorem}
Before this proof, we give some technical lemmas, which are partially motivated from \cite{K-L-S, L-W,Tere} and refferences therein. We prove them without the injective assumption on the operator $T\in L(X)$. Let $C$ and $\tilde{C}$ be complemented and $k$-solid cones.

\begin{lemma}\label{P-alp-0}
{\rm (i)} For any $T\in \mathcal{SP}(C,\tilde{C})$ and $u,v\in{\rm Int} C$ with $u\notin {\rm Span}\{v\}$, one has that $\alpha^{\tilde{C}}_0(Tu, Tv)<\alpha^C_0(u,v)$, and hence, $\alpha^{\tilde{C}}(Tu, Tv)<\alpha^C(u,v)$;

{\rm (ii)} For any $u,v\in{\rm Int} C$ and $r,s\in\mathbb{R}^+\setminus\{0\}$, one has $\alpha^C_0(su,rv)=\frac{s}{r}\alpha^C_0(u,v),$ and hence, $\alpha^C(u,v)=\alpha^C(\frac{u}{\norm{u}},\frac{v}{\norm{v}})$;

{\rm (iii)} $\alpha^C_0(\cdot,\cdot)$ on ${\rm Int}C\times {\rm Int}C$ is upper semicontinuous, and hence, $\alpha^C(\cdot,\cdot)$ on ${\rm Int}C\times {\rm Int}C$ is also upper semicontinuous.
\end{lemma}

\begin{proof}
{\rm (i)} By $u\notin {\rm Span}\{v\}$, ones have that $v-su\in {\rm Int} C$ for any $s\in [0,\frac{1}{\alpha^C_0(u,v)})$, and $v-\frac{1}{\alpha^C_0(u,v)} u\in \partial C\setminus\{0\}$. Then, $Tv-s Tu\in {\rm Int} \tilde{C}$ for any $s\in [0,\frac{1}{\alpha^C_0(u,v)}]$. It implies that $\alpha^{\tilde{C}}_0(Tu, Tv)<\alpha^C_0(u,v)$.

{\rm (ii)} It is dericlty implied by the definition of $\alpha^C_0$.

{\rm (iii)} For any $u,v\in{\rm Int} C$, one has that $v-su\in {\rm Int} C$ for any $s\in [0,\frac{1}{\alpha^C_0(u,v)})$. Clearly, $L_{(u,v,\epsilon)}=\{v-su: s\in [0,\frac{1}{\alpha^C_0(u,v)+\epsilon}]\}$ is compact for any $\epsilon>0$ small enough. Then, there is a $\delta_{(u,v,\epsilon)}>0$ such that $B_{\delta_{(u,v,\epsilon)}}(L_{(u,v,\epsilon)})\subset {\rm Int}\,C$, where $B_{\delta_{(u,v,\epsilon)}}(L_{(u,v,\epsilon)})=\{x\in X:\,\,\inf\limits_{y\in L_{(u,v,\epsilon)}}\norm{x-y}\leq \delta_{(u,v,\epsilon)}\}$. So, for any sequences $\{u_i\}_{i\in \mathbb{N}^+},\,\{v_i\}_{i\in\mathbb{N}^+}$ such that $\lim\limits_{i\rightarrow \infty} u_i=u$ and $\lim\limits_{i\rightarrow \infty} v_i=v$, one has that $L_{(u_i,v_i,\epsilon)}\subset B_{\delta_{(u,v,\epsilon)}}(L_{(u,v,\epsilon)})\subset {\rm Int} C$ for any sufficiently large $i$, where $L(u_i,v_i,\epsilon)=\{v_i-su_i: s\in[0, \frac{1}{\alpha_0^{C}(u,v)+\epsilon}]\}$. Thus, $\limsup\limits_{i\rightarrow \infty}\alpha^C_0(u_i,v_i)\leq \alpha^C_0(u,v)+\epsilon$ for any $\epsilon>0$ small enough. It then follows that $\limsup\limits_{i\rightarrow \infty}\alpha^C_0(u_i,v_i)\leq \alpha^C_0(u,v)$.

Therefore, we have completed the proof.
\end{proof}

\begin{lemma}\label{Jump-alp} Let $u,v\in{\rm Int}\,C$. 

{\rm (i)} If $\alpha^C(u,v)<1$, then $\alpha^C_0(u,v)=\alpha^C_0(v,u)=0$ and hence, $\alpha^C(u,v)=0$;

{\rm (ii)} If $\alpha^C(u,v)=1$, then $sv-u, su-v\in C$ for any $s\in \mathbb{R}^+$; furthermore, 

$$\{sv-u:\,s\in\mathbb{R}^+\setminus\{\alpha_0^C(u,v)\}\} \cup \{su-v:\,s\in\mathbb{R}^+\setminus\{\alpha_0^C(v,u)\}\}  \subset {\rm Int} C.$$
\end{lemma}

\begin{proof} If $\alpha^C_0(v,u)>0$, then $\beta u-v\in {\rm Int}C$ for any $\beta\in (\alpha^C_0(v,u),+\infty)$ implies that $sv-u\in {\rm Int}C$ for any $s\in [0, \frac{1}{\alpha^C_0(v,u)})$. On the other hand, $sv-u\in {\rm Int}C$ for any $s\in (\alpha^C_0(u,v),+\infty)$. 

(i) Suppose that $\alpha^C_0(v,u)>0$. Since $\alpha^C(u,v)<1$, $\alpha^C_0(u,v)<\frac{1}{\alpha^C_0(v,u)}$. It then follows from $\alpha^C_0(u,v)<\frac{1}{\alpha^C_0(v,u)}$ that $sv-u\in {\rm Int}C$ for any $s\in [0,+\infty)$. Thus, $\alpha^C_0(u,v)=0$ and hence $\alpha^C_0(v,u)=0$, a contradiction. By similiar arguments, one has that $\alpha^C_0(u,v)=0$. Thus, we have proved (i).

(ii) Since $\alpha^{C}(u,v)=1$, $\alpha^C(u,v), \alpha^C(v,u)>0$ and more $\alpha^{C}(u,v)=\frac{1}{ \alpha^C(v,u)}$. Together with $C$ being a closed set such that $\mathbb{R}\cdot C\subset C$, one has that (ii) holds.

Therefore, we have completed the proof.
\end{proof}

\begin{lemma}\label{subspace-alp}  Let $P_1, P_2\in G_k({\rm Int} C)$. Then, ones have:

{\rm (i)} $\alpha^C(P_1,P_2)=1\iff P_1=P_2$;

{\rm (ii)} If $\lim\limits_{i\rightarrow +\infty}P_{1i}=P_1$ and $\lim\limits_{i\rightarrow +\infty}P_{2i}=P_2$ with $P_{1i}, P_{2i}\in G_k({\rm Int} C)$ for all $i\in\mathbb{N}^+$ and $P_1,\,P_2\in G_k({\rm Int} C)$, then 

$$\limsup\limits_{i\rightarrow +\infty}\alpha^C(P_{1i},P_{2i})\leq \alpha^C(P_1,P_2);$$

{\rm (iii)} Let $T\in \mathcal{SP}(C,\tilde{C})$. If $P_1\neq P_2$ and $TP_1\neq TP_2$, then $\alpha^{\tilde{C}}(TP_1,TP_2)<\alpha^C(P_1,P_2)$.
\end{lemma}

\begin{proof} (i)  It is clear that $P_1=P_2$ implies for any $u,v\in P_1$, 

\begin{equation}\nonumber\alpha^C(u,v)=\begin{cases}0,\,\,&\,\,{\rm if}\,\,u\notin {\rm Span}\{v\}\setminus\{0\}\\ 
1,\,\,&\,\,{\rm if}\,\,u\in\mathbb{R}^+\cdot\{v\}\setminus\{0\}\\ 
0,\,\,&\,\,{\rm if}\,\,u\in\mathbb{R}^-\cdot\{v\}\setminus\{0\}\end{cases}\end{equation} Thus, $\alpha^C(P_1, P_2)=1$.

We prove by contrary that $\alpha^C(P_1, P_2)=1$ implies $P_1=P_2$. Suppose that there is a $u\in P_2\setminus P_1$. Let $\{v_1,v_2,\cdots,v_k\}$ be a basis of $P_1$. By $\alpha^C(\Sigma_{i=1}^k a_iv_i, u)\leq 1$ for any nonzero $k$-tuple $(a_1,a_2,\cdots,a_k)\in \mathbb{R}^k$, it then follows from Lemma \ref{Jump-alp} that ${\rm Span}\{v_1,v_2,\cdots,v_k,u\}\subset  C$, a contradiction. Thus, $P_2\subset P_1$.  Together with both $P_1$ and $P_2$ being $k$ dimensional, one has $P_1=P_2$.

(ii) $P_1,\,P_2\in G_k({\rm Int} C)$ implies that there is a $\tilde{\delta}>0$ such that $P_1\cap S+B_{\tilde{\delta}}(0),\,P_2\cap S+B_{\tilde{\delta}}(0)\subset {\rm Int} C$, where $B_{\tilde{\delta}}(0)=\{x\in X:\,\norm{x}\leq\delta\}$. $\lim\limits_{i\rightarrow +\infty}P_{1i}=P_1$ and $\lim\limits_{i\rightarrow +\infty}P_{2i}=P_2$ yield that $P_{1i}\cap S\subset P_1\cap S+B_{\frac{\tilde{\delta}}{2}}(0)$ and $P_{2i}\cap S\subset P_2\cap S+B_{\frac{\tilde{\delta}}{2}}(0)$ for any sufficiently large $i\in \mathbb{N}^+$. Note that $\alpha^C(u,v)=\alpha^C(v,u)=\alpha^C(\frac{u}{\norm{u}},\frac{v}{\norm{v}})$ for any $u,v\in{\rm Int} C$. Together with {\rm (i)} and Lemma \ref{Jump-alp},

$$\alpha^C(\tilde{P}_1, \tilde{P}_2)=\max\{1,\,\,\sup\{\alpha^C(u,v):\,u\in \tilde{P}_1\cap S,\,\,v\in \tilde{P}_2\cap S\,\,\text{such that} \,\,u\notin {\rm Span}\{v\}\}\}$$  for any $\tilde{P}_1,\tilde{P}_2\in G_k({\rm Int} C)$. Then, $\max\{1,\frac{4}{\tilde{\delta}^2}\}$ is an upper bound of $\{\alpha^C(P_{1i},P_{2i})\}_{i=N}^{+\infty}$ for some positive integer $N$. So, $\limsup\limits_{i\rightarrow +\infty}\alpha^C(P_{1i}, P_{2i})$ exists.

For any $\epsilon>0$ small enough, take $u^{\epsilon}_i\in P_{1i}\cap S$ and $v^{\epsilon}_i\in P_{2i}\cap S$ such that $\alpha^C(u^{\epsilon}_{i},v^{\epsilon}_i)\geq \alpha^C(P_{1i}, P_{2i})-\epsilon$. Clearly, $\alpha^C(u^{\epsilon}_{i},v^{\epsilon}_i)\leq \max\{1,\frac{4}{\tilde{\delta}^2}\}$ for any $i\geq N$. Recall that $P_1,\,P_2$ are $k$-dimensional. One has that $P_1\cap S$ and $P_2\cap S$ are compact. It then follows that there is a subsquence $\{i_m\}_{m=1}^{+\infty}$ such that 

$$\lim\limits_{m\rightarrow +\infty}u^{\epsilon}_{i_m}=u^{\epsilon}\in P_1\cap S,\,\, \lim\limits_{m\rightarrow +\infty}v^{\epsilon}_{i_m}=v^{\epsilon}\in P_2\cap S,$$ and

$$\begin{aligned}&\lim\sup\limits_{i\rightarrow +\infty}\alpha^C(P_{1i}, P_{2i})-\epsilon\leq \lim\sup\limits_{i\rightarrow +\infty}\alpha^C(u^{\epsilon}_{i},v^{\epsilon}_i)=\lim\limits_{m\rightarrow +\infty}\alpha^C(u^{\epsilon}_{i_m},v^{\epsilon}_{i_m})\\
&\overset{Lemma\, \ref{P-alp-0}(iii)}{\leq} \alpha^C(u^{\epsilon},v^{\epsilon}) \leq \alpha^C(P_1, P_2).\end{aligned}$$ Together with the arbitrariness of $\epsilon$, one has 

$$\lim\sup\limits_{i\rightarrow +\infty}\alpha^C(P_{1i}, P_{2i})\leq \alpha^C(P_1, P_2).$$

(iii) $T\in \mathcal{SP}(C,\tilde{C})$ implies that $T\mid_{\tilde{P}}$ is injective for any $\tilde{P}\in G_k(C)$, and hence, $T P_1,T P_2  \in G_{k}({\rm Int} C)$. By utilizing (i) and Lemma \ref{Jump-alp}, $P_1\neq P_2$ and $T P_1\neq T P_2$ imply that 

$$\alpha^{C}(P_1, P_2), \alpha^{C}(T P_1, T P_2)>1.$$

We assert that there exist $u\in T P_1\cap S$ and $v\in (T P_2\setminus T P_1)\cap S$ such that $\alpha^{\tilde{C}}(u,v)=\alpha^{\tilde{C}}(T P_1,T P_2)$. Clearly, for any $u,v\in{\rm Int}\tilde{C}$, $\alpha^{\tilde{C}}(u,v)=1$ if $u\in\mathbb{R}^+\cdot\{v\}\setminus\{0\}$, and $\alpha^{\tilde{C}}(u,v)=0$ if $u\in\mathbb{R}^-\cdot\{v\}\setminus\{0\}$. By the definition of $\alpha^{\tilde{C}}(\tilde{P}_1, \tilde{P}_2)$ for any $\tilde{P}_1, \tilde{P}_2\in G_k({\rm Int} \tilde{C})$, there exists a sequence $u_n\in T P_1\cap S$ and $v_n\in T P_2\cap S$ such that $u_n\notin {\rm Span}\{v_n\}$ and $\alpha^{\tilde{C}}(u_n,v_n)\in (\alpha^{\tilde{C}}(T P_1,T P_2)-\frac{\varepsilon}{n}, \alpha^{\tilde{C}}(T P_1,T P_2)]$, where $\varepsilon>0$ small enough such that $\alpha^{\tilde{C}}(T P_1,T P_2)-\varepsilon>1$. By the comapctness of $T P_1\cap S$ and $T P_2\cap S$, there are subsequences, also denoted by $u_n, v_n$, such that 

$$\lim\limits_{n\rightarrow +\infty}u_n=u\in T P_1\cap S, \lim\limits_{n\rightarrow +\infty}v_n=v\in T P_2\cap S$$ and

$$1<\lim\limits_{n\rightarrow +\infty}\alpha^{\tilde{C}}(u_n,v_n)=\alpha^{\tilde{C}}(T P_1,T P_2)\overset{Lemma \ref{P-alp-0}(iii)}{\leq}\alpha^{\tilde{C}}(u,v).$$ Clearly, $u\notin {\rm Span}\{v\}$. Together with $\alpha^{\tilde{C}}(u,v)\leq \alpha^{\tilde{C}}(T P_1,T P_2)$, one has that $\alpha^{\tilde{C}}(u,v)=\alpha^{\tilde{C}}(T P_1,T P_2)$. Of course, $u\in T P_1\cap S$ and $v\in (T P_2\setminus T P_1)\cap S$. (Otherwise, $\alpha^{\tilde{C}}(u,v)=0$ for the case $u\notin {\rm Span}\{v\}$ with $u,v\in T P_1$.) Thus, we proved the assertion.

By utilizing the injectivity of $T\mid_{\tilde{P}}$ for any $\tilde{P}\in G_k(C)$, there are unique points $T^{-1}u\in P_1$ and $T^{-1}v\in P_2$ such that $T(T^{-1}u)=u$ and $T(T^{-1}v)=v$. Clearly, $T^{-1}u\notin {\rm Span}\{T^{-1}v\}$. Togther with Lemma \ref{P-alp-0}, one has that $\alpha^{\tilde{C}}(u,v)<\alpha^C(T^{-1}u,T^{-1}v)$. Thus, $\alpha^{\tilde{C}}(T P_1,T P_2)<\alpha^C(T^{-1}u,T^{-1}v)\leq \alpha^C(P_1, P_2)$.

Therefore, we have completed the proof.
\end{proof}

\begin{lemma}\label{SF-angles} Assume that $T\in \mathcal{SF}(C,\tilde{C})$ with the separation index $\kappa \in(0,1]$. 

\noindent{\rm (i)} If $\kappa=1$, then $TC=\tilde{P}\subset {\rm Int} \tilde{C}\cup\{0\}$ is a $k$-dimesional linear subspace;

\noindent{\rm (ii)} If $\kappa\in(0,1)$, then for any $P_1, P_2\in G_k({\rm Int}C)$,
 
\begin{equation}\label{k-angle-constraction}{\rm In}(\alpha^{\tilde{C}}(T P_1, T P_2))\leq\frac{1-\kappa}{1+\kappa}\cdot{\rm In}(\alpha^C(P_1,P_2));\end{equation}

\noindent{\rm (iii)} Let $\tilde{\kappa}\in(0,\kappa)$ and $\tilde{\tilde{C}}=\overline{\mathbb{R}\cdot (TC\cap S+B_{\tilde{\kappa}}(0))}$. Then, $\tilde{\tilde{C}}\subset {\rm Int}\tilde{C}\cup\{0\}$ is a complemented and $k$-solid cone, and $T\in \mathcal{SF}(C,\tilde{\tilde{C}})$ with the separation index $\tilde{\kappa}$.

\noindent{\rm (iv)} For any $P_1, P_2\in G_k({\rm Int}C)$,

\begin{equation}\label{gapmetric-angle-relation}\begin{aligned} &d(TP_1,TP_2)\leq \min\{\frac{ 2(1+\kappa)(1+\tilde{\kappa})}{(\kappa-\tilde{\kappa})^2}\cdot(\alpha^{\tilde{\tilde{C}}}(TP_1,TP_2)-1),\,2\}\\
&\quad\quad\quad\quad\leq \min\{(2+\frac{ 2(1+\kappa)(1+\tilde{\kappa})}{(\kappa-\tilde{\kappa})^2})\cdot{\rm In}(\alpha^{\tilde{\tilde{C}}}(TP_1,TP_2)),\,2\}.\end{aligned}\end{equation}
\end{lemma}

\begin{proof} Firstly, we assert that 

\begin{equation}\label{angle-upper-bound}\sup\limits_{u,v\in C\setminus\{0\} }\{\alpha^{\tilde{C}}(Tu,Tv)\}\leq\frac{1}{\kappa^2}.\end{equation}

 Clearly, $\frac{Tu}{\norm{Tu}}+{\rm Int }B_{\kappa}(0)\subset {\rm Int }\tilde{C}$ and $\frac{Tv}{\norm{Tv}}+{\rm Int}B_{\kappa}\subset {\rm Int }\tilde{C}$ for any $u,v\in C\setminus\{0\}$, where $B_{\kappa}(0)=\{y\in X:\norm{y}\leq \kappa\}$. Thus, $\alpha^{\tilde{C}}_0(\frac{Tu}{\norm{Tu}}, \frac{Tv}{\norm{Tv}})\leq\frac{1}{\kappa}$ and $\alpha^{\tilde{C}}_0(\frac{Tv}{\norm{Tv}}, \frac{Tu}{\norm{Tu}})\leq\frac{1}{\kappa}$. It then follows that $\alpha^{\tilde{C}}(\frac{Tv}{\norm{Tv}}, \frac{Tu}{\norm{Tu}})=\alpha^{\tilde{C}}(Tu,Tv)\leq \frac{1}{\kappa^2}$. Therefore, (\ref{angle-upper-bound}) holds.

\vskip 3mm
{\it The proof of (i).} Since $C$ is $k$-solid, there is a $P\in G_k({\rm Int} C)$. Together with $T\in \mathcal{SF}(C,\tilde{C})$, one has that $T P=\tilde{P}\in G_k({\rm Int}\tilde{C})$. Suppose that the basis of $\tilde{P}$ is $\tilde{e}_1,\tilde{e}_2,\cdots,\tilde{e}_k$. Combining with $\kappa=1$ and (\ref{angle-upper-bound}), $\alpha^{\tilde{C}}(Tu,\sum\limits_{i=0}^{k}l_i\tilde{e}_i)\leq 1$ for any $u\in C\setminus\{0\}$, where $(l_1,l_2,\cdots,l_k)$ is a nontrivial $k$-tuple. Note that $\tilde{C}$ is of rank $k$. It then follows from Lemma \ref{Jump-alp} that $Tu\in \tilde{P}$ for any $u\in C\setminus\{0\}$, and hence $TC=\tilde{P}$. Therefore, we have completed the proof of (i).

\vskip 3mm
{\it The proof of (ii).} When $T P_1=T P_2$, it then follows from Lemma \ref{subspace-alp}(i) that $\alpha^{\tilde{C}}(TP_1,TP_2)=1$. Hence, ${\rm In}(\alpha^{\tilde{C}}(TP_1,TP_2))=0$ and (\ref{k-angle-constraction}) holds. In the following, we only consider the case $P_1,P_2\in G_k({\rm Int} C)$ such that $P_1\neq P_2$ and $TP_1\neq TP_2$.

We assert that ${\rm In}(\alpha^{\tilde{C}}(Tu, Tv))<\frac{1-\kappa}{1+\kappa}\cdot{\rm In}(\alpha^C(u,v))$ for any $u,v\in {\rm Int} C$ such that $\alpha^{C}(u,v)>1$ and $\alpha^{\tilde{C}}(Tu,Tv)>1$. Let $w_1=\frac{1}{\alpha_0^{C}(u,v)}u-v$ and $w_2=v-\alpha_0^{C}(v,u)u$. Obviously, $T w_1, T w_2\neq 0$. (Otherwise, $T\in\mathcal{SF}(C,\tilde{C})$ implies $Tv\in{\rm Span}\{Tu\}\setminus\{0\}$, and hence, $\alpha^{\tilde{C}}(Tu,Tv)\leq1$, a contradiction.) Moreover, $Tw_1, Tw_2\in {\rm Int} \tilde{C}$ because of $T\in\mathcal{SF}(C,\tilde{C})$. 

Let $sTw_1-T w_2=(s+1)\cdot[f(s)Tu-Tv]\,\,{\rm with}\,\,f(s)=\frac{1}{\alpha^{C}_0(u,v)}+\frac{1}{s+1}\cdot[\alpha_0^{C}(v,u)-\frac{1}{\alpha^{C}_0(u,v)}]$ for any $s\in[0,+\infty)$. $\alpha^{C}(u,v)>1$ yields that $\alpha_0^{C}(v,u)-\frac{1}{\alpha^{C}_0(u,v)}>0$. It then follows that $f^{\prime}(s)<0$ for any $s\geq 0$, and $f(0)=\alpha_0^{C}(v,u)$, $f(+\infty)=\frac{1}{\alpha_0^{C}(u,v)}$. Note that $sTu-Tv\in {\rm Int}\tilde{C}$ for any $s\in[0,\frac{1}{\alpha_0^{C}(u,v)}]\cup [\alpha_0^{C}(v,u),+\infty)$. By Lemma \ref{Jump-alp} and $\alpha^{\tilde{C}}(Tu,Tv)>1$, one has that $\alpha^{\tilde{C}}_0(Tv,\,Tu)>0$ and then $\alpha^{\tilde{C}}_0(Tw_2,Tw_1)>0$. By Lemma \ref{Jump-alp}{\rm (i)}, $\alpha^{\tilde{C}}_0(Tw_2,Tw_1)\geq 1$. Together with Lemma \ref{Jump-alp}{\rm (ii)}, $\alpha^{\tilde{C}}(Tw_1,Tw_2)=1$ implies $\alpha^{\tilde{C}}(Tu,Tv)=1$. It then follows from $\alpha^{\tilde{C}}(Tu,Tv)>1$ that $\alpha^{\tilde{C}}(Tw_1,Tw_2)>1$. Furthermore, ones have

$$\begin{aligned}\alpha^{\tilde{C}}_0(Tu,Tv)&\leq\frac{1}{\frac{\alpha_0^{C}(v,u)+\frac{\alpha^{\tilde{C}}_0(Tw_2,Tw_1)}{\alpha_0^{C}(u,v)}}{1+\alpha^{\tilde{C}}_0(Tw_2,Tw_1)}}=\frac{1+\alpha^{\tilde{C}}_0(Tw_2,Tw_1)}{\alpha_0^{C}(v,u)+\frac{\alpha^{\tilde{C}}_0(Tw_2,Tw_1)}{\alpha_0^{C}(u,v)}}\\
&=\frac{\alpha_0^{C}(u,v)+\alpha^{C}_0(u,v)\cdot\alpha^{\tilde{C}}_0(Tw_2,Tw_1)}{\alpha^{C}(u,v)+\alpha^{\tilde{C}}_0(Tw_2,Tw_1)  }.\end{aligned}$$ By the similiar manner, we have that 

$$\begin{aligned}\alpha_0^{\tilde{C}}(Tv,Tu)\leq \frac{\alpha^{\tilde{C}}_0(Tw_1,Tw_2)\alpha^{C}_0(v,u)+\frac{1}{\alpha^{C}_0(u,v)}} {1+ \alpha^{\tilde{C}}_0(Tw_1,Tw_2)}=\frac{\alpha^{\tilde{C}}_0(Tw_1,Tw_2)\alpha^{C}(u,v)+1}{\alpha^{C}_0(u,v)+\alpha^{C}_0(u,v)\cdot\alpha^{\tilde{C}}_0(Tw_1,Tw_2)}.\end{aligned}$$ Then, one has that 

$$\begin{aligned}&\alpha^{\tilde{C}}(Tu,Tv)\leq \frac{(1+\alpha^{\tilde{C}}_0(Tw_2,Tw_1))\cdot(\alpha^{\tilde{C}}_0(Tw_1,Tw_2)\alpha^{C}(u,v)+1)} {(1+\alpha^{\tilde{C}}_0(Tw_1,Tw_2) )\cdot(\alpha^{C}(u,v)+\alpha^{\tilde{C}}_0(Tw_2,Tw_1))}\\
=&\frac{ (\sqrt{\alpha^{\tilde{C}}(Tw_1,Tw_2)\alpha^C(u,v)}+1)^2+\alpha^{\tilde{C}}_0(Tw_2,Tw_1)+ \alpha^{\tilde{C}}_0(Tw_1,Tw_2)\alpha^C(u,v)-2\sqrt{\alpha^{\tilde{C}}(Tw_1,Tw_2)\alpha^C(u,v)}     }{ (\sqrt{\alpha^{\tilde{C}}(Tw_1,Tw_2)}+\sqrt{\alpha^C(u,v)})^2  +\alpha^{\tilde{C}}_0(Tw_2,Tw_1)+ \alpha^{\tilde{C}}_0(Tw_1,Tw_2)\alpha^C(u,v)-2\sqrt{\alpha^{\tilde{C}}(Tw_1,Tw_2)\alpha^C(u,v)}         }  \\
\leq& \frac{ (\sqrt{\alpha^{\tilde{C}}(Tw_1,Tw_2)\alpha^C(u,v)}+1)^2}{(\sqrt{\alpha^{\tilde{C}}(Tw_1,Tw_2)}+\sqrt{\alpha^C(u,v)})^2}\leq\frac{((\sqrt{\alpha^{C}(u,v)})^{\frac{\sqrt{\alpha^{\tilde{C}}(Tw_1,Tw_2)}-1}{1+\sqrt{\alpha^{\tilde{C}}(Tw_1,Tw_2)}}}\cdot(\sqrt{\alpha^{\tilde{C}}(Tw_1,Tw_2)}+\sqrt{\alpha^C(u,v)}))^2}{(\sqrt{\alpha^{\tilde{C}}(Tw_1,Tw_2)}+\sqrt{\alpha^C(u,v)})^2}\\
=&\alpha^{C}(u,v)^{\frac{\sqrt{\alpha^{\tilde{C}}(Tw_1,Tw_2)}-1}{1+\sqrt{\alpha^{\tilde{C}}(Tw_1,Tw_2)}}}.
\end{aligned}$$ Hence, ${\rm In}(\alpha^{\tilde{C}}(Tu,Tv))\leq \frac{\sqrt{\alpha^{\tilde{C}}(Tw_1,Tw_2)}-1}{1+\sqrt{\alpha^{\tilde{C}}(Tw_1,Tw_2)}} {\rm In}( \alpha^{C}(u,v)).$ By virtue of (\ref{angle-upper-bound}), $\alpha^{\tilde{C}}(Tw_1,Tw_2)\leq \frac{1}{\kappa^2}.$ Then, 

\begin{equation}\label{k-angle-constraction-vector}{\rm In}(\alpha^{\tilde{C}}(Tu,Tv))\leq \frac{1-\kappa}{1+\kappa}\cdot {\rm In}( \alpha^{C}(u,v)).\end{equation} Therefore, we have proved the assertion.

Recall that $TP_1\neq TP_2$. It then follows from Lemma \ref{subspace-alp}(i) and $T\in \mathcal{SF}(C,\tilde{C})$ that $\alpha^{\tilde{C}}(TP_1,TP_2)>1$ and there are $u_{\varepsilon}\in P_1, v_{\varepsilon}\in P_2$ such that $\alpha^{\tilde{C}}(Tu_{\varepsilon}, Tv_{\varepsilon})\in (\alpha^{\tilde{C}}(TP_1,TP_2)-\varepsilon, \alpha^{\tilde{C}}(TP_1,TP_2)]$ and $\alpha^{\tilde{C}}(TP_1,TP_2)-\varepsilon>1$ for any positive number $\varepsilon$ small enough. By virtue of the assertion above, one has that ${\rm In}(\alpha^{\tilde{C}}(Tu_{\varepsilon}, Tv_{\varepsilon}))\leq\frac{1-\kappa}{1+\kappa}\cdot {\rm In}(\alpha^{C}(u_{\varepsilon},v_{\varepsilon})) \leq \frac{1-\kappa}{1+\kappa}\cdot {\rm In}(\alpha^{C}(P_1,P_2))$. Let $\varepsilon\rightarrow 0$. By virtue of Lemma \ref{subspace-alp}(ii), one has $ {\rm In}(\alpha^{\tilde{C}}(TP_1,TP_2))\leq \frac{1-\kappa}{1+\kappa}\cdot {\rm In}(\alpha^{C}(P_1,P_2))$.

\vskip 3mm
{\it The proof of (iii).} $T\in\mathcal{SF}(C,\tilde{C})$ implies that $T\mid_{P}$ is injective for any $P\in G_k(C)$. Then, $T P\subset TC\subset \tilde{C}$ is a $k$-dimensional linear subspace. Note that $C$ and $S$ are symmetrical w.r.t the original point of $X$. For any given $w\in \tilde{\tilde{C}}$, there are sequences $\{t_n\}_{n\in \mathbb{N}^+}\subset \mathbb{R}^+,\,\,\{v_{1n}\}_{n\in \mathbb{N}^+}\subset TC\cap S,\,\,\{v_{2n}\}_{n\in \mathbb{N}^+}\subset B_{\tilde{\kappa}}(0)$ such that $\lim\limits_{n\rightarrow +\infty}t_n\cdot v_{1n}+t_n\cdot v_{2n}=w$. Clearly, $\norm{w}<+\infty$ for such a $w\in \tilde{\tilde{C}}$, and $\norm{t_n\cdot v_{1n}+t_n\cdot v_{2n} }\in [(1-\tilde{\kappa})t_n,  (1+\tilde{\kappa})t_n ]$ for any $n\in\mathbb{N}^+$. Together with $\tilde{\kappa}\in(0,\kappa)\subset (0,1)$, one has that $\{t_n\}_{n\in \mathbb{N}^+}$ is bounded. Then, one can take a subsequence $\{n_i\}_{i\in \mathbb{N}^+}\subset \mathbb{N}^+$ such that $\lim\limits_{i\rightarrow +\infty}t_{n_i}=t_{w}$ for some $t_{w}\in \mathbb{R}^+$, and $\lim\limits_{i\rightarrow +\infty}t_{n_i}v_{1n_i}+t_{n_i}v_{2n_i}=w$. It entails that 

\begin{equation}\label{layer-Ctt} w\in t_w\cdot \overline{TC\cap S+B_{\tilde{\kappa}}(0)}
\end{equation} for any given $w\in \tilde{\tilde{C}}$. Note that for any $w\in \overline{TC\cap S+B_{\tilde{\kappa}}(0)}$, one has that 

\begin{equation}\label{w-dist-Ct} \inf\limits_{v\in X\setminus \tilde{C}}\norm{w-v}\geq \inf\limits_{v\in X\setminus \tilde{C}}\{\inf\limits_{u\in TC\cap S}\abs{\norm{w-u}-\norm{u-v}}\}\geq \kappa-\tilde{\kappa}.\end{equation} It then follows from (\ref{layer-Ctt})-(\ref{w-dist-Ct}) that 

\begin{equation}\label{ttC-inttC}\tilde{\tilde{C}}\subset{\rm Int} \tilde{C}\cup\{0\}.\end{equation} Together with $\tilde{C}$ being a complemented and $k$-solid cone, $\tilde{\tilde{C}}$ is also a complemented and $k$-solid cone.

Clearly, 

$$\inf\limits_{u\in TC\cap S,v\in X\setminus\tilde{\tilde{C}}}\norm{u-v}\geq \tilde{\kappa}.$$ On the other hand, it follows from (\ref{layer-Ctt}) and $\tilde{\tilde{C}}$ being complemented that there is a $v\in \overline{TC\cap S+B_{\tilde{\kappa}}(0)}\cap \partial \tilde{\tilde{C}}$, and then, there is a sequence $v_n=v_{1n}+v_{2n}$ such that $v_{1n}\in TC\cap S$, $v_{2n}\in B_{\tilde{\kappa}}(0)$ and $\lim\limits_{n\rightarrow +\infty} v_n=v$ for such a $v$. It implies that 

$$\inf\limits_{u\in TC\cap S}\norm{u-v}\leq \norm{v_{1n}-v}\leq\tilde{\kappa}+\norm{v_{n}-v}.$$ Consequently,

$$\inf\limits_{u\in TC\cap S,v\in X\setminus\tilde{\tilde{C}}}\norm{u-v}\leq \tilde{k}.$$ Thus, $T\in \mathcal{SF}(C,\tilde{\tilde{C}})$ with the separation index $\tilde{\kappa}$. Therefore, we have proved (iii).

\vskip 3mm
{\it The proof of (iv).} If $TP_1=TP_2$, then (\ref{gapmetric-angle-relation}) holds. Clearly, the second inequality in (\ref{gapmetric-angle-relation}) holds.  In the following, we only prove the first inequality in (\ref{gapmetric-angle-relation}) for the case $TP_1\neq TP_2$. Note that for any given $v\in \partial \tilde{\tilde{C}}\cap S$ and $\varepsilon>0$ small enough, there is $v_{tc}\in TC$ and $v_{tc\tilde{\tilde{c}}}\in B_{\norm{v_{tc}}\cdot\tilde{\kappa}}(0)$ such that $\norm{v-(v_{tc}+v_{tc\tilde{\tilde{c}}})}<\varepsilon$. Then, $ \frac{1+\varepsilon}{1-\tilde{\kappa}}\geq\norm{v_{tc}}\geq \frac{1-\varepsilon}{1+\tilde{\kappa}}$. Note that $\inf\limits_{v^{\prime}\in X\setminus \tilde{C}}\norm{v_{tc}-v^{\prime}}\geq\kappa\cdot\norm{v_{tc}}$ and $\norm{v_{tc}-v}\leq\varepsilon+\norm{v_{tc\tilde{\tilde{c}}}}\leq\varepsilon+\frac{\kappa}{2}\cdot\norm{v_{tc}}$. Then, $\inf\limits_{v^{\prime}\in X\setminus \tilde{C}}\norm{v-v^{\prime}}\geq(\kappa-\tilde{\kappa})\cdot\norm{v_{tc}}- \varepsilon\geq \frac{(\kappa-\tilde{\kappa})\cdot(1-\varepsilon)}{1+\tilde{\kappa}}-\varepsilon$ for any given $v\in \partial \tilde{\tilde{C}}\cap S$ and $\varepsilon>0$ small enough. Together with the arbitrariness of $\varepsilon>0$, one has that 

\begin{equation}\label{dist-c..-c.}\inf\limits_{v\in \tilde{\tilde{C}}\cap S ,v^{\prime}\in X\setminus \tilde{C}}\norm{v-v^{\prime}}\geq \frac{\kappa-\tilde{\kappa}}{1+\tilde{\kappa}}\end{equation} By virtue of Lemma \ref{SF-angles}(iii), one has that $T\in\mathcal{SF}(C,\tilde{\tilde{C}})$; and more, $TP_1,\,TP_2\in {\rm Int}\tilde{\tilde{C}}\cup\{0\}\subset {\rm Int}\tilde{C}\cup\{0\}$. Since $\tilde{C}$ is complemented and $k$-solid, there is a $k$-codimensional linear subspace $L\subset (X\setminus \tilde{C})\cup\{0\}$ such that 

$$X=TP_2\oplus L.$$ It then follows from (\ref{dist-c..-c.}) that for any $v\in \tilde{\tilde{C}}$ and $v_l\in L\cap S$,

\begin{equation}\label{C..L-distance}\norm{v-v_l}\geq \max\{\mid1-\norm{v}\mid, \norm{v}\cdot\frac{\kappa-\tilde{\kappa}}{1+\tilde{\kappa}}\}\geq\frac{\kappa-\tilde{\kappa}}{1+\kappa}.\end{equation} Let $u\in TP_1\cap S$.  By utilizing (\ref{dist-c..-c.}) again, one has that $u=u_{tp2}+u_{l}$ such that $u_{tp2}\in TP_2\setminus\{0\}$, $u_{l}\in L$ and 

\begin{equation}\label{utp2-range}\norm{u_{tp2}}\in [\frac{\kappa-\tilde{\kappa}}{1+\tilde{\kappa}},\,\frac{1+\tilde{\kappa}}{\kappa-\tilde{\kappa}}]\end{equation} Note that $u_{tp2}-u=-u_l\notin{\rm Int}\tilde{\tilde{C}}$ implies 

\begin{equation}\label{al-uutp2}\alpha^{\tilde{\tilde{C}}}_0(u,u_{tp2}),\,\,\alpha^{\tilde{\tilde{C}}}_0(u_{tp2},u)\geq1.\end{equation} By virtue of (\ref{C..L-distance}), one has that
 $u_{l}+{\rm Int}B_{\norm{u_{l}}\cdot\frac{\kappa-\tilde{\kappa}}{1+\kappa}}(0)\subset X\setminus \tilde{\tilde{C}}$. Then, one has that
 
\begin{equation}\label{al-utp2u}\alpha^{\tilde{\tilde{C}}}_0(u_{tp2},u)\geq1+ \frac{(\kappa-\tilde{\kappa})\cdot \norm{u_l}}{(1+\kappa)\cdot \norm{u_{tp2}}}.\end{equation} It yields that
 
\begin{equation}\label{point-dist-num}\begin{aligned}& \norm{u-\frac{u_{tp2}}{\norm{ u_{tp2}}}}\leq \norm{ u-u_{tp2}}+\abs{1-\norm{u_{tp2}}}\leq 2\norm{u-u_{tp2}}=2\norm{u_l}\\
&\overset{(\ref{al-utp2u})}{\leq}\frac{2(1+\kappa)}{\kappa-\tilde{\kappa}}\cdot(\alpha^{\tilde{\tilde{C}}}_0(u_{tp2},u)-1)\cdot\norm{u_{tp2}}\overset{(\ref{al-uutp2})}{\leq}\frac{2(1+\kappa)}{\kappa-\tilde{\kappa}}\cdot(\alpha^{\tilde{\tilde{C}}}(u,u_{tp2})-1)\cdot\norm{u_{tp2}}\\
&\overset{(\ref{utp2-range})}{\leq}\frac{ 2(1+\kappa)(1+\tilde{\kappa})}{(\kappa-\tilde{\kappa})^2}\cdot(\alpha^{\tilde{\tilde{C}}}(TP_1,TP_2)-1),
\end{aligned}\end{equation} and hence, 

$$\inf\limits_{v_{tp2}\in TP_2\cap S}\norm{u-v_{tp2}}\leq \frac{ 2(1+\kappa)(1+\tilde{\kappa})}{(\kappa-\tilde{\kappa})^2}\cdot(\alpha^{\tilde{\tilde{C}}}(TP_1,TP_2)-1).$$ By the arbitrariness of $u\in TP_1\cap S$ and $TP_1, TP_2$, one has that (\ref{gapmetric-angle-relation}) holds. Hence, we have proved (iv).

Therefore, we have completed the proof.
\end{proof}

Now, we prove the main theorem.
\begin{proof} Since $(F,\mathcal{T})$ on $K\times X$ admits a $k$-ES, there are $k$-dimensional continuous bundle $K\times (P_x)$ and $k$-codimensional continuous bundle $K\times (Q_x)$ such that (i)-(iii) in Definition \ref{k-ES} hold. (\ref{Separation}) in Definition \ref{k-ES}(iii) states that 

$$\norm{T^n_x w}\leq M\gamma^n\norm{T^n_x v}$$ for any $v\in P_x\cap S$, $w\in Q_x\cap S$, $x\in K$ and $n\in\mathbb{N}^+$, where $M>0$ and $\gamma\in(0,1)$. Note that $P_x$ is $k$-dimensional  and $T_xP_x=P_{F(x)}$ for any $x\in K$. Then, 

$$m(T_x\mid_{P_x})=\min\limits_{v\in P_x\cap S}\{\norm{T_x v}\}$$ exists such that $m(T_x\mid_{P_x})>0$ for any $x\in K$; and more, together with the compactness of $K$ and the continuity of $\mathcal{T}$, $K\times (P_x)$, one has that 

$$\delta_{normp}=\inf\limits_{x\in K}\{m(T_x\mid_{P_x})\}$$ is a well-defined number in $(0,+\infty)$. Note that there is a $k$-dimensional continuous bundle $K\times (L_x)\subset K\times X^*$ such that $Q_x={\rm Ker}(L_x)$ and hence, $Q_x$ is a closed $k$-codimensional linear subspace of $X$ for any $x\in K$. Together with $P_x\cap S$ being compact and $X=P_x\oplus Q_x$, one has that 

$$d^{inf}_x=\inf\limits_{v\in P_x\cap S}\{\inf\limits_{w\in Q_x}\{\norm{v-w}\}\}>0,$$ and then for any $v\in X\cap S$, one has $\norm{\Pi^{P_x}_Q v}\leq \frac{1}{d^{inf}_x}$. Thus, 

$$ \norm{\Pi^{P_x}_Q}_{L(X)}\leq \frac{1}{d^{inf}_x}$$ for any $x\in K$. For any given $\varepsilon_0>0$, one has that 

\begin{equation}\label{L-Q-conti}\begin{aligned}&\quad\quad {\rm if}\,\,\tilde{x},\,x\in K\,\,{\rm such \,\,that}\,\, d^*(L_{\tilde{x}},\,L_{x})\leq\varepsilon_0,\,\,{\rm then}\\
&\sup\limits_{v\in Q_{\tilde{x}}\cap S}\{\inf\limits_{w\in Q_x}\{\norm{v-w}\}\}\leq\varepsilon_0\quad {\rm and} \quad\sup\limits_{v\in Q_{x }\cap S}\{\inf\limits_{w\in Q_{\tilde{x}}}\{\norm{v-w}\}\}\leq\varepsilon_0.
\end{aligned}\end{equation} {\big(}Otherwise, suppose that there is a $v\in Q_{\tilde{x}}\cap S$ such that $\inf\limits_{w\in Q_x}\{\norm{v-w}\}>\varepsilon_0$ without loss of generality. It then follows from Hahn-Banach theorem that there is a $l_x\in L_x\cap S^*$ such that $l_x(v)=\inf\limits_{w\in Q_x}\{\norm{v-w}\}>\varepsilon_0$. As a consequence, $\inf\limits_{l_{\tilde{x}}\in L_{\tilde{x}}}\{\norm{l_x-l_{\tilde{x}}}_*\}>\varepsilon_0$, a contradiction to $d^*(L_{\tilde{x}},\,L_{x})\leq\varepsilon_0$.{\big)} Together with the continuity of $K\times (P_x),\,\,K\times (Q_x)$ and triangle inequality of norms, there is a $\delta^{d^{inf}}_{x,\varepsilon_0}>0$ such that $\abs{d^{inf}_x-d^{inf}_{\tilde{x}}}\leq \varepsilon_0$ if $\norm{x-\tilde{x}}\leq \delta^{d^{inf}}_{x,\varepsilon_0}$ for any $\varepsilon_0>0$ and $x\in K$. Furthermore, combinating with the compactness of $K$, we obtain that

$$\delta_{\Pi^{P_K}_Q}=\sup\limits_{x\in K}\{\norm{\Pi^{P_x}_Q}_{L(X)}\}\in (0, +\infty)$$ Together with (\ref{Pertu-fibre-map}), there is a $\delta^{\prime}_0\in(0,1)$ such that  

$$\delta_{norm}=\sup\limits_{x\in B_{\delta^{\prime}_0}(K)}\{\norm{T_x}_{L(X)},\,\,\norm{G_x}_{L(X)}\}\in (0,+\infty).$$ 

We further assert that $x:\mapsto \Pi^{Q_x}_P$ is continuous on $K$. As a consequence, $x:\mapsto \Pi^{P_x}_Q$ is continuous on $K$. The continuity of $x:\mapsto P_x$ on $K$ and (\ref{L-Q-conti}) implies that for any $x\in K$ and $\varepsilon_0>0$, there is a $\delta^{\Pi}_{x,\varepsilon_0}>0$ such that for any $v\in X\cap S$, there are $v_{P_x}\in P_x, \,\,v_{Q_x}\in Q_x$ such that 

$$\begin{aligned}&\norm{\Pi^{P_{\tilde{x}}}_Q v-v_{P_x}}\leq \varepsilon_0\cdot \norm{\Pi^{P_{\tilde{x}}}_Q v}\leq \varepsilon_0\cdot \delta_{\Pi^{P_K}_Q},\\
&\norm{\Pi^{Q_{\tilde{x}}}_Pv-v_{Q_x}}\leq \varepsilon_0\cdot \norm{\Pi^{Q_{\tilde{x}}}_P v}\leq \varepsilon_0\cdot (1+\delta_{\Pi^{P_K}_Q}),
\end{aligned}$$ and 

$$\begin{aligned}&\norm{\Pi^{Q_{x}}_P\circ \Pi^{P_{\tilde{x}}}_Q v}\leq \varepsilon_0\cdot\delta_{\Pi^{P_K}_Q}(1+\delta_{\Pi^{P_K}_Q}),\\
&\norm{\Pi^{P_{x}}_Q\circ \Pi^{Q_{\tilde{x}}}_P v}\leq \varepsilon_0\cdot\delta_{\Pi^{P_K}_Q}(1+\delta_{\Pi^{P_K}_Q}).\end{aligned}$$ Note that for any $v\in X\cap S$ and $x,\tilde{x}\in K$, one has that

$$\begin{aligned}&\Pi^{Q_x}_P v=\Pi^{Q_{x}}_P\circ \Pi^{P_{\tilde{x}}}_Q v+\Pi^{Q_x}_P\circ \Pi^{Q_{\tilde{x}} }_P v,\\
&\Pi^{Q_{\tilde{x}} }_P v= \Pi^{P_x}_Q\circ\Pi^{Q_{\tilde{x}} }_P v+\Pi^{Q_x}_P\circ \Pi^{Q_{\tilde{x}} }_P v.\end{aligned}$$ It then follows that for any $\varepsilon_0>0$ and $x\in K$, there is a $ \delta^{\Pi}_{x,\varepsilon_0}>0$ such that if $\tilde{x}\in K$ such that $\norm{x-\tilde{x}}\leq \delta^{\Pi}_{x,\varepsilon_0}$, then

$$\norm{\Pi^{Q_x}_P v-\Pi^{Q_{\tilde{x}} }_P v}\leq 2\varepsilon_0\cdot\delta_{\Pi^{P_K}_Q}(1+\delta_{\Pi^{P_K}_Q})$$ for any $v\in X\cap S$ . So, we have proved this assetion.

Recall that $F$ is continuous maps on $B_1(K)$; $F$ is a homeomorphism on $K$; and $K\times (P_x),\,K\times (Q_x)$ are continuous. Furthermore, there are a positive constant in $(0, \delta^{\prime}_0)$, denoted by $\delta_0$, and a nonnegative and increasing function $c(\cdot)$ with $\lim\limits_{t\rightarrow 0^+}c(t)=c(0)=0$ such that 

\begin{equation}\label{FTG-C-bound}\begin{aligned}&\norm{F(x)-F(y)}\leq c(\norm{x-y}),\quad\norm{\tilde{F}(x)-\tilde{F}(y)}\overset{(\ref{Pertu-K-F})}{\leq} c(\norm{x-y})+2\delta\\
&{\rm and}\\
&\norm{T_{x}-T_y}_{L(X)}\leq c(\norm{x-y}),\quad\norm{G_{x}-G_y}_{L(X)}\overset{(\ref{Pertu-fibre-map})}{\leq} c(\norm{x-y})+2\varepsilon\end{aligned}\end{equation} for any $x,y\in B_{\delta_0}(K)$ such that $x\in K$ and $\norm{x-y}\leq\delta_0$; and 

\begin{equation}\label{Pi-c-bound}\norm{\Pi^{P_x}_Q-\Pi^{P_y}_Q}_{L(X)}\leq c(\norm{x-y}),\quad \norm{\Pi^{Q_x}_P-\Pi^{Q_y}_P}_{L(X)}\leq c(\norm{x-y})\end{equation} for any $x,y\in K$ such that $\norm{x-y}\leq\delta_0$.

In the rest of the proof, we take $\delta\in(0,\delta_0)$. Since $\sup\limits_{x\in \tilde{K}}\{\norm{h(x)-x}\}\leq\delta$ and $h$ is a map from $\tilde{K}$ to $K$, one has that $\tilde{K}\subset B_{\delta}(K) \subset B_{\delta_0}(K)$. Recall that $\sup\limits_{x\in B_1(K)}\{\norm{\tilde{F}(x)-F(x)}\}\leq \delta$. By mathematical induction, ones have the following estimates

\begin{equation}\begin{aligned}\label{esti-pertu-maps}&\norm{\tilde{F}^n(x)-F^n(x)}\leq c_n(\delta),\quad \norm{h(\tilde{F}^n(x))-F^n(x)}\leq \hat{c}_n(\delta),\\
&\norm{F^n(x)-F^n(h(x))}\leq \mathring{c}_{n}(\delta),\quad {\rm and \,\,hence,}\\
&\norm{G_{\tilde{F}^n(x)}-T_{F^n(h(x))} }_{L(X)}\leq \varepsilon+c\big(c_n(\delta)+\mathring{c}_{n}(\delta)\big)
\end{aligned}\end{equation} for any $x\in \tilde{K}$ and $\delta>0$ such that 

$$\begin{cases}\hat{c}_i(\delta),\,\,\,c_i(\delta)+\mathring{c}_i(\delta)\leq\delta_0\,\,{\rm for\,\,all}\,i\leq n-1\,{\rm with}\,\,i\in\mathbb{N}^+,&\,{\rm if}\,n\in\mathbb{N}^+\setminus\{1\}\\ {\rm no\,\,other\,\,restriction},\,\,&{\rm if}\,\, n=1\end{cases},$$ where $c_1(\delta)=\delta$, $\mathring{c}_{1}(\delta)=c(\delta)$, $\hat{c}_n(\delta)=\delta+c_n(\delta)$ for any $n\in\mathbb{N}^+$, and $c_n(\delta)=\hat{c}_1(\delta)+c\circ\hat{c}_{n-1}(\delta)$, $\mathring{c}_{n}(\delta)=c\circ \mathring{c}_{n-1}(\delta)$ for any $n\in\mathbb{N}^+\setminus\{1\}$.  Note that 

$$\norm{F^n(h(x))-h(\tilde{F}^n(x))}\leq\norm{F^n(h(x))-F^n(x)}+\norm{F^n(x)-h(\tilde{F}^n(x))}$$ for any $x\in \tilde{K}$ and $n\in\mathbb{N}^+$. By virtue of (\ref{esti-pertu-maps}), one has that

\begin{equation}\label{pertu-F-esti}\norm{F^n(h(x))-h(\tilde{F}^n(x))}\leq \tilde{c}_n(\delta)=\hat{c}_n(\delta)+\mathring{c}_n(\delta)\end{equation} and hence,

\begin{equation}\label{pertu-Pi-esti}\norm{\Pi^{Q_{h(\tilde{F}^n(x))}}_P- \Pi^{Q_{F^n (h(x))}}_P}_{L(X)},\, \norm{\Pi^{P_{h(\tilde{F}^n(x))}}_Q- \Pi^{P_{F^n (h(x))}}_Q}_{L(X)}\leq \tilde{\tilde{c}}_n(\delta)=c\circ\tilde{c}_n(\delta)\end{equation} for any $x\in \tilde{K}$ and $\delta>0$ such that $\begin{cases}\tilde{c}_i(\delta)\leq\delta_0\,\, {\rm for\,\, all}\,\, i\leq n-1\,\, {\rm with}\,\, i\in\mathbb{N}^+,& \,\,{\rm if}\,\,n\in\mathbb{N}^+\setminus\{1\}\\ {\rm no\,\,other\,\, restriction},\,\,&{\rm if}\,\, n=1\end{cases}$. Note that 

$$\begin{aligned}&\norm{G_x-T_{h(x)}}_{L(X)}\leq \norm{G_x-T_x}_{L(X)}+\norm{T_x-T_{h(x)}}_{L(X)}\leq \varepsilon+c(\delta),\\
&\norm{G_x^n-T^n_{h(x)}}_{L(X)}\leq\norm{G_{\tilde{F}^{n-1}(x)}}_{L(X)}\cdot\norm{G_x^{n-1}-T_{h(x)}^{n-1}}_{L(X)}\\
&\quad\quad\quad\quad\quad\quad\quad\quad+\norm{G_{\tilde{F}^{n-1}(x)}-T_{F^{n-1}(h(x))}}_{L(X)}\cdot\norm{T^{n-1}_{h(x)}}_{L(X)}\end{aligned}$$ for any $x\in \tilde{K}$ and $n\in\mathbb{N}^+$. By utilizing mathematical induction again, one has that 

\begin{equation}\label{pertu-T-G-esti}\norm{G_x^n-T^n_{h(x)}}_{L(X)}\leq \tilde{\tilde{\tilde{c}}}_n(\delta)
\end{equation} for any $x\in \tilde{K}$ and $\delta>0$ such that $\begin{cases}\tilde{c}_i(\delta)\leq\delta_0\,\, {\rm for\,\, all}\,\, i\leq n-1\,\, {\rm with}\,\, i\in\mathbb{N}^+,& \,\,{\rm if}\,\,n\in\mathbb{N}^+\setminus\{1\}\\ {\rm no\,\,other\,\, restriction},\,\,&{\rm if}\,\, n=1\end{cases},$ where $\tilde{\tilde{\tilde{c}}}_1(\delta)=\varepsilon+c(\delta)$ and $\tilde{\tilde{\tilde{c}}}_n(\delta)=\delta_{norm}\cdot \tilde{\tilde{\tilde{c}}}_{n-1}(\delta)+\delta_{norm}^{n-1}\cdot[\varepsilon+c\big(c_{n-1}(\delta)+\mathring{c}_{n-1}(\delta) \big)]$ for any $n\in\mathbb{N}^+\setminus\{1\}$. Clearly, $c_n,\,\hat{c}_n\,\mathring{c}_n$ and $\tilde{c}_n$, $\tilde{\tilde{c}}_n$, $\tilde{\tilde{\tilde{c}}}_n$ are nonnegative and increasing functions on $\mathbb{R}^+$ for any $n\in\mathbb{N}^+$ such that 

\begin{equation}\label{pertu-esti-func-bound}\begin{aligned}&\lim\limits_{t\rightarrow 0^+}c_n(t)=\lim\limits_{t\rightarrow 0^+}\hat{c}_n(t)=\lim\limits_{t\rightarrow 0^+}\mathring{c}_n(t)=\lim\limits_{t\rightarrow 0^+}\tilde{c}_n(t)=\lim\limits_{t\rightarrow 0^+}\tilde{\tilde{c}}_{n}(t)=0\\
&{\rm and}\\
&\lim\limits_{t\rightarrow 0^+}\tilde{\tilde{\tilde{c}}}_{n}(t)=n \varepsilon\cdot \delta_{norm}^{n-1}. \end{aligned}\end{equation} 

\vskip 5mm
\noindent{\bf Step 1. The invariant cone fields w.r.t. $(\tilde{F}^{N},\mathcal{G}^{N})$.}
\vskip 5mm
Denoted by 

\begin{equation}\label{def-Cx}C_{x}=\{v\in X:\,\norm{\Pi^{Q_{h(x)}}_P v}\leq \norm{\Pi^{P_{h(x)}}_Qv }\}\end{equation} for any $x\in \tilde{K}$. Let $N^{\prime}$ be the smallest integer in $\mathbb{N}^+\setminus\{1\}$ such that $M\gamma^{N^{\prime}}\in(0,\frac{1}{2})$, where $M,\gamma>0$ are the constants for $(F,\mathcal{T})$ matching the property described in Definition \ref{k-ES}(iii). Define the map $\mathcal{G}^{N}$ from $\mathcal{D}^{\prime}(\mathcal{G}^N)\subset B_1(K)$ to $L(X)$ by 

$$\mathcal{G}^{N}(x)=G^{N}_x,$$ where $\mathcal{D}^{\prime}(\mathcal{G}^N)=\{x\in B_1(K):\, \tilde{F}^i(x)\in B_1(K)\,\,{\rm for \,\,all}\,\,i\in\mathbb{N}^+\,\,{\rm with}\,\,i\leq N-1 \}$ and 

\begin{equation}\label{N-value} N\in \mathbb{N}^+ \quad {\rm with}\quad N\geq N^{\prime}.\end{equation} It then follows from $\tilde{F}$ being a homeomorphism on $\tilde{K}$ that $\tilde{K}\subset \mathcal{D}^{\prime}(\mathcal{G}^N)$ and $(\tilde{F}^{N},\mathcal{G}^{N})$ is a linear cocycle on $\tilde{K}\times X$ for (\ref{N-value}) holds. In the following, we always take $N$ such that (\ref{N-value}) holds.

Take

\begin{equation}\label{ep-etat} \varepsilon\in [0, \frac{(\frac{1}{2}-M\gamma^{2N^{\prime}-1})\cdot\delta^{2N^{\prime}-1}_{normp}}{(2N^{\prime}-1)\cdot\delta_{norm}^{2N^{\prime}-2}\cdot(2+3\delta_{\Pi_Q^{P_K}})})\end{equation} Note that $f(N)=\frac{(\frac{1}{2}-M\gamma^{N})\cdot\delta^N_{normp}}{N\delta_{norm}^{N-1}\cdot(2+3\delta_{\Pi_Q^{P_K}})}$ is a directly decreasing function w.r.t $N$. One has that 

\begin{equation}\label{vare-low-boun}\varepsilon<f(N)\end{equation} for all $N\in \mathbb{N}^+$ such that $N\leq 2N^{\prime}-1$. It then follows from (\ref{Pertu-K-F})-(\ref{Pertu-fibre-map}) and the continuity of $F$ and $\mathcal{T}$ that there is a $\delta^{\prime}_{\varepsilon}\in(0,\delta_0)$ such that for any $\delta\in[0,\delta^{\prime}_{\varepsilon})$, ones have that  $\tilde{c}_i(\delta)\leq\delta_0$ for all $i\in\mathbb{N}^+$ with $i\leq 2N^{\prime}-2$.  Hence, for any $N\in\mathbb{N}^+$ such that $N^{\prime}\leq N\leq 2N^{\prime}-1$, the inequality 

\begin{equation}\label{prori-esti-ine}[1-\tilde{\tilde{c}}_N(\delta)\cdot (1+M\gamma^N)]\cdot\norm{T_{h(x)}^N \Pi^{ P_{h(x)}}_Q v}-\tilde{\tilde{\tilde{c}}}_N(\delta)\cdot\norm{\Pi^{P_{h(\tilde{F}^N(x))}}_Q}_{L(X)}>0\end{equation} implies the following estimates:

\begin{equation}\nonumber\begin{aligned}&\frac{\norm{\Pi^{Q_{h(\tilde{F}^N(x))}}_PG^N_xv}}{\norm{\Pi^{P_{h(\tilde{F}^N(x))}}_QG^N_x v}}\overset{(\ref{pertu-T-G-esti})}{\leq} \frac{ \norm{\Pi^{Q_{h(\tilde{F}^N(x))}}_P T^N_{h(x)}v}+ \tilde{\tilde{\tilde{c}}}_N(\delta)\cdot\norm{\Pi^{Q_{h(\tilde{F}^N(x))}}_P}_{L(X)}}{\norm{\Pi^{P_{h(\tilde{F}^N(x))}}_Q T^N_{h(x)} v} -\tilde{\tilde{\tilde{c}}}_N(\delta)\cdot\norm{\Pi^{P_{h(\tilde{F}^N(x))}}_Q}_{L(X)}}\\
&\leq \frac{ \norm{\Pi^{Q_{h(\tilde{F}^N(x))}}_P\circ \Pi^{ P_{F^N(h(x))}}_Q T^N_{h(x)}v} + \norm{\Pi^{Q_{h(\tilde{F}^N(x))}}_P\circ \Pi^{ Q_{F^N(h(x))}}_P T^N_{h(x)} v}+ \tilde{\tilde{\tilde{c}}}_N(\delta)\cdot\norm{\Pi^{Q_{h(\tilde{F}^N(x))}}_P}_{L(X)}}{\norm{\Pi^{ P_{h(\tilde{F}^N(x))}}_Q \circ \Pi^{ P_{F^N(h(x))}}_Q T^N_{h(x)} v}-\norm{\Pi^{P_{h(\tilde{F}^N(x))}}_Q\circ \Pi^{ Q_{F^N(h(x))}}_P T^N_{h(x)} v}- \tilde{\tilde{\tilde{c}}}_N(\delta)\cdot\norm{\Pi^{P_{h(\tilde{F}^N(x))}}_Q}_{L(X)}}\\
&\overset{(\ref{pertu-Pi-esti})}{\leq} \frac{ \norm{\Pi^{ Q_{F^N(h(x))}}_P T^N_{h(x)} v}+\tilde{\tilde{c}}_N(\delta)\cdot\big[\norm{\Pi^{ Q_{F^N(h(x))}}_P T^N_{h(x)}v}+ \norm{\Pi^{ P_{F^N(h(x))}}_Q T^N_{h(x)}v}\big]+\tilde{\tilde{\tilde{c}}}_N(\delta)\cdot\norm{\Pi^{Q_{h(\tilde{F}^N(x))}}_P}_{L(X)}}{\norm{\Pi^{P_{F^N(h(x))}}_Q T^N_{h(x)} v}-\tilde{\tilde{c}}_N(\delta)\cdot\big[\norm{\Pi^{ Q_{F^N(h(x))}}_P T^N_{h(x)}v}+\norm{\Pi^{ P_{F^N(h(x))}}_Q T^N_{h(x)} v }\big]-\tilde{\tilde{\tilde{c}}}_N(\delta)\cdot\norm{\Pi^{P_{h(\tilde{F}^N(x))}}_Q}_{L(X)}}\\
&\overset{(\ref{Separation})}{\leq} \frac{  [(1+\tilde{\tilde{c}}_N(\delta))\cdot M\gamma^N +  \tilde{\tilde{c}}_N(\delta)]\cdot\norm{\Pi^{ P_{F^N(h(x))}}_Q T_{h(x)}^N v} + \tilde{\tilde{\tilde{c}}}_N(\delta)\cdot\norm{\Pi^{Q_{h(\tilde{F}^N(x))}}_P}_{L(X)}}{[1-\tilde{\tilde{c}}_N(\delta)\cdot (1+M\gamma^N)]\cdot\norm{\Pi^{ P_{F^N(h(x))}}_Q T_{h(x)}^N v} -\tilde{\tilde{\tilde{c}}}_N(\delta)\cdot\norm{\Pi^{P_{h(\tilde{F}^N(x))}}_Q}_{L(X)}}\\
&\overset{{\rm Def.}\,\ref{k-ES}\, (ii)}{=}\frac{  [(1+\tilde{\tilde{c}}_N(\delta))\cdot M\gamma^N +  \tilde{\tilde{c}}_N(\delta)]\cdot\norm{T_{h(x)}^N \Pi^{ P_{h(x)}}_Qv} + \tilde{\tilde{\tilde{c}}}_N(\delta)\cdot\norm{\Pi^{Q_{h(\tilde{F}^N(x))}}_P}_{L(X)}}{[1-\tilde{\tilde{c}}_N(\delta)\cdot (1+M\gamma^N)]\cdot\norm{T_{h(x)}^N \Pi^{ P_{h(x)}}_Q v}-\tilde{\tilde{\tilde{c}}}_N(\delta)\cdot\norm{\Pi^{P_{h(\tilde{F}^N(x))}}_Q}_{L(X)}}\\
&\leq M\gamma^N+\frac{ \tilde{\tilde{c}}_N(\delta)\cdot(1+M\gamma^N)^2\cdot\norm{T_{h(x)}^N \Pi^{ P_{h(x)}}_Q v}+\tilde{\tilde{\tilde{c}}}_N(\delta)\cdot(\norm{\Pi^{Q_{h(\tilde{F}^N(x))}}_P}_{L(X)} +M\gamma^N\norm{\Pi^{P_{h(\tilde{F}^N(x))}}_Q}_{L(X)})}{[1-\tilde{\tilde{c}}_N(\delta)\cdot (1+M\gamma^N)]\cdot\norm{T_{h(x)}^N \Pi^{ P_{h(x)}}_Q v}-\tilde{\tilde{\tilde{c}}}_N(\delta)\cdot\norm{\Pi^{P_{h(\tilde{F}^N(x))}}_Q}_{L(X)}}\end{aligned}\end{equation} 

\begin{equation}\label{G-contraction-delta-0}\begin{aligned}&\overset{(\ref{pertu-esti-func-bound})\,{\rm and}\,\delta\rightarrow 0}{\mathop{\longrightarrow}\limits_{ {\rm uniformly\,for} \,x\in \tilde{K}\,{\rm and}\,v\in C_x\cap S }} M\gamma^N+\frac{N\varepsilon\delta^{N-1}_{norm}\cdot(\norm{\Pi_P^{Q_{h(\tilde{F}^N(x))}}}_{L(X)}+M\gamma^N \norm{\Pi_Q^{  P_{  h(\tilde{F}^N(x)) }}}_{L(X)})}{\norm{T_{h(x)}^N\Pi_Q^{P_{h(x)}}v}-N\varepsilon\delta_{norm}^{N-1}\norm{\Pi_Q^{P_{h(\tilde{F}^N(x))}}}_{L(X)}}\\
&\overset{{\rm Note \,that}\, 2\norm{\Pi_{Q}^{P_{h(x)}}v}\geq \norm{v}=1 }{\leq} M\gamma^N+\frac{N\varepsilon\delta^{N-1}_{norm}\cdot(\norm{\Pi_P^{Q_{h(\tilde{F}^N(x))}}}_{L(X)}+M\gamma^N \norm{\Pi_Q^{  P_{  h(\tilde{F}^N(x)) }}}_{L(X)})}{\frac{\delta_{normp}^N}{2}-N\varepsilon\delta_{norm}^{N-1}\norm{\Pi_Q^{P_{h(\tilde{F}^N(x))}}}_{L(X)}}\\
&\leq M\gamma^N+\frac{N\varepsilon\delta^{N-1}_{norm}\cdot(1+\delta_{\Pi_Q^{P_K}}\cdot(1+M\gamma^N))}{\frac{\delta_{normp}^N}{2}-N\varepsilon\delta_{norm}^{N-1}\cdot\delta_{\Pi_Q^{P_K}}}\overset{(\ref{ep-etat})\,\,{\rm and}\,\,(\ref{vare-low-boun})}{<} \frac{1}{2}.\end{aligned}\end{equation}  for any $v\in C_x\cap S$ and $x\in \tilde{K}$ with $\delta\in [0,\delta^{\prime}_{\varepsilon})$. Note that 

$$\norm{T_{h(x)}^N \Pi^{ P_{h(x)}}_Q v}\geq \frac{\delta^N_{normp}}{2} $$ for any $v\in C_x\cap S$, $x\in \tilde{K}$ and $N\in\mathbb{N}^+$. So, there is a positive contant $\delta^{\prime\prime}_{\varepsilon}\in [0,\delta^{\prime}_{\varepsilon})$ such that (\ref{prori-esti-ine}) holds and more,

\begin{equation}\label{G-contraction-delta}\tilde{\tilde{\tilde{c}}}_N(\delta)\leq \frac{\delta^N_{normp}}{8}\quad{\rm and} \quad\frac{\norm{\Pi^{Q_{h(\tilde{F}^N(x))}}_PG^N_xv}}{\norm{\Pi^{P_{h(\tilde{F}^N(x))}}_QG^N_x v}}<\frac{1}{2}
\end{equation} for any $v\in C_x\cap S$, $x\in \tilde{K}$, $\delta\in [0,\delta^{\prime\prime}_{\varepsilon})$, and $N\in\mathbb{N}^+$ such that $N^{\prime}\leq N\leq 2N^{\prime}-1$. Then, (\ref{pertu-T-G-esti}) implies that 

\begin{equation}\label{G-contraction-delta-coro}\quad\norm{\Pi^{P_{h(\tilde{F}^N(x))}}_QG^N_x v}>\frac{1}{12}\cdot \delta^N_{normp}
\end{equation} for any $v\in C_x\cap S$, $x\in \tilde{K}$, $\delta\in [0,\delta^{\prime\prime}_{\varepsilon})$, and $N\in\mathbb{N}^+$ such that $N^{\prime}\leq N\leq 2N^{\prime}-1$. 

In the rest of the proof, we always assume $\delta\in[0,\delta^{\prime\prime}_{\epsilon})$. Let 

\begin{equation}\label{Def-t-ttCxN}\begin{aligned}&\tilde{C}_{x,N}=\overline{\mathbb{R}\cdot((G_{\tilde{F}^{-N}(x)}^{N}C_{ \tilde{F}^{-N}(x)})\cap S+B_{\frac{1}{6\delta_{\Pi^{P_K}_Q}+3}}(0))}\\
&{\rm and}\\
&\tilde{\tilde{C}}_{x,N}=\overline{\mathbb{R}\cdot((G_{\tilde{F}^{-N}(x)}^{N}C_{ \tilde{F}^{-N}(x)})\cap S+B_{\frac{1}{12\delta_{\Pi^{P_K}_Q}+6}}(0))}
\end{aligned}\end{equation} for any $x\in\tilde{K}$ and $N\in\mathbb{N}^+$ such that $N^{\prime}\leq N\leq 2N^{\prime}-1$. It follows from (\ref{G-contraction-delta}) that 

\begin{equation}\label{first-layer-focu}\begin{aligned}&\tilde{C}_{x,N}\subset\overline{ \mathbb{R}\cdot \bigcup_{v\in (G_{\tilde{F}^{-N}(x)}^{N}C_{ \tilde{F}^{-N}(x)})\cap S } (v+B_{\frac{\norm{\Pi^{P_{h(x)}}_Qv}}{2\cdot(\norm{\Pi^{P_{h(x)}}_Q}_{L(X)}+\norm{\Pi^{Q_{h(x)}}_P}_{L(X)})}}(0)) }\subset C_x,\\
&{\rm and}\\
&G_{\tilde{F}^{-N}(x)}^{N}\in \mathcal{SF}(C_{ \tilde{F}^{-N}(x)},\,C_x) \,\,{\rm with\,\,the \,\,separation\,\,index}\,\,\kappa_{ \tilde{F}^{-N}(x)}(N)\geq\frac{1}{6\delta_{\Pi^{P_K}_Q}+3}\end{aligned}
\end{equation} for any $x\in\tilde{K}$ and $N\in\mathbb{N}^+$ such that $N^{\prime}\leq N\leq 2N^{\prime}-1$. Together with Lemma \ref{SF-angles}(iii), ones have that 

\begin{equation}\label{second-layer-focu}\begin{aligned}&\tilde{\tilde{C}}_{x,N}\subset {\rm Int}C_x\cup \{0\} \,\,{\rm is\,\,a\,\, complemented\,\, and}\,\, k{\rm-solid}\,\, {\rm cone},\\
&{\rm and}\\
& G_{\tilde{F}^{-N}(x)}^{N}\in \mathcal{SF}(C_{ \tilde{F}^{-N}(x)},\tilde{\tilde{C}}_{x,N}) \,\,{\rm with\,\,the \,\,separation\,\,index\,\,}\tilde{\kappa}_{ \tilde{F}^{-N}(x) }(N)=\frac{1}{12\delta_{\Pi^{P_K}_Q}+6}
\end{aligned}\end{equation} for any $x\in\tilde{K}$ and $N\in\mathbb{N}^+$ such that $N^{\prime}\leq N\leq 2N^{\prime}-1$. 

In the rest of proof, we denote that 

\begin{equation}\label{N-pri-denot}\tilde{\tilde{C}}_{x}=\tilde{\tilde{C}}_{x,N^{\prime}}\quad{\rm and}\quad\kappa_x=\kappa_x(N^{\prime})\end{equation} for any $x\in\tilde{K}$. Thus, 

$$\tilde{K}\times (C_x),\,\tilde{K}\times (\tilde{\tilde{C}}_{x}),\,\tilde{K}\times (\tilde{\tilde{C}}_{x,N})$$ are cone fields on $\tilde{K}$ for any $N\in\mathbb{N}^+$ such that $N^{\prime}\leq N\leq 2 N^{\prime}-1$. Combinating (\ref{angle-upper-bound}), (\ref{first-layer-focu})-(\ref{second-layer-focu}) and Lemma \ref{SF-angles}, for any $x\in \tilde{K}$ and $N\in\mathbb{N}^+$ such that $N^{\prime}\leq N\leq 2 N^{\prime}-1$, ones have that: 

{\rm(a.)} $G^N_x\in \mathcal{SF}(\tilde{\tilde{C}}_{x,N},\tilde{\tilde{C}}_{\tilde{F}^N(x),N})$ with the separation index $\tilde{\tilde{\kappa}}_x(N)\geq\frac{1}{12\delta_{\Pi^{P_K}_Q}+6}$;

{\rm(b.)} for any $P_{1x},P_{2x}\in G_k({\rm Int}\tilde{\tilde{C}}_{x,N})$,
$$\begin{aligned}&{\rm In}(\alpha^{\tilde{\tilde{C}}_{\tilde{F}^N(x),N}}(G^N_x P_{1x}, G^N_x P_{2x}))\leq\frac{1-\tilde{\tilde{\kappa}}_x(N)}{1+\tilde{\tilde{\kappa}}_x(N)}\cdot{\rm In}(\alpha^{\tilde{\tilde{C}}_{x,N}}(P_{1x},P_{2x}))\\
&\leq \frac{12\delta_{\Pi_Q^{P_K}}+5}{12\delta_{\Pi_Q^{P_K}}+7}\cdot{\rm In}(\alpha^{\tilde{\tilde{C}}_{x,N}}(P_{1x},P_{2x}));\end{aligned}$$

{\rm(c.)}  for any $P_{1x},P_{2x}\in G_k({\rm Int}C_x)$,

$$\begin{aligned}&d(G^N_x P_{1x},G^N_x P_{2x})\\
&\leq \min\{ (2+\frac{ 2(1+\kappa_x(N))(1+\tilde{\kappa}_x(N))}{(\kappa_x(N)-\tilde{\kappa}_x(N))^2})\cdot{\rm In}(\alpha^{\tilde{\tilde{C}}_{\tilde{F}^N(x),N}}(G^N_xP_{1x},G^N_xP_{2x})),\,2\}\\
&\overset{(\ref{first-layer-focu})\,{\rm and}\,(\ref{second-layer-focu})}{\leq} \min\{ (24(2\delta_{\Pi^{P_K}_Q}+1)(12\delta_{\Pi^{P_K}_Q}+7)+2)\cdot{\rm In}(\alpha^{\tilde{\tilde{C}}_{\tilde{F}^N(x),N}}(G^N_xP_{1x},G^N_xP_{2x})),\,2\};\end{aligned}$$

{\rm(d.)} $\alpha^{\tilde{\tilde{C}}_{\tilde{F}^N(x),N}}(G^N_{x} P_{1x}, G^N_x P_{2x})\leq\frac{1}{(\tilde{\kappa}_x(N))^2}=36(2\delta_{\Pi^{P_K}_Q}+1)^2$ for any $P_{1x}, P_{2x}$
$\in G_k({\rm Int} C_x)$. 
 
\vskip 5mm
\noindent{\bf Step 2. The unique invariant $k$-dimensional bundle w.r.t. $(\tilde{F},\mathcal{G})$ in the cone field.}
\vskip 5mm

Let $\tilde{K}\times (\mathring{P}_x)$ be a $k$-dimensional bundle such that $\mathring{P}_x \subset C_{x}$ for any $x\in \tilde{K}$. It then follows from (\ref{second-layer-focu}) and (a.)-(d.) that for any $x\in \tilde{K}$ and given $N\in\mathbb{N}^+$ such that $N^{\prime}\leq N\leq 2N^{\prime}-1$,

$$\{G^{Nn}_{\tilde{F}^{-Nn}(x)}\mathring{P}_{\tilde{F}^{-Nn}(x)}\}_{n\in\mathbb{N}^+}\subset G_k(\tilde{\tilde{C}}_{x,N})$$ is a cauchy sequence in the completed metric space $G(k,\,X)$; and more, (b.)-(c.) imply that it converges to $\tilde{P}_x\in G_k({\rm Int}\tilde{\tilde{C}}_{x,N})$ uniformly for all $x\in \tilde{K}$. Clearly, 

$$\begin{aligned}&\tilde{P}_{\tilde{F}^{N}(x)}=\lim\limits_{n\rightarrow +\infty} G_x^N\circ G^{N(n-1)}_{\tilde{F}^{-N(n-1)}(x)}\mathring{P}_{\tilde{F}^{-N(n-1)}(x)}\\
&=G_x^N\lim\limits_{n\rightarrow +\infty}G^{N(n-1)}_{\tilde{F}^{-N(n-1)}(x)}\mathring{P}_{\tilde{F}^{-N(n-1)}(x)}\\
&=G_x^N\tilde{P}_{x}\subset{\rm Int}\tilde{\tilde{C}}_{\tilde{F}^{N}(x)}\cup\{0\}\end{aligned}$$ for any $x\in \tilde{K}$ and given $N\in\mathbb{N}^+$ such that $N^{\prime}\leq N\leq 2N^{\prime}-1$. So, $\tilde{K}\times (\tilde{P}_x)$ is a $k$-dimensional bundle such that

\begin{equation}\label{Invar-P}G^{N}_x\tilde{P}_x=\tilde{P}_{\tilde{F}^{N}(x)}\end{equation} for any $x\in \tilde{K}$ and given $N\in\mathbb{N}$ such that $N^{\prime}\leq N\leq 2N^{\prime}-1$. 

Now, we assert that for any given $N\in\mathbb{N}$ such that $N^{\prime}\leq N\leq 2N^{\prime}-1$,

\begin{equation}\label{kd-uni-bundle-invariant-N}\begin{aligned}& \tilde{K}\times (\tilde{P}_x) \,\,{\rm is\,\, the\,\, unique}\,\, k-{\rm dimensional\,\, bundle\,\,such\,\,that}\\
& \tilde{K}\times (\tilde{P}_x)\subset \tilde{K}\times (C_x)\,\,{\rm and}\,\,G^{N}_x\tilde{P}_x=\tilde{P}_{\tilde{F}^{N}(x)}\,\,{\rm for\,\,any}\,\,x\in \tilde{K}.
\end{aligned}\end{equation} Let  $\tilde{K}\times (\mathring{P}^{\prime}_x)$ be a $k$-dimensional bundle such that $\tilde{K}\times (\mathring{P}^{\prime}_x)\subset \tilde{K}\times (C_x)$. By utilizing (\ref{second-layer-focu}) and (a.)-(d.) again, (\ref{Invar-P}) implies that 

$$\begin{aligned}&\lim\limits_{n\rightarrow +\infty}{\rm In}(\alpha^{\tilde{\tilde{C}}_{x,N}}(\tilde{P}_x,\,G^{Nn}_{F^{-Nn}_{\delta}(x)}\mathring{P}^{\prime}_{F^{-Nn}_{\delta}(x)} ))=0\\
&{\rm and}\\
&\lim\limits_{n\rightarrow +\infty} G^{Nn}_{F^{-Nn}_{\delta}(x)}\mathring{P}^{\prime}_{F^{-Nn}_{\delta}(x)}=\tilde{P}_x\end{aligned}$$ for any $x\in K$. It implies that if $\tilde{K}\times (\mathring{P}^{\prime}_x)$ satisfies $G^{N}_x\mathring{P}^{\prime}_x=\mathring{P}^{\prime}_{\tilde{F}^{N}(x)}$ for any $x\in \tilde{K}$, then $\mathring{P}^{\prime}_x=\tilde{P}_x$ for any $x\in \tilde{K}$. Thus, we have proved the assertion.

Note that $G_x^{N^{\prime}\cdot (N^{\prime}+1)}=G_{\tilde{F}^{(N^{\prime})^2}(x)}^{N^{\prime}}\circ\cdots\circ G_x^{N^{\prime}}$ for any $x\in \tilde{K}$. By taking value of $N$  in (\ref{kd-uni-bundle-invariant-N}) as the certain $N^{\prime}$, (\ref{kd-uni-bundle-invariant-N}) and the same arguments in the assertion above imply that 

\begin{equation}\label{kd-uni-bundle-invariant-Np}\begin{aligned}& \tilde{K}\times (\tilde{P}_x) \,\,{\rm is\,\,also\,\, the\,\, unique\,\, k-dimensional\,\, bundle\,\,such\,\,that}\\
& \tilde{K}\times (\tilde{P}_x)\subset \tilde{K}\times (C_x)\,\,{\rm and}\,\,G^{N^{\prime}\cdot(N^{\prime}+1)}_x\tilde{P}_x=\tilde{P}_{\tilde{F}^{N^{\prime}\cdot(N^{\prime}+1)}(x)}\,\,{\rm for\,\,any}\,\,x\in \tilde{K}.\end{aligned}\end{equation} Clearly, 

$$G^{N^{\prime}+1}_x\tilde{P}_x=G_{\tilde{F}^{N^{\prime}}(x)}\circ G^{N^{\prime}}_x\tilde{P}_x=G_{\tilde{F}^{N^{\prime}}(x)}\tilde{P}_{\tilde{F}^{N^{\prime}}(x)}$$ for any $x\in \tilde{K}$. Recall that $\tilde{F}$ is a homoemorphism on $\tilde{K}$ and hence, it is a bijection on $\tilde{K}$. Together with (\ref{kd-uni-bundle-invariant-Np}) and considering the value of $N$ in (\ref{kd-uni-bundle-invariant-N}) being $N^{\prime}+1$, one has that 

\begin{equation}\label{kd-uni-bundle-invariant}\begin{aligned}& \tilde{K}\times (\tilde{P}_x) \,\,{\rm is\,\, the\,\, unique\,\, k-dimensional\,\, bundle\,\,such\,\,that}\\
& \tilde{K}\times (\tilde{P}_x)\subset \tilde{K}\times (C_x)\,\,{\rm and}\,\,G_x\tilde{P}_x=\tilde{P}_{\tilde{F}(x)}\,\,{\rm for\,\,any}\,\,x\in \tilde{K}.\end{aligned}\end{equation} It then follows from (\ref{second-layer-focu}) and (\ref{kd-uni-bundle-invariant}) that there is a $\delta_c\in [\frac{1}{12\delta_{\Pi_{Q}^{P_K}}+7},\frac{1}{12\delta_{\Pi_{Q}^{P_K}}+6})$ such that

\begin{equation}\label{tP-inner}\tilde{P}_x\cap S+B_{\delta_c}(0)\subset {\rm Int} \tilde{\tilde{C}}_x\subset  {\rm Int}C_x\cup \{0\} \end{equation} for any $x\in\tilde{K}$.

\vskip 5mm
\noindent{\bf Step 3. The correction maps for the positive invariant $k$-codimensional bundle.}
\vskip  5mm

Recall that $Q_{h(x)}\setminus\{0\}\subset X\setminus C_x$ is a $k$-codimensional subspace of $X$. (\ref{tP-inner}) implies that $\tilde{P}_x\cap Q_{h(x)}=\{0\}$ and hence, 

\begin{equation}\label{tP-Qh-dsum} X=\tilde{P}_x\oplus Q_{h(x)}\,\, {\rm for\,\, any}\,\, x\in \tilde{K}.\end{equation} Let $\Pi^{\tilde{P}_x}_{Q_h}$ be the natural projection onto $\tilde{P}_x$ along $Q_{h(x)}$, and $\Pi^{Q_{h(x)}}_{\tilde{P}}={\rm Id}-\Pi^{\tilde{P}_x}_{Q_h}$ for any $x\in \tilde{K}$. By virtue of (\ref{first-layer-focu}) and (\ref{kd-uni-bundle-invariant})-(\ref{tP-inner}), one has that 

$$1=\norm{v}=\norm{\Pi^{\tilde{P}_x}_{Q_h}v+ \Pi_{\tilde{P}}^{Q_h(x)}v}\geq\kappa_{\tilde{F}^{-N^{\prime}}(x)}\cdot \norm{\Pi^{\tilde{P}_x}_{Q_h}v}\geq \frac{\norm{\Pi^{\tilde{P}_x}_{Q_h}v}}{6\delta_{\Pi_{Q}^{P_K}}+3}$$ for any $x\in\tilde{K}$ and $v\in S$. As a consequence,

$$\norm{\Pi_{Q_h}^{\tilde{P}_x}}_{L(X)}\leq 6\delta_{\Pi_{Q}^{P_K}}+3$$ for any $x\in \tilde{K}$. Then,

$$\delta_{\Pi_{Q_h}^{\tilde{P}_K}}=\sup\limits_{x\in \tilde{K}}\{\norm{\Pi_{Q_h}^{\tilde{P}_x}}_{L(X)} \} \in (0,\,\, 6\delta_{\Pi_{Q}^{P_K}}+3].$$ It follows from (\ref{kd-uni-bundle-invariant}) and $\tilde{P}_x$ being $k$-dimensional that $G_{x}\mid_{\tilde{P}_x}$ is bijective for any $x\in \tilde{K}$. Together with (\ref{G-contraction-delta})-(\ref{G-contraction-delta-coro}), one has that for any $x\in \tilde{K}$, $v\in\tilde{P}_x\cap S$ and $N\in \mathbb{N}^+$ such that $N^{\prime}\leq N\leq 2N^{\prime}-1$, 

$$\begin{aligned}\norm{G_x^N v}&\geq\norm{\Pi^{P_{h(\tilde{F}^N(x))}}_Q G_x^N v}-\norm{\Pi^{Q_{h(\tilde{F}^N(x))}}_P G_x^N v}\\
&\geq\frac{1}{2}\norm{\Pi^{P_{h(\tilde{F}^N(x))}}_Q G_x^N v}\geq \frac{\delta^N_{normp}}{24}.\end{aligned}$$ It then follows that for any $x\in \tilde{K}$, $v\in\tilde{P}_x\cap S$, 

$$\begin{aligned}&\delta_{norm}^{N^{\prime}}\norm{G_x v}\geq \norm{G_x^{N^{\prime}+1} v}\geq \frac{\delta^{N^{\prime}+1}_{normp}}{24}\\
&{\rm and \,\,hence,}\\
&\norm{G_x v}\geq \frac{\delta^{N^{\prime}+1}_{normp}}{24\delta_{norm}^{N^{\prime}}}.
\end{aligned}$$ So, ones have that

\begin{equation}\label{inverse-G-upb}\begin{aligned}&\delta_{G^{-1}_{M.}}=\sup\limits_{x\in \tilde{K}}\{ \norm{( G^{N^{\prime}}_{x}\mid_{\tilde{P}_x})^{-1} }_{L(X)}\}\in(0, \frac{24}{\delta^{N^{\prime}}_{normp}}]\\
&{\rm and}\\
&\delta_{G^{-1}_{R.}}=\sup\limits_{x\in \tilde{K},r\in\mathbb{N},0\leq r\leq N^{\prime}-1}\{ \norm{( G^{r}_{x}\mid_{\tilde{P}_x})^{-1} }_{L(X)}\}\\   
&\quad\quad\quad\in(0, \max\limits_{r\in\mathbb{N},0\leq r\leq N^{\prime}-1}\{(\frac{24\delta_{norm}^{N^{\prime}}}{\delta^{N^{\prime}+1}_{normp}})^r \}].
\end{aligned}\end{equation} For any $x\in \tilde{K}$, define the following correction map:

\begin{equation}\label{derivation}\begin{aligned} \Psi_x&=\sum_{N=0}^{+\infty} (G^{N+1}_x\mid_{\tilde{P}_x})^{-1}\circ\Pi^{\tilde{P}_{\tilde{F}^{N+1}(x)}}_{Q_h} \circ G_{\tilde{F}^{N}(x)}\circ \Pi^{Q_{h(\tilde{F}^{N}(x))}}_{\tilde{P}}\circ G^{N}_{x}\circ\Pi^{Q_{h(x)}}_{\tilde{P}}.
\end{aligned}
\end{equation} It then follows from (\ref{kd-uni-bundle-invariant}) that  for any $x\in\tilde{K}$, $\Psi_x$ has the following equivalent forms:

\begin{equation}\label{derivation-v}\begin{aligned}&\Psi_x=\sum_{n=0}^{+\infty}\sum_{r=0}^{N^{\prime}-1} (G^{nN^{\prime}+r+1}_x\mid_{\tilde{P}_x})^{-1}\circ\Pi^{\tilde{P}_{\tilde{F}^{nN^{\prime}+r+1}(x)}}_{Q_h} \circ G_{\tilde{F}^{nN^{\prime}+r}(x)}\circ \Pi^{Q_{h(\tilde{F}^{nN^{\prime}+r}(x))}}_{\tilde{P}}\circ\\
&\quad\quad\quad\quad\quad  G^{nN^{\prime}+r}_{x}\circ\Pi^{Q_{h(x)}}_{\tilde{P}}\\
&=\sum_{n=0}^{+\infty}\sum_{r=0}^{N^{\prime}-1}(G_x^r\mid_{\tilde{P}_x})^{-1}\circ(G^{nN^{\prime}}_{\tilde{F}^r(x)}\mid_{\tilde{P}_{\tilde{F}^r(x)}})^{-1}\circ(G_{    \tilde{F}^{nN^{\prime}+r}(x)  }   \mid_{    \tilde{P}_{\tilde{F}^{nN^{\prime}+r}(x)}    })^{-1}\circ\\
&\quad\quad\quad\quad\quad\Pi^{\tilde{P}_{\tilde{F}^{nN^{\prime}+r+1}(x)}}_{Q_h} \circ G_{\tilde{F}^{nN^{\prime}+r}(x)}\circ \Pi^{Q_{h(\tilde{F}^{nN^{\prime}+r}(x))}}_{\tilde{P}}\circ\\
&\quad\quad\quad\quad\quad G^{nN^{\prime}}_{\tilde{F}^r(x)}\circ \Pi_{\tilde{P}}^{Q_{h(\tilde{F}^r(x))}}\circ G_x^r \circ\Pi^{Q_{h(x)}}_{\tilde{P}}.
\end{aligned}\end{equation}

Now, we assert that $\Psi_x$ is well-defined for any $x\in \tilde{K}$. By F. Riesz Lemma on Banach spaces, for any given $x\in \tilde{K}$, there is a basis of $\tilde{P}_x$, denoted by $\{u_{x_{\tilde{p}_1}}, u_{x_{\tilde{p}_2}},\cdots,u_{x_{\tilde{p}_k}}\}$, such that $\norm{u_{x_{\tilde{p}_i}}}=1$ for any $i\in\{1,2,\cdots,k\}$, and $d_{usual}(u_{x_{\tilde{p}_{i+1}}}, {\rm Span}\{u_{x_{\tilde{p}_1}},\cdots,u_{x_{\tilde{p}_i}}\})>\frac{3}{4}$ for any $i\in\{1,2,\cdots,k-1\}$, where $d_{usual}(x,X_0)=\inf\limits_{y\in X_0}\{\norm{x-y}\}$ for any nonempty proper closed subspace $X_0$ of $X$. Denoted by $Y_{-l}={\rm Span}_{ i\in\{1,2,\cdots,k\}\setminus\{l\} }\{u_{x_{\tilde{p}_i}}\}$ for any $l\in\{1,2,\cdots,k\}$. One has that 

$$d_{usual}(u_{x_{\tilde{p}_l}}, Y_{-l})\geq\frac{3}{4}\cdot(\frac{3}{7})^{k-l} \,\,{\rm for}\,\, l\in\{2,\,3,\,\cdots,k\}\,\,{\rm and} \,\,d_{usual}(u_{x_{\tilde{p}_1}}, Y_{-1})\geq(\frac{3}{7})^{k-1}.$$ Hence, 

\begin{equation}\label{dx} d_x=\inf\limits_{l\in \{1,2,\cdots,k\}}d_{usual}(u_{x_{\tilde{p}_l}}, Y_{-l})\geq(\frac{3}{7})^{k-1}\end{equation} for any given $x\in \tilde{K}$. 

Let $\delta_c$ be the constant in (\ref{tP-inner}). Let $\delta_{\tilde{c}}=\frac{\delta_c\cdot d_x}{k}$ and $\tilde{P}^{\prime}_x={\rm Span}\{v_{x_{\tilde{p}_1}},v_{x_{\tilde{p}_2}},\cdots,v_{x_{\tilde{p}_k}}\}$ for any $x\in \tilde{K}$ such that $v_{x_{\tilde{p}_i}}\in u_{x_{\tilde{p}_i}}+B_{\delta_{\tilde{c}}}(0),\,i\in\{1,2,\cdots,k\}$, and it then follows that $\tilde{P}^{\prime}_x$ is a $k$-dimensional subspace such that $\tilde{P}^{\prime}_x\setminus\{0\}\subset{\rm Int}\tilde{\tilde{C}}_x$ and $d_{usual}(u_{\tilde{P}_x},\tilde{P}^{\prime}_x)\leq \delta_c$ for any $\tilde{u}_{\tilde{P}_x}\in \tilde{P}_x\cap S$. For such a $k$-dimensional subspace $\tilde{P}^{\prime}_x$ with $x\in \tilde{K}$, we call that

$$\tilde{P}^{\prime}_x {\rm \,\,with\,\, the\,\, property}\,\, (\tilde{P}^{\prime}).$$ Let $\delta_{\tilde{\tilde{c}}_x}\in(0,\frac{\delta_{c}\cdot d_x }{k^2})\cap(0, \frac{\delta_c\cdot d_x}{6k})$ for any $x\in \tilde{K}$. For any given $u_{\tilde{P}_x}=\sum_{i=1}^{k}a_i\cdot u_{x_{\tilde{p}_i}}\in \tilde{P}_x\cap S$, $u_Q\in Q_{h(x)}\cap S$ and $x\in \tilde{K}$, one has that the $k$-dimensional space

$$\tilde{P}^{\prime\prime}_x={\rm Span}\{u_{x_{\tilde{p}_1}},\cdots,u_{x_{\tilde{p}_l}}+\frac{\delta_{\tilde{\tilde{c}}_x}}{a_l}u_Q,\cdots,u_{x_{\tilde{p}_k}}\}$$ with the property $(\tilde{P}^{\prime})$ such that 

$$u_{\tilde{P}_x}+\delta_{\tilde{\tilde{c}}_x}u_Q\in \tilde{P}^{\prime\prime}_x,$$ where $l\in\{1,2,\cdots,k\}$ is the chosen index such that $\mid a_l\mid\geq\frac{1}{k}$ (Since $\norm{u_{\tilde{P}_x}}=1$, such a $l$ must exist.). It is clear that $\abs{a_i}\cdot d_x\leq\norm{u_{\tilde{P}_x}}=1$ and hence, $\abs{a_i}\leq \frac{1}{d_x}$ for any $i\in\{1,2,\cdots,k\}$. Let 

$$\tilde{\tilde{P}}^{\prime\prime}_x=\{\frac{u_{x_{\tilde{p}_l}}+\frac{\delta_{\tilde{\tilde{c}}_x}}{a_l}u_Q+ \sum_{i\in\{1,2,\cdots,k\}\setminus\{l\}}  b_i u_{x_{\tilde{p}_i}}}{\norm{u_{x_{\tilde{p}_l}}+\frac{\delta_{\tilde{\tilde{c}}_x}}{a_l}u_Q+ \sum_{i\in\{1,2,\cdots,k\}\setminus\{l\}}  b_i u_{x_{\tilde{p}_i}}}}: b_i\in\mathbb{R}\,\, {\rm for\,\, any}\,\, i\in \{1,2,\cdots,k\}\setminus\{l\}\}.$$ Then, 

$$\alpha^{\tilde{\tilde{C}}_x}(\tilde{P}_x,\tilde{P}^{\prime\prime}_x)=\sup\limits_{u_{x_{\tilde{p}}}\in\tilde{P}_x\cap S,\,\,u^{\prime\prime}_{x_{\tilde{\tilde{p}}}}\in\tilde{\tilde{P}}^{\prime\prime}_x\cup(-\tilde{\tilde{P}}^{\prime\prime}_x)} \alpha^{\tilde{\tilde{C}}_x}(u_{x_{\tilde{p}}}, u^{\prime\prime}_{x_{\tilde{\tilde{p}}}}).$$ Denote $u_{x\tilde{\tilde{p}}_1}=\frac{u_{x_{\tilde{p}_l}}+ \sum_{i\in\{1,2,\cdots,k\}\setminus\{l\}}  b_i u_{x_{\tilde{p}_i}}}{\norm{u_{x_{\tilde{p}_l}}+\frac{\delta_{\tilde{\tilde{c}}_x}}{a_l}u_Q+ \sum_{i\in\{1,2,\cdots,k\}\setminus\{l\}}  b_i u_{x_{\tilde{p}_i}}}}$ and $u_{x\tilde{\tilde{p}}_2}=  \frac{  \frac{\delta_{\tilde{\tilde{c}}_x}}{a_l}u_Q }{\norm{u_{x_{\tilde{p}_l}}+\frac{\delta_{\tilde{\tilde{c}}_x}}{a_l}u_Q+ \sum_{i\in\{1,2,\cdots,k\}\setminus\{l\}}  b_i u_{x_{\tilde{p}_i}}}}$. Then, 

$$u^{\prime\prime}_{x_{\tilde{\tilde{p}}}}=u_{x\tilde{\tilde{p}}_1}+u_{x\tilde{\tilde{p}}_2}\,\,{\rm or}\,\,-u_{x\tilde{\tilde{p}}_1}-u_{x\tilde{\tilde{p}}_2}.$$ One has that $\norm{ u_{x_{\tilde{p}_l}}+\sum_{i\in\{1,2,\cdots,k\}\setminus\{l\}}  b_i u_{x_{\tilde{p}_i}} }\geq d_{usual}(u_{x_{\tilde{p}_l}}, Y_{-l})\geq d_x$ for any $b_i\in\mathbb{R},\, i\in\{1,2,\cdots,k\}\setminus\{l\}$. It then follows from $\frac{\delta_{\tilde{\tilde{c}}_x} }{\abs{a_l}}<\frac{\delta_c\cdot d_x}{6k}\cdot k=\frac{\delta_c\cdot d_x}{6}<\frac{d_x}{2}$ that 

$$\norm{u_{x\tilde{\tilde{p}}_1}}\geq \frac{2}{3}.$$ Furthermore, together with $\delta_{\tilde{\tilde{c}}_x}\in(0,\frac{\delta_c d_x}{6k})$ and $\abs{a_l}\geq\frac{1}{k}$, ones have that

$$\begin{aligned}&\alpha^{\tilde{\tilde{C}}_x}_0(u_{x_{\tilde{p}}}, u^{\prime\prime}_{x_{\tilde{\tilde{p}}}})\leq \frac{\delta_c}{\norm{u_{x\tilde{\tilde{p}}_1}}\cdot \delta_c-\frac{2\delta_{\tilde{\tilde{c}}_x}}{\abs{a_l}\cdot d_x}}\leq \frac{\delta_c}{\frac{2\delta_c}{3}-\frac{2\delta_{\tilde{\tilde{c}}_x}}{\abs{a_l}\cdot d_x}}\\
&{\rm and}\\
&\alpha^{\tilde{\tilde{C}}_x}_0(u^{\prime\prime}_{x_{\tilde{\tilde{p}}}}, u_{x_{\tilde{p}}})\leq \frac{2\delta_{\tilde{\tilde{c}}_x}}{\delta_c\abs{a_l} d_x}+\norm{u_{x\tilde{\tilde{p}}_1}} \leq\frac{2\delta_{\tilde{\tilde{c}}_x}}{\delta_c\abs{a_l} d_x}+2.\end{aligned}$$ It yields that 

$$ \alpha^{\tilde{\tilde{C}}_x}(u^{\prime\prime}_{x_{\tilde{\tilde{p}}}}, u_{x_{\tilde{p}}} )-1\leq\frac{4\delta_{\tilde{\tilde{c}}_x}}{ \norm{u_{x\tilde{\tilde{p}}_1}}\cdot\delta_c\cdot\abs{a_l}d_x-2\delta_{\tilde{\tilde{c}}_x}}\leq \frac{6\delta_{\tilde{\tilde{c}}_x}}{ \delta_c\cdot\abs{a_l}d_x-3\delta_{\tilde{\tilde{c}}_x}}\leq\frac{12\delta_{\tilde{\tilde{c}}_x}}{\delta_c\cdot\abs{a_l}d_x}\leq\frac{12k\delta_{\tilde{\tilde{c}}_x}}{\delta_cd_x} $$ for any $x\in \tilde{K}$ and $u_{x_{\tilde{p}}}\in\tilde{P}_x\cap S,\,\,u^{\prime\prime}_{x_{\tilde{\tilde{p}}}}\in\tilde{\tilde{P}}^{\prime\prime}_x\cup(-\tilde{\tilde{P}}^{\prime\prime}_x)$. Hence,

\begin{equation}\label{PPprime-up-esti}{\rm In}(\alpha^{\tilde{\tilde{C}}_x}(\tilde{P}_x,\tilde{P}^{\prime\prime}_x)) \leq \alpha^{\tilde{\tilde{C}}_x}(\tilde{P}_x,\tilde{P}^{\prime\prime}_x)-1\leq\frac{12k}{\delta_cd_x}\cdot\delta_{\tilde{\tilde{c}}_x}\end{equation} for any $x\in \tilde{K}$. Let 

$$\tilde{M}_0=24(2\delta_{\Pi_Q^{P_K}}+1)(12\delta_{\Pi_Q^{P_K}}+7)+2.$$ It then follows that

$$\begin{aligned}&\norm{\frac{G^{N^{\prime}n}_{x}(u_{\tilde{P}_x}+\delta_{\tilde{\tilde{c}}_x} u_Q)}{\norm{G^{N^{\prime}n}_{x}(u_{\tilde{P}_x}+\delta_{\tilde{\tilde{c}}_x} u_Q)}}-\frac{G^{N^{\prime}n}_{x}u_{\tilde{P}_x}+\delta_{\tilde{\tilde{c}}_x  }\Pi_{Q_h}^{\tilde{P}_{ \tilde{F}^{N^{\prime}n}(x) }}G_x^{N^{\prime}n}u_{Q}}{\norm{ G^{N^{\prime}n}_{x}u_{\tilde{P}_x}+\delta_{\tilde{\tilde{c}}_x}\Pi_{Q_h}^{\tilde{P}_{\tilde{F}^{N^{\prime}n}(x) }}G_x^{N^{\prime}n}u_{Q}}}}\\
\overset{(\ref{point-dist-num})+(\ref{second-layer-focu})}{\leq}&\min\{\frac{2 (1+\kappa_{\tilde{F}^{N^{\prime}(n-1)}(x) })(1+\frac{1}{12\delta_{\Pi_Q^{P_K}}+6})}{(\kappa_{\tilde{F}^{N^{\prime}(n-1)}(x)}-\frac{1}{12\delta_{\Pi_Q^{P_K}}+6})^2}\cdot(\alpha^{ \tilde{\tilde{C}}_{\tilde{F}^{N^{\prime}n}(x)} }(G^{N^{\prime}n}_x\tilde{P}_{x}, G^{N^{\prime}n}_x\tilde{P}^{\prime\prime}_{x}) -1),2\}\\
\overset{(\ref{first-layer-focu})}{\leq}&\min\{\tilde{M}_0\cdot{\rm In}(\alpha^{ \tilde{\tilde{C}}_{\tilde{F}^{N^{\prime}n}(x)} }(G^{N^{\prime}n}_x\tilde{P}_{x}, G^{N^{\prime}n}_x\tilde{P}^{\prime\prime}_{x})), 2\}\\
\overset{(b.)}{\leq}&\min\{\tilde{M}_0\cdot (\prod_{i=0}^{n-1}\frac{1-\tilde{\tilde{\kappa}}_{\tilde{F}^{N^{\prime}i}(x)}}{1+\tilde{\tilde{\kappa}}_{\tilde{F}^{N^{\prime}i}(x)}})\cdot {\rm In}(\alpha^{\tilde{\tilde{C}}_x}(\tilde{P}_x,\tilde{P}^{\prime\prime}_x)),\,2\}\\
\leq &\min\{\tilde{M}_0\cdot(\frac{12\delta_{\Pi^{P_K}_Q}+5}{12\delta_{\Pi^{P_K}_Q}+7})^n\cdot {\rm In}(\alpha^{\tilde{\tilde{C}}_x}(\tilde{P}_x,\tilde{P}^{\prime\prime}_x)),\,2\}\\
\overset{(\ref{PPprime-up-esti})}{\leq} &\min\{\tilde{M}_0 \delta_{\tilde{\tilde{c}}_x}\cdot \frac{12k}{\delta_cd_x}\cdot(\frac{12\delta_{\Pi^{P_K}_Q}+5}{12\delta_{\Pi^{P_K}_Q}+7})^n,\,2\}\\
\overset{(\ref{dx})}{\leq} &\min\{\frac{12k\tilde{M}_0 \delta_{\tilde{\tilde{c}}_x}}{\delta_c}\cdot (\frac{7}{3})^{k-1}\cdot(\frac{12\delta_{\Pi^{P_K}_Q}+5}{12\delta_{\Pi^{P_K}_Q}+7})^n,\,2\}
\end{aligned}$$ for any $n\in \mathbb{N}^+$ and $x\in \tilde{K}$. Note that $\delta_{\tilde{\tilde{c}}_x}\in (0,\frac{\delta_c\cdot d_x}{k}\cdot\min\{\frac{1}{6},\,\frac{1}{k}\})$ and $d_x\in[(\frac{3}{7})^{k-1},1]$ for any $x\in \tilde{K}$. Thus, there is a positive constant, denoted by $\tilde{M}>0$, such that 

\begin{equation}\label{GPQnbound}\norm{\frac{G^{N^{\prime}n}_{x}u_{\tilde{P}_x}+\delta_{\tilde{\tilde{c}}_x}\Pi_{Q_h}^{\tilde{P}_{\tilde{F}^{N^{\prime}n}(x)}}G_x^{N^{\prime}n}u_{Q}}{\norm{ G^{N^{\prime}n}_{x}u_{\tilde{P}_x}+\delta_{\tilde{\tilde{c}}_x}\Pi_{Q_h}^{\tilde{P}_{\tilde{F}^{N^{\prime}n}(x) }}G_x^{N^{\prime}n}u_{Q}}}-\frac{G^{N^{\prime}n}_{x}(u_{\tilde{P}_x}+\delta_{\tilde{\tilde{c}}_x} u_Q)}{\norm{G^{N^{\prime}n}_{x}(u_{\tilde{P}_x}+\delta_{\tilde{\tilde{c}}_x} u_Q)}}}\leq \frac{12k\tilde{M} \delta_{\tilde{\tilde{c}}_x}}{\delta_c}\cdot (\frac{7}{3})^{k-1}\cdot(\frac{12\delta_{\Pi^{P_K}_Q}+5}{12\delta_{\Pi^{P_K}_Q}+7})^n
\end{equation} for any $n\in\mathbb{N}^+$ and $x\in \tilde{K}$. Note that

$$\begin{aligned}&\delta_{\tilde{\tilde{c}}_x}\cdot \norm{ \frac{\Pi^{Q_{h(\tilde{F}^{N^{\prime}n}(x))}}_{\tilde{P}} G^{N^{\prime}n}_{x} u_Q}{\norm{G^{N^{\prime}n}_{x}(u_{\tilde{P}_x}+\delta_{\tilde{\tilde{c}}_x} u_Q)}}}=\norm{\Pi^{Q_{h(\tilde{F}^{N^{\prime}n}(x))}}_{\tilde{P}}(\frac{G^{N^{\prime}n}_{x}u_{\tilde{P}_x}+\delta_{\tilde{\tilde{c}}_x}\Pi_{Q_h}^{\tilde{P}_{\tilde{F}^{N^{\prime}n}(x)}}G_x^{N^{\prime}n}u_{Q}}{\norm{ G^{N^{\prime}n}_{x}u_{\tilde{P}_x}+\delta_{\tilde{\tilde{c}}_x}\Pi_{Q_h}^{\tilde{P}_{\tilde{F}^{N^{\prime}n}(x) }}G_x^{N^{\prime}n}u_{Q}}}-\frac{G^{N^{\prime}n}_{x}(u_{\tilde{P}_x}+\delta_{\tilde{\tilde{c}}_x} u_Q)}{\norm{G^{N^{\prime}n}_{x}(u_{\tilde{P}_x}+\delta_{\tilde{\tilde{c}}_x} u_Q)}})}\\
\leq&(1+\delta_{\Pi^{\tilde{P}_K}_{Q_h}})\cdot\norm{\frac{G^{N^{\prime}n}_{x}u_{\tilde{P}_x}+\delta_{\tilde{\tilde{c}}_x}\Pi_{Q_h}^{\tilde{P}_{\tilde{F}^{N^{\prime}n}(x)}}G_x^{N^{\prime}n}u_{Q}}{\norm{ G^{N^{\prime}n}_{x}u_{\tilde{P}_x}+\delta_{\tilde{\tilde{c}}_x}\Pi_{Q_h}^{\tilde{P}_{\tilde{F}^{N^{\prime}n}(x) }}G_x^{N^{\prime}n}u_{Q}}}-\frac{G^{N^{\prime}n}_{x}(u_{\tilde{P}_x}+\delta_{\tilde{\tilde{c}}_x} u_Q)}{\norm{G^{N^{\prime}n}_{x}(u_{\tilde{P}_x}+\delta_{\tilde{\tilde{c}}_x} u_Q)}}}\end{aligned}$$ for any $n\in\mathbb{N}^+$ and $x\in \tilde{K}$. Thus,

$$\norm{\Pi^{Q_{h(\tilde{F}^{N^{\prime}n}(x))}}_{\tilde{P}} G^{N^{\prime}n}_{x} u_Q}\leq  \frac{12k\tilde{M}\cdot (1+\delta_{\Pi^{\tilde{P}_K}_{Q_h}})}{\delta_c}\cdot (\frac{7}{3})^{k-1}\cdot(\frac{12\delta_{\Pi^{P_K}_Q}+5}{12\delta_{\Pi^{P_K}_Q}+7})^n\cdot \norm{G^{N^{\prime}n}_{x}(u_{\tilde{P}_x}+\delta_{\tilde{\tilde{c}}_x} u_Q)}$$ for any $n\in\mathbb{N}^+$ and $x\in \tilde{K}$ with $\delta_{\tilde{\tilde{c}}_x}\in(0, \frac{\delta_c d_x}{k}\cdot\min\{\frac{1}{6},\,\frac{1}{k}\})$. Let 

$$\delta_{\tilde{\tilde{c}}_x}\rightarrow 0\quad {\rm and} \quad\tilde{M}^{\prime}=(\frac{7}{3})^{k-1}\cdot\frac{12k\tilde{M}\cdot (1+\delta_{\Pi^{\tilde{P}_K}_{Q_h}})}{\delta_c}.$$ Then, 

$$\norm{\Pi^{Q_{h(\tilde{F}^{N^{\prime}n}(x))}}_{\tilde{P}} G^{N^{\prime}n}_{x} u_Q}\leq \tilde{M}^{\prime}\cdot(\frac{12\delta_{\Pi^{P_K}_Q}+5}{12\delta_{\Pi^{P_K}_Q}+7})^n \cdot\norm{G^{N^{\prime}n}_{x}u_{\tilde{P}_x}}$$ for any $n\in\mathbb{N}^+$ and $x\in \tilde{K}$. Together with the arbitrariness of $u_{\tilde{P}_x}\in \tilde{P}_x\cap S,\,u_Q\in Q_{h(x)}\cap S$, one has that 

\begin{equation}\label{Separation-G-tilde}\norm{\Pi^{Q_{h(\tilde{F}^{N^{\prime}n}(x))}}_{\tilde{P}} G^{N^{\prime}n}_{x}\mid_{Q_{h(x)}}}_{L(X)}\cdot\norm{  (G^{N^{\prime}n}_{x}\mid_{\tilde{P}_x})^{-1} }_{L(X)}\leq  \tilde{M}^{\prime} \cdot(\frac{12\delta_{\Pi^{P_K}_Q}+5}{12\delta_{\Pi^{P_K}_Q}+7})^n\end{equation} for any $n\in \mathbb{N}^+$ and $x\in \tilde{K}$. Let $\tilde{\tilde{M}}=\max\{\tilde{M}^{\prime},1\}$. Combinating with (\ref{derivation})-(\ref{derivation-v}), one has that

\begin{equation}\label{Covergence-norm-Psi}\begin{aligned}& \norm{\Psi_x}_{L(X)}\leq \sum\limits_{n=0}^{+\infty}\big(\tilde{\tilde{M}}\cdot\delta_{norm}\cdot(\delta_{G_{R.}^{-1}})^2\cdot(1+\delta_{\Pi_{Q_h}}^{\tilde{P}_K})^2\cdot\delta_{\Pi_{Q_h}}^{\tilde{P}_K}\cdot(\frac{12\delta_{\Pi^{P_K}_{Q}}+5}{12\delta_{\Pi^{P_K}_{Q}}+7})^n\cdot(\sum\limits_{r=0}^{N^{\prime}-1}(\delta_{norm})^r)\big)\\
&=\sum\limits_{n=0}^{+\infty}\tilde{\tilde{M}}\cdot\tilde{\tilde{M}}^{\prime} \cdot(\delta_{G_{R.}^{-1}})^2\cdot(1+\delta_{\Pi_{Q_h}}^{\tilde{P}_K})^2\cdot\delta_{\Pi_{Q_h}}^{\tilde{P}_K}\cdot(\frac{12\delta_{\Pi^{P_K}_{Q}}+5}{12\delta_{\Pi^{P_K}_{Q}}+7})^n\\
&=M_4\end{aligned}
\end{equation} for any $x\in \tilde{K}$, where $\tilde{\tilde{M}}^{\prime}=\begin{cases}\frac{\delta_{norm}(1-(\delta_{norm})^{N^{\prime}})} {1-\delta_{norm}},& {\rm if} \,\,\delta_{norm}\neq 1\\ N^{\prime}, & {\rm if}\,\,\delta_{norm}=1\end{cases}$, $\tilde{\tilde{\tilde{M}}}=\frac{\tilde{\tilde{M}}\cdot\tilde{\tilde{M}}^{\prime}}{2}\cdot(\delta_{G_{R.}^{-1}})^2\cdot(1+\delta_{\Pi_{Q_h}}^{\tilde{P}_K})^2\cdot\delta_{\Pi_{Q_h}}^{\tilde{P}_K}$ and $M_4=\tilde{\tilde{\tilde{M}}}\cdot(12\delta_{\Pi^{P_K}_{Q}}+7)$. 

\vskip 3mm
Thus, $\Psi_x$ is well-defined for any $x\in \tilde{K}$ and we proved this assertion. 

\vskip 5mm
\noindent{\bf Step 4. The certain positive invariant $k$-codimensional bundle w.r.t. $(\tilde{F},\mathcal{G})$.}
\vskip 5mm

We further assert that $G_{x}\Psi_x(u)=\Pi^{\tilde{P}_{\tilde{F}(x)}}_{Q_h}\circ G_x u+\Psi_{\tilde{F}(x)}(G_x u)$ for any $u\in Q_{h(x)}$ and $x\in \tilde{K}$. By (\ref{Covergence-norm-Psi}) and the boundness of $G_x,\,x\in \tilde{K}$, one has that 

$$G_{x}\Psi_x=\sum\limits_{N=0}^{+\infty}G_{x} \circ(G^{N+1}_x\mid_{\tilde{P}_x})^{-1}\circ\Pi^{\tilde{P}_{\tilde{F}^{N+1}(x)}}_{Q_h} \circ G_{\tilde{F}^{N}(x)}\circ \Pi^{Q_{h(\tilde{F}^{N}(x))}}_{\tilde{P}}\circ G^{N}_{x}\circ\Pi^{Q_{h(x)}}_{\tilde{P}}$$ for any $x\in \tilde{K}$. It then follows from (\ref{kd-uni-bundle-invariant}) that for any $u\in Q_{h(x)}$ and $x\in \tilde{K}$,

$$\begin{aligned}&G_{x}\Psi_x (u)=\Pi^{\tilde{P}_{\tilde{F}(x)}}_{Q_h}\circ G_x u+\\
&\sum_{N=0}^{+\infty}(G^{N+1}_{\tilde{F}(x)}\mid_{\tilde{P}_{\tilde{F}(x)}})^{-1}\circ\Pi^{\tilde{P}_{\tilde{F}^{N+1}(\tilde{F}(x))}}_{Q_h} \circ G_{\tilde{F}^{N}(\tilde{F}(x))}\circ \Pi^{Q_{h(\tilde{F}^{N}(\tilde{F}(x)))}}_{\tilde{P}}\circ G^{N}_{\tilde{F}(x)}\circ\Pi^{Q_{h(\tilde{F}(x))}}_{\tilde{P}}G_x u\\
&=\Pi^{\tilde{P}_{\tilde{F}(x)}}_{Q_h}\circ G_x u+\Psi_{\tilde{F}(x)}(G_x u).\end{aligned}$$ Thus, we have proved the assertion.

Let 

\begin{equation}\label{Def-tiQ}\tilde{Q}_x=\{u-\Psi_x(u):\,u\in Q_{h(x)}\}\end{equation} for any $x\in \tilde{K}$. For any $u-\Psi_x(u)$ with $u\in Q_{h(x)}$, one has that 

$$\begin{aligned}&G_x(u-\Psi_x(u))=G_xu-G_x\circ\Psi_x(u)=G_xu-\Pi^{\tilde{P}_{\tilde{F}(x)}}_{Q_h}\circ G_x u-\Psi_{\tilde{F}(x)}(G_x u)\\
=&\Pi^{Q_{h(\tilde{F}(x))}}_{\tilde{P}}G_x u-\Psi_{\tilde{F}(x)}(G_x u)=\Pi^{Q_{h(\tilde{F}(x))}}_{\tilde{P}}G_x u-\Psi_{\tilde{F}(x)}(\Pi^{Q_{h(\tilde{F}(x))}}_{\tilde{P}}G_x u)\in\tilde{Q}_{\tilde{F}(x)}.\end{aligned}$$ Thus, $G_x\tilde{Q}_x\subset \tilde{Q}_{\tilde{F}(x)}$ for any $x\in \tilde{K}$. Clearly, $\tilde{Q}_x\cap\tilde{P}_x=\{0\}$ for any $x\in \tilde{K}$. (Otherwise, $u-\Psi_x(u)\in\tilde{P}_x$ for some $u\in Q_{h(x)}\setminus\{0\}$. Together with $\Psi_x(u)\in\tilde{P}_x\setminus\{0\}$, one has that $u\in\tilde{P}_x$, a contracdition to $Q_{h(x)}\cap\tilde{P}_x=\{0\}$.) On the other hand, 

$$v=\Pi^{\tilde{P}_x}_{Q_h} v+\Pi^{Q_{h(x)}}_{\tilde{P}}v=\Pi^{Q_{h(x)}}_{\tilde{P}}v-\Psi_x(\Pi^{Q_{h(x)}}_{\tilde{P}}v)+  \Pi^{\tilde{P}_x}_{Q_h}v+\Psi_x(\Pi^{Q_{h(x)}}_{\tilde{P}}v)$$ for any $v\in X$. Note that $Q_{h(x)}$ is closed and so is $\tilde{Q}_x$ for any $x\in\tilde{K}$. By virtue of Hahn-Banach theorem, there is a $l_{x_i}\in X^*\cap S^*$ such that ${\rm Ker}(l_{x_i})=\tilde{Q}_x\oplus Y_{-i}$ for any $i\in\{1,2,\cdots,k\}$, where $Y_{-i}$ is the $k-1$ dimensional subspace mentioned in (\ref{dx}); and then,  

\begin{equation}\label{Q-invariance} \begin{aligned}&\tilde{K}\times (\tilde{Q}_x) \,\,{\rm is \,\,a}\,\, k-{\rm codimensional\,\, bundle\,\, such\,\, that}\\
&X=\tilde{P}_x\oplus \tilde{Q}_x\quad {\rm and} \quad G_x\tilde{Q}_x\subset \tilde{Q}_{\tilde{F}(x)}\end{aligned}\end{equation} for any $x\in \tilde{K}$.

\vskip 5mm
\noindent{\bf Step 5.  The $k$-exponential separation type property of $(\tilde{F},\,\mathcal{G})$.}
\vskip 5mm

Clearly, for any $v_{\tilde{Q}_x}\in \tilde{Q}_x$ with $x\in \tilde{K}$, ones have that $v_{\tilde{Q}_x}=u_{Q_{h(x)}}-\Psi_x(u_{Q_{h(x)}})$, where $u_{Q_{h(x)}}=\Pi^{Q_{h(x)}}_{\tilde{P}}v_{\tilde{Q}_x} \in Q_{h(x)}$.  By virtue of (\ref{Covergence-norm-Psi}), one has that 

\begin{equation}\label{Q-t-Q-propotion}1+\delta_{\Pi^{\tilde{P}_K}_{Q_h}}\geq \frac{\norm{u_{Q_{h(x)}}}}{\norm{v_{\tilde{Q}_x}}}\geq\frac{\norm{u_{Q_{h(x)}}}}{\norm{u_{Q_{h(x)}}}+\norm{\Psi_x(u_{Q_{h(x)}})}}\geq \frac{1}{1+M_4}\end{equation} for any $v_{\tilde{Q}_x}\in \tilde{Q}_x\setminus\{0\}$ with $x\in \tilde{K}$. Note that $G_x^{n}v_{\tilde{Q}_x}=\Pi^{Q_{h(\tilde{F}^{n}(x))}}_{\tilde{P}}G^{n}_x v_{\tilde{Q}_x}-\Psi_x(\Pi^{Q_{h(\tilde{F}^{n}(x))}}_{\tilde{P}}G^{n}_x v_{\tilde{Q}_x})$ for any $v_{\tilde{Q}_x}\in \tilde{Q}_x$ with $x\in \tilde{K}$. Then, 

\begin{equation}\label{Q-t-Q-propotion-with-n} \norm{  \Pi^{Q_{h(\tilde{F}^{n}(x))}}_{\tilde{P}}G^{n}_x v_{\tilde{Q}_x} } \geq  \frac{1}{1+M_4}\cdot\norm{G_x^{n}v_{\tilde{Q}_x}}\end{equation} for any $n\in\mathbb{N}^+$ and $v_{\tilde{Q}_x}\in \tilde{Q}_x\setminus\{0\}$ with $x\in \tilde{K}$. It then follows that 

\begin{equation}\label{Separation-G-N}\begin{aligned}&\frac{\norm{G^{N^{\prime}n}_x v_{\tilde{Q}_{x}}} }{1+M_4}\leq\norm{\Pi^{Q_{h(\tilde{F}^{N^{\prime}n}(x))}}_{\tilde{P}}G^{N^{\prime}n}_x v_{\tilde{Q}_x}}\\
&\overset{(\ref{kd-uni-bundle-invariant})\,{\rm and}\,{\rm Range}(\Psi_x)\subset\tilde{P}_x}{=}\norm{\Pi^{Q_{h(\tilde{F}^{N^{\prime}n}(x))}}_{\tilde{P}}G^{N^{\prime}n}_x u_{Q_{h(x)}}}\\
&\overset{(\ref{Separation-G-tilde})\,{\rm and}\,\tilde{\tilde{M}}=\max\{\tilde{M}^{\prime},1\}}{\leq} \tilde{\tilde{M}} \cdot \tilde{\eta}^n\cdot \norm{u_{Q_{h(x)}}}\cdot\norm{G^{N^{\prime}n}_x v_{\tilde{P}_{x}}}\\
&\overset{(\ref{Q-t-Q-propotion})}{\leq}  \tilde{\tilde{M}} \cdot \tilde{\eta}^n\cdot(1+\delta_{\Pi_{Q_h}^{\tilde{P}_K}})\cdot\norm{G^{N^{\prime}n}_x v_{\tilde{P}_{x}}}
\end{aligned}\end{equation} for any $n\in\mathbb{N}$, $v_{\tilde{P}_{x}}\in \tilde{P}_{x}\cap S$ and $v_{\tilde{Q}_{x}}\in \tilde{Q}_{x}\cap S$, where $\tilde{\eta}=\frac{12\delta_{\Pi^{P_K}_Q}+5}{12\delta_{\Pi^{P_K}_Q}+7}$. Let 

$$M_5=(\frac{\delta_{norm}}{\tilde{\eta}{\frac{1}{N^{\prime}}}})^{N^{\prime}-1}\cdot\delta_{G^{-1}_{R.}}\cdot\tilde{\tilde{M}}(1+M_4)\cdot(1+\delta_{\Pi_{Q_h}^{\tilde{P}_K}}).$$ Then, one has that 

\begin{equation}\label{Separation-G}\norm{G^{n}_x v_{\tilde{Q}_{x}}}\leq M_5\cdot  \eta^n\cdot \norm{G^{n}_x v_{\tilde{P}_{x}}}
\end{equation}for any $n\in\mathbb{N}$, $v_{\tilde{P}_{x}}\in \tilde{P}_{x}\cap S$ and $v_{\tilde{Q}_{x}}\in \tilde{Q}_{x}\cap S$, where $\eta=\tilde{\eta}^{\frac{1}{N^{\prime}}}$. By virtue of (\ref{kd-uni-bundle-invariant}), (\ref{Q-invariance}) and (\ref{Separation-G}), 

\begin{equation}\label{k-es-tp} (\tilde{F},\mathcal{G})\,\, {\rm on} \tilde{K}\times X\,\, {\rm has\,\, a}\,\, k-{\rm exponential\,\, separation\,\, type\,\, property}.\end{equation}

\vskip 5mm
\noindent{\bf Step 6. The continuity of the certain invariant $k$-dimensional bundle w.r.t. $(\tilde{F},\mathcal{G})$.}
\vskip 5mm
Let 

$$C^{\lambda}_{x}=\{v\in X:\,\norm{\Pi^{Q_{h(x)}}_P v}\leq\lambda\cdot\norm{\Pi^{P_{h(x)}}_Qv }\}$$ for any $\lambda\geq 0$ and $x\in \tilde{K}$. We have the following estimates for any $x,\,y\in \tilde{K}$ such that $\norm{x-y}+2\delta\leq\delta_0$ and $v_x\in C^{\lambda}_{x}\cap S$:

\begin{equation}\label{Difference-cone-field}\begin{aligned}&\norm{\Pi_P^{Q_{h(y)}}v_x}\leq \norm{\Pi_P^{Q_{h(x)}}v_x}+\norm{(\Pi_P^{Q_{h(y)}}-\Pi_P^{Q_{h(x)}})v_x}\\
&\overset{(\ref{Pertu-K-F})+(\ref{Pi-c-bound})}{\leq}\norm{\Pi_P^{Q_{h(x)}}v_x}+c(\norm{x-y}+2\delta)\\
&\overset{v_x\in C^{\lambda}_{x}\cap S+(\ref{Pertu-K-F})+(\ref{Pi-c-bound})}{\leq}\lambda\cdot \norm{\Pi_Q^{P_{h(y)}}v_x}+(1+\lambda)\cdot c(\norm{x-y}+2\delta).
\end{aligned}\end{equation} Note that $1=\norm{v_x}\leq \norm{\Pi^{Q_{h(x)}}_P v_x}+ \norm{\Pi^{P_{h(x)}}_Q v_x}\leq (1+\lambda)\cdot \norm{\Pi^{P_{h(x)}}_Q v_x}\leq (1+\lambda)\cdot(\norm{\Pi^{P_{h(y)}}_Q v_x}+ c(\norm{x-y}+2\delta))$; and more, 

\begin{equation}\label{QtoP-Lowbound} \norm{\Pi_Q^{P_{h(y)}} v_x}\geq \frac{1}{1+\lambda}-c(\norm{x-y}+2\delta)
\end{equation} for any $x,\,y\in \tilde{K}$ such that $\norm{x-y}+2\delta\leq\delta_0$ and $v_x\in C^{\lambda}_{x}\cap S$. Together with (\ref{Difference-cone-field}) and (\ref{QtoP-Lowbound}), we have the following estimates for any $x,\,y\in \tilde{K}$ such that $\norm{x-y}+2\delta\leq\delta_0$, $v_x\in C^{\lambda}_{x}\cap S$ and $\lambda\in [0,\frac{1}{c(\norm{x-y}+2\delta)}-1)$:

\begin{equation}\label{Intersection-cone-field} \frac{\norm{\Pi^{Q_{h(y)}}_P v_x} }{ \norm{\Pi^{P_{h(y)}}_Q v_x} }\leq \lambda+\frac{(1+\lambda)^2\cdot c(\norm{x-y}+2\delta)}{1-(1+\lambda)\cdot c(\norm{x-y}+2\delta)}.\end{equation} Recall that $c(\cdot)$ is a nonnegative and increasing function on $\mathbb{R}^+$ such that $\lim\limits_{t\rightarrow 0^+}c(t)=0$. It entails that one can find a positive constant $\delta_{\varepsilon}\in (0,\delta_{\varepsilon}^{\prime\prime})$ such that $c(3\delta)\leq \frac{1}{10}$ for any $\delta\in[0,\delta_{\epsilon}]$. 

In the rest of proof, we always assume $\delta\in[0,\delta_{\epsilon}]$. Thus, the following relationships hold:

\begin{equation}\label{Is-cf-cndelta}\begin{aligned}&C^{\lambda}_{x}\subset C_{y}^{f(\lambda)},\,\,C_{y}^{\lambda}\subset C_{x}^{f(\lambda)}\,\,{\rm with}\,\,\lambda\in[0,9)\\
&{\rm and\,\,specially,}\\
&C_{x}^{\frac{2}{3}}\cup C_{y}^{\frac{2}{3}}\subset  C_x\cap C_y,\\
&\overline{\mathbb{R}\cdot (G_{\tilde{F}^{-N^{\prime}}(x)}^{N^{\prime}} C_{\tilde{F}^{-N^{\prime}}(x)}\cap S+B_{\frac{4}{15\delta_{\Pi^{P_K}_Q}+9}}(0))}\subset C_{x}^{\frac{2}{3}},\\
&\overline{\mathbb{R}\cdot (G_{\tilde{F}^{-N^{\prime}}(y)}^{N^{\prime}} C_{\tilde{F}^{-N^{\prime}}(y)}\cap S+B_{\frac{4}{15\delta_{\Pi^{P_K}_Q}+9}}(0))}\subset C_{y}^{\frac{2}{3}}
\end{aligned}\end{equation} for any $x,\,y\in \tilde{K}$ with $\delta\in [0,\delta_{\epsilon}]$ such that $\norm{x-y}\leq \delta$, where $f(\lambda)=\frac{11\lambda+1}{9-\lambda}$. It then follows from (\ref{G-contraction-delta}), (\ref{kd-uni-bundle-invariant}), (\ref{Is-cf-cndelta}) and $\tilde{F}$ being a homemorphism on $\tilde{K}$ that for any $x,y\in\tilde{K}$ such that $\norm{x-y}\leq\delta$, one has that 

\begin{equation}\label{dist-tP-bel}\tilde{P}_x\subset C_x^{\frac{1}{2}}\subset C^{\frac{13}{17}}_y\quad{\rm and}\quad \tilde{P}_y\subset C_y^{\frac{1}{2}}\subset C^{\frac{13}{17}}_x\end{equation}

Now, we assert that 

\begin{equation}\label{loc-tilide-Q} \tilde{Q}_x\cap C_x=\{0\}\quad {\rm for\,\,any}\,\,x\in \tilde{K}.
\end{equation} Prove by contrary. Suppose that there is a $w_{\tilde{Q}_z}\in \tilde{Q}_z\cap C_z\cap S$ with a certain point $z\in \tilde{K}$. Together with (\ref{tP-Qh-dsum}), $Q_{h(z)}\cap C_z=\{0\}$ and (\ref{Q-t-Q-propotion}), one has that $\Pi^{\tilde{P}_z}_{Q_h} w_{\tilde{Q}_z}=-\Psi_{z}(\Pi_{\tilde{P}}^{Q_{h(z)}}w_{\tilde{Q}_z})\in \tilde{P}_z\setminus\{0\}$ such that $\norm{\Pi_{\tilde{P}}^{Q_{h(z)}}w_{\tilde{Q}_z}}\in [\frac{1}{1+M_4},\,1+\delta_{\Pi^{\tilde{P}_K}_{Q_h}}]$. It furthermore follows from (\ref{second-layer-focu}) and (\ref{Q-invariance}) that 

$$G^{N^{\prime}n}_z w_{\tilde{Q}_z}\in \tilde{Q}_{\tilde{F}^{N^{\prime}n}(z)}\cap {\rm Int}\,\tilde{\tilde{C}}_{\tilde{F}^{N^{\prime}n}(z)}$$ for any $n\in\mathbb{N}^+$. Combinating with (\ref{kd-uni-bundle-invariant})-(\ref{tP-Qh-dsum}) and (\ref{Covergence-norm-Psi})-(\ref{Q-invariance}), ones have that

\begin{equation}\label{ratio-tP-Qh}\begin{aligned}&\Pi^{\tilde{P}_{\tilde{F}^{N^{\prime}n}(z)}}_{Q_h}G^{N^{\prime}n}_z w_{\tilde{Q}_z}=-\Psi_{\tilde{F}^{N^{\prime}n}(z)}(\Pi_{\tilde{P}}^{Q_{h(\tilde{F}^{N^{\prime}n}(z) )}}G^{N^{\prime}n}_z w_{\tilde{Q}_z})\in\tilde{P}_{\tilde{F}^{N^{\prime}n}(z)}\setminus\{0\}\subset {\rm Int}\tilde{\tilde{C}}_{ \tilde{F}^{N^{\prime}n}(z)}\\
&{\rm and}\\
& \Pi_{\tilde{P}}^{Q_{h( \tilde{F}^{N^{\prime}n}(z) )}}G^{N^{\prime}n}_z w_{\tilde{Q}_z}\in Q_{h(\tilde{F}^{N^{\prime}n}(z))}\setminus\{0\}\subset X\setminus C_{ \tilde{F}^{N^{\prime}n}(z)}\,\,{\rm such\,\,that}\\
&\frac{\norm{\Pi_{\tilde{P}}^{Q_{h(\tilde{F}^{N^{\prime}n}(z))}}G^{N^{\prime}n}_z w_{\tilde{Q}_z}}}{\norm{\Pi^{\tilde{P}_{\tilde{F}^{N^{\prime}n}(z)}}_{Q_h}G^{N^{\prime}n}_z w_{\tilde{Q}_z}}}\geq \frac{1}{M_4}
\end{aligned}\end{equation} for any $n\in\mathbb{N}^+$. Obviously,

$$G^{N^{\prime}n}_z w_{\tilde{Q}_z}-\Pi^{\tilde{P}_{\tilde{F}^{N^{\prime}n}(z)}}_{Q_{h}}G^{N^{\prime}n}_z w_{\tilde{Q}_z}=\Pi_{\tilde{P}}^{Q_{h(\tilde{F}^{N^{\prime}n}(z))}}G^{N^{\prime}n}_z w_{\tilde{Q}_z}\in X\setminus C_{ \tilde{F}^{N^{\prime}n}(z)}$$ for any $n\in\mathbb{N}$. It then follows that  

\begin{equation}
\begin{aligned}&\alpha^{\tilde{\tilde{C}}_{\tilde{F}^{N^{\prime}n}(z)}}_0(G^{N^{\prime}n}_z w_{\tilde{Q}_z}, \Pi^{\tilde{P}_{\tilde{F}^{N^{\prime}n}(z)}}_{Q_{h}}G^{N^{\prime}n}_z w_{\tilde{Q}_z}) >1\\
&{\rm and}\\
&\alpha^{\tilde{\tilde{C}}_{\tilde{F}^{N^{\prime}n}(z)}}_0(\Pi^{\tilde{P}_{\tilde{F}^{N^{\prime}n}(z)}}_{Q_{h}}G^{N^{\prime}n}_z w_{\tilde{Q}_z}, G^{N^{\prime}n}_z w_{\tilde{Q}_z}) >1
\end{aligned}
\end{equation} for any $n\in\mathbb{N}^+$. By virtue of the estimate (\ref{C..L-distance}) in Lemma \ref{SF-angles}(iii), one has that 

$$\Pi_{\tilde{P}}^{Q_{h(\tilde{F}^{N^{\prime}n}(z))}}G^{N^{\prime}n}_z w_{\tilde{Q}_z}+{\rm Int}\,B_{\norm{\Pi_{\tilde{P}}^{Q_{h(\tilde{F}^{N^{\prime}n}(z))}}G^{N^{\prime}n}_z w_{\tilde{Q}_z}}\cdot \frac{\kappa_{\tilde{F}^{N^{\prime}(n-1)}(z)}-\tilde{\kappa}_{\tilde{F}^{N^{\prime}(n-1)}(z)}}{1+\kappa_{\tilde{F}^{N^{\prime}(n-1)}(z)}}}(0)\subset X\setminus \tilde{\tilde{C}}_{\tilde{F}^{N^{\prime}n}(z)}$$ and then, by virtue of (\ref{first-layer-focu})-(\ref{second-layer-focu}) and (\ref{ratio-tP-Qh}), one has 

\begin{equation}\label{alphtt-lowb}\alpha^{\tilde{\tilde{C}}_{\tilde{F}^{N^{\prime}n}(z)}}_0(\Pi^{\tilde{P}_{\tilde{F}^{N^{\prime}n}(z)}}_{Q_{h}}G^{N^{\prime}n}_z w_{\tilde{Q}_z}, G^{N^{\prime}n}_z w_{\tilde{Q}_z})\geq 1+\frac{1}{4M_4(3\delta_{\Pi_{Q_h}^{P_K}}+2)}
\end{equation} for any $n\in\mathbb{N}^+$. Together with Lemma \ref{P-alp-0}(ii)-(iii) and the compactness of $\tilde{P}_{\tilde{F}^{N^{\prime}n}(z)}\cap S$ for any $n\in\mathbb{N}^+$, one has that 

\begin{equation}\label{Gnwq-P-angle-upper-b}1+\frac{1}{4M_4(3\delta_{\Pi_{Q_h}^{P_K}}+2)}\leq \sup\limits_{v\in \tilde{P}_{\tilde{F}^{N^{\prime}n}(z)}\cap S }\{ \alpha^{\tilde{\tilde{C}}_{\tilde{F}^{N^{\prime}n}(z)}}(G^{N^{\prime}n}_z w_{\tilde{Q}_z}, v)\}<+\infty \,\,{\rm for\,\, any\,}\, n\in\mathbb{N}^+.\end{equation} Furthermore, together with (a.), Lemma \ref{P-alp-0}(ii), the first assertion with (\ref{k-angle-constraction-vector}) in proof of Lemma \ref{SF-angles}(ii) and (\ref{kd-uni-bundle-invariant})-(\ref{tP-inner}), one has that 

\begin{equation}\label{tQ-loca-angl-contract-b} \begin{aligned}&\sup\limits_{v\in \tilde{P}_{\tilde{F}^{N^{\prime}(n+1)}(z)}\cap S }\{ {\rm In}(\alpha^{\tilde{\tilde{C}}_{\tilde{F}^{N^{\prime}(n+1)}(z)}}(G^{N^{\prime}(n+1)}_z w_{\tilde{Q}_z}, v)) \} \\
&\leq \frac{1-\tilde{\tilde{\kappa}}_{ \tilde{F}^{N^{\prime}n}(z) }(N^{\prime})}{1+\tilde{\tilde{\kappa}}_{ \tilde{F}^{N^{\prime}n}(z) }(N^{\prime})}\cdot \sup\limits_{v\in \tilde{P}_{\tilde{F}^{N^{\prime}n}(z) }\cap S }\{ {\rm In}(\alpha^{\tilde{\tilde{C}}_{\tilde{F}^{N^{\prime}n}(z) }}(G^{N^{\prime}n}_z w_{\tilde{Q}_z}, v) )\}\\
&\leq \frac{12\delta_{\Pi_{Q}^{P_K}} +5}{12\delta_{\Pi_{Q}^{P_K}}+7}\cdot \sup\limits_{v\in \tilde{P}_{\tilde{F}^{N^{\prime}n}(z) }\cap S }\{ {\rm In}(\alpha^{\tilde{\tilde{C}}_{\tilde{F}^{N^{\prime}n}(z) }}(G^{N^{\prime}n}_z w_{\tilde{Q}_z}, v) )\}\end{aligned}
\end{equation} for any $n\in\mathbb{N}^+$. It entails that 

$$\sup\limits_{v\in \tilde{P}_{\tilde{F}^{N^{\prime}n}(z)}\cap S }\{ \alpha^{\tilde{\tilde{C}}_{\tilde{F}^{N^{\prime}n}(z)}}(G^{N^{\prime}n}_z w_{\tilde{Q}_z}, v)\}<1+\frac{1}{4M_4(3\delta_{\Pi_{Q_h}^{P_K}}+2)}$$ for sufficiently large $n\in\mathbb{N}^+$, a contradiction to (\ref{Gnwq-P-angle-upper-b}). Thus, we have proved this assertion (\ref{loc-tilide-Q}).

Now, we assert that the invariant $k$-dimensional bundle $\tilde{K}\times (\tilde{P}_x)$ w.r.t. $(\tilde{F},\mathcal{G})$ is continuous. Prove by contrary. Suppose that there is a point $x\in\tilde{K}$, a sequence $\{x_m\}\subset\tilde{K}$ and a positive number $M_6>0$ such that 

$$\lim\limits_{n\rightarrow +\infty}x_m=x\quad {\rm and} \quad d(\tilde{P}_{x_m},\tilde{P}_x)> M_6.$$ As a consequence, there is a number $M_7\in (0,1)$ and a sequence $\{v_m\}_{m=1}^{+\infty}$ such that $v_m\in\tilde{P}_{x_m}\cap S$ and 

\begin{equation}\label{sequ-propo-lb}\frac{\norm{\Pi_{\tilde{P}}^{\tilde{Q}_x} v_m}}{\norm{\Pi_{\tilde{Q}}^{\tilde{P}_x} v_m}}>M_7\end{equation} for any $m\in\mathbb{N}^+$. By virtue of (\ref{kd-uni-bundle-invariant}), there is a unique point $v_{m,-n}\in \tilde{P}_{\tilde{F}^{-n}(x_m)}\setminus\{0\}$ such that $$G_{\tilde{F}^{-n}(x_m)}^n v_{m,-n}=v_m$$ for any $n\in\mathbb{N}^+$ and $m\in \mathbb{N}^+$. Take $N^{\prime\prime}\in\mathbb{N}^+$ such that 

$$M_5\cdot\eta^{N^{\prime\prime}}<\frac{4(1+M_4)}{17+30\delta_{\Pi_Q^{P_K}}}\cdot M_7,$$ where $M_5,\eta$ are the constants in (\ref{Separation-G}). It entails that 

$$ \frac{\norm{\Pi_{\tilde{P}}^{\tilde{Q}_{\tilde{F}^{-N^{\prime\prime}}(x) }} v_{m,-N^{\prime\prime}}}}{\norm{\Pi_{\tilde{Q}}^{\tilde{P}_{\tilde{F}^{-N^{\prime\prime}}(x)}}v_{m,-N^{\prime\prime}}}}>\frac{17+30\delta_{\Pi_Q^{P_K}}}{4(1+M_4) }$$ for sufficiently large $m\in\mathbb{N}^+$. Together with (\ref{Q-t-Q-propotion}), one has that 

$$\norm{\Pi^{Q_{h(\tilde{F}^{-N^{\prime\prime}}(x))}}_P \circ \Pi_{\tilde{P}}^{\tilde{Q}_{\tilde{F}^{-N^{\prime\prime}}(x)}}v_{m,-N^{\prime\prime}}}> \frac{17+30\delta_{\Pi_Q^{P_K}}}{4}\cdot \norm{\Pi_{\tilde{Q}}^{\tilde{P}_{\tilde{F}^{-N^{\prime\prime}}(x) }} v_{m,-N^{\prime\prime}}}.$$ It then follows from (\ref{G-contraction-delta}), (\ref{kd-uni-bundle-invariant}), (\ref{loc-tilide-Q}) that 

$$  \frac{\norm{\Pi_{P}^{Q_{h(\tilde{F}^{-N^{\prime\prime}}(x) )}} v_{m,-N^{\prime\prime}}}}  {\norm{\Pi_{Q}^{P_{h(\tilde{F}^{-N^{\prime\prime}}(x))}}v_{m,-N^{\prime\prime}}}}>\frac{13}{17}$$ for sufficiently large $m\in\mathbb{N}^+$. Since $\tilde{F}$ is a homeomorphism on $\tilde{K}$, there is an integer $M_8>0$ such that 

$$\norm{\tilde{F}^{-N^{\prime\prime}}(x_m)-\tilde{F}^{-N^{\prime\prime}}(x) }\leq \delta$$ for any $m\geq M_8$. It then follows from (\ref{dist-tP-bel}) that for any $m\geq M_8$, ones have that $\tilde{P}_{\tilde{F}^{-N^{\prime\prime}}(x_m)}\subset C_{\tilde{F}^{-N^{\prime\prime}}(x_m)}^{\frac{1}{2}}\subset C_{\tilde{F}^{-N^{\prime\prime}}(x) }^{\frac{13}{17}}$ and hence, 

\begin{equation}\label{propotion-inti} \frac{\norm{\Pi_{P}^{Q_{h(\tilde{F}^{-N^{\prime\prime}}(x) )}} v_{m,-N^{\prime\prime}}}}  {\norm{\Pi_{Q}^{P_{h(\tilde{F}^{-N^{\prime\prime}}(x))}}v_{m,-N^{\prime\prime}}}}\leq\frac{13}{17},
\end{equation} a contradiction. Thus, $\tilde{K}\times (\tilde{P}_x)$ is continuous, and we have proved this assertion.

\vskip 5mm
\noindent{\bf Step 7. The continuity of the certain positive invariant $k$-codimensional bundle w.r.t. $(\tilde{F},\mathcal{G})$.}
\vskip 5mm

In this step, we will prove that $\tilde{K}\times(\tilde{Q}_x)$ is continuous, where $\tilde{K}\times(\tilde{Q}_x)$ is the one in (\ref{Q-invariance}). After giving necessary objects for decribing properties of positive invariant bundles, we will give some assertions to draw the skeleton of the positive invariant bundles w.r.t. $(\tilde{F},\mathcal{G})$. For any given $z\in \tilde{K}$, it follows from the arguments for well-defining the correction map $\Psi_x, x\in \tilde{K}$ in step 3 of proof that there is a $k$-tuple $\{v_z^1,v_z^2,\cdots,v_z^k\}\subset\tilde{P}_z\cap S$ such that 

\begin{equation}\label{Pz-sk}\begin{aligned} &{\rm Span}\{v_z^1,v_z^2,\cdots,v_z^k\}=\tilde{P}_z,\\
&d_{usual}(v_z^{i+1},\,{\rm Span}\{v_z^1,\,v_z^2\,\cdots,v_z^i\})>\frac{3}{4}\quad{\rm for\,\,all}\,\,i\in\{1,2,\cdots,k-1\},\\
&d_z=\inf\limits_{l\in\{1,2,\cdots,k\}} d_{usual}(v_z^l,Y_z^{-l})\geq (\frac{3}{7})^{k-1},
\end{aligned}\end{equation} where $Y_z^{-l}={\rm Span}\{\{v_z^1,v_z^2,\cdots,v_z^k\}\setminus\{v_z^l\}\}$. Since $\tilde{K}\times (\tilde{P}_x)$ is continuous, for any $\tilde{\varepsilon}\in (0, \frac{\min\{\frac{3}{4},\,(\frac{3}{7})^{k-1}\}}{3})$, there is a $\tilde{\delta}_{\tilde{\varepsilon}}>0$ such that for any $\tilde{z}\in B_{\tilde{\delta}_{\tilde{\varepsilon}}}(z)$, there is a $k$-tuple $\{v_{\tilde{z}}^1,v_{\tilde{z}}^2,\cdots,v_{\tilde{z}}^k\}\subset\tilde{P}_{\tilde{z}}\cap S$ such that  

\begin{equation}\label{Ptz-sk}\begin{aligned}&{\rm Span}\{v_{\tilde{z}}^1,v_{\tilde{z}}^2,\cdots,v_{\tilde{z}}^k\}=\tilde{P}_{\tilde{z}}\,\,{\rm with  }\,\,\norm{v_{\tilde{z}}^i-v_z^i}\leq\tilde{\varepsilon}\,\,{\rm for\,\,all}\,\,i\in\{1,2,\cdots,k\},\\
&d_{usual}(v_{\tilde{z}}^{i+1},\,{\rm Span}\{v_{\tilde{z}}^1,\,v_{\tilde{z}}^2\,\cdots,v_{\tilde{z}}^i\})>\frac{3}{4}-\tilde{\varepsilon}\,\,{\rm for\,\,all}\,\,i\in\{1,2,\cdots,k-1\},\\
&\abs{d_{usual}(v_{\tilde{z}}^i,Y_{\tilde{z}}^{-i})-d_{usual}(v_z^i,Y_z^{-i})}<\tilde{\varepsilon}\,\,{\rm for\,\, any}\,\,i\in\{1,2,\cdots,k\}\\
&{\rm and\,\,hence},\,\,d_{\tilde{z}}=\inf\limits_{i\in\{1,2,\cdots,k\}} d_{usual}(v_{\tilde{z}}^i,Y_{\tilde{z}}^{-i})\geq (\frac{3}{7})^{k-1}-\tilde{\varepsilon},
\end{aligned}\end{equation} where $Y_{\tilde{z}}^{-i}={\rm Span}\{\{v_{\tilde{z}}^1,v_{\tilde{z}}^2,\cdots,v_{\tilde{z}}^k\}\setminus\{v_{\tilde{z}}^i\}\}$. Let 

$$ d^i_{\tilde{z}}=d_{usual}(v_{\tilde{z}}^i, Y^{-i}_{\tilde{z}}\oplus \tilde{Q}_{\tilde{z}}) $$ for any $i\in\{1,2,\cdots,k\}$ and $\tilde{z}\in B_{\tilde{\delta}_{\tilde{\varepsilon}}}(z)\cap \tilde{K}$ with $\tilde{\varepsilon}\in(0, \frac{\min\{\frac{3}{4},\,(\frac{3}{7})^{k-1}\}}{3})$. Together with (\ref{Def-t-ttCxN})-(\ref{first-layer-focu}), (\ref{kd-uni-bundle-invariant}), (\ref{loc-tilide-Q}) and (\ref{Ptz-sk}), one has that 

\begin{equation}\label{dzi-lower-b}d^i_{\tilde{z}}\in[\frac{(\frac{3}{7})^{k-1}-\tilde{\varepsilon}}{6\delta_{\Pi_Q^{P_K}}+3},1]
\end{equation} for any $i\in\{1,2,\cdots,k\}$ and $\tilde{z}\in B_{\tilde{\delta}_{\tilde{\varepsilon}}}(z)\cap \tilde{K}$ with $\tilde{\varepsilon}\in(0, \frac{\min\{\frac{3}{4},\,(\frac{3}{7})^{k-1}\}}{3})$. By Hahn-Banach theorem, there is an unit functional $l_{\tilde{z}}^i\in X^*$ such that for any $i\in\{1,2,\cdots,k\}$ and $\tilde{z}\in B_{\tilde{\delta}_{\tilde{\varepsilon}}}(z)\cap \tilde{K}$ with $\tilde{\varepsilon}\in (0, \frac{\min\{\frac{3}{4},\,(\frac{3}{7})^{k-1}\}}{3})$, ones have that

\begin{equation}\label{Q*-sk}\begin{aligned}&\norm{l_{\tilde{z}}^i}=1,\,\,l_{\tilde{z}}^i(v_{\tilde{z}}^i)=d_{\tilde{z}}^i\\
&{\rm and}\\
&l_{\tilde{z}}^i(v)=0\,\,{\rm for\,\,any\,\,} v\in Y_{\tilde{z}}^{-i}\oplus\tilde{Q}_{\tilde{z}}.
\end{aligned}\end{equation}  Denoted by 

$$L_{\tilde{z}}={\rm Span}\{l_{\tilde{z}}^i: i\in\{1,2,\cdots,k\}\}$$ for any $\tilde{z}\in B_{\tilde{\delta}_{\tilde{\varepsilon}}}(z)\cap \tilde{K}$ with $\tilde{\varepsilon}\in (0, \frac{\min\{\frac{3}{4},\,(\frac{3}{7})^{k-1}\}}{3})$. It follows from (\ref{Q*-sk}) that $L_{\tilde{z}}\subset X^*$ are $k$-dimensional subspace of $X^*$ such that 

\begin{equation}\label{Ker-Ltz}{\rm Ker}(L_{\tilde{z}})=\tilde{Q}_{\tilde{z}}\end{equation} for any $\tilde{z}\in B_{\tilde{\delta}_{\tilde{\varepsilon}}}(z)\cap \tilde{K}$ with $\tilde{\varepsilon}\in (0, \frac{\min\{\frac{3}{4},\,(\frac{3}{7})^{k-1}\}}{3})$.

\vskip 5mm
{\bf (As-1)} We assert that 

\begin{equation}l_{\tilde{z}}\in L_{\tilde{z}} \,\,{\rm if\,\, and\,\, only\,\, if}\,\, \tilde{Q}_{\tilde{z}}\subset {\rm Ker}(l_{\tilde{z}})\end{equation} for any $\tilde{z}\in B_{\tilde{\delta}_{\tilde{\varepsilon}}}(z)\cap \tilde{K}$ with $\tilde{\varepsilon}\in (0, \frac{\min\{\frac{3}{4},\,(\frac{3}{7})^{k-1}\}}{3})$. For such a $\tilde{z}$, it is clear that $l_{\tilde{z}}\in L_{\tilde{z}}$ implies $\tilde{Q}_{\tilde{z}}\subset {\rm Ker}(l_{\tilde{z}}) $. It suffices to prove that $l_{\tilde{z}}\in L_{\tilde{z}}$ for any $l_{\tilde{z}}\in X^*$ such that $\tilde{Q}_{\tilde{z}}\subset {\rm Ker}(l_{\tilde{z}}) $. For a given $v\in X$, suppose that $v=w_{\tilde{z}}+\sum\limits_{i=1}^k\lambda_i v_{\tilde{z}}^i$ with $w_{\tilde{z}}\in \tilde{Q}_{\tilde{z}}$ and $\lambda_i\in \mathbb{R},\,i\in\{1,2,\cdots,k\}$. Then, $l_{\tilde{z}}(v)=\sum\limits_{i=1}^k\lambda_i l_{\tilde{z}}(v_{\tilde{z}}^i)$. It then follows from (\ref{dzi-lower-b})-(\ref{Q*-sk}) that 

\begin{equation}\label{lr} l_{\tilde{z}}(v)=\sum\limits_{i=1}^k\frac{l_{\tilde{z}}(v_{\tilde{z}}^i)}{d^i_{\tilde{z}}} l_{\tilde{z}}^i(v)\,\,{\rm for\,\, any}\,\,v\in X\,\,{\rm and\,\, hence,} \,\,l_{\tilde{z}}=\sum\limits_{i=1}^k\frac{l_{\tilde{z}}(v_{\tilde{z}}^i)}{d^i_{\tilde{z}}} l_{\tilde{z}}^i
\end{equation} Thus, $l_{\tilde{z}}\in L_{\tilde{z}}$ and the assertion (As-1) has been proved.

\vskip 5mm
{\bf (As-2)} We assert that for any $\tilde{\varepsilon}\in (0, \min\{\frac{1}{4},\,\frac{1}{3}\cdot(\frac{3}{7})^{k-1}, \frac{1}{24\delta_{\Pi_Q^{P_K}}+14}\})$, there is a $\tilde{\delta}^{\prime}_{\tilde{\varepsilon}}\in (0,\tilde{\delta}_{\tilde{\varepsilon}})$ such that for any $k$-codimensional subspace $H\setminus\{0\}\subset X\setminus C_{\tilde{z}}$ with $\tilde{z}\in \tilde{K}\cap B_{\tilde{\delta}^{\prime}_{\tilde{\varepsilon}}}(z)$, ones have that

\begin{equation}\label{rela-dis-H-tQ}\sup\limits_{v\in H\cap S}\{d_{usual}(v, \tilde{Q}_z\cap S)\}<\tilde{\varepsilon} \Rightarrow \sup\limits_{v\in \tilde{Q}_z\cap S}\{d_{usual}(v, H\cap S)\}<2(24\delta_{\Pi_Q^{P_K}}+13)\cdot\tilde{\varepsilon},
\end{equation} and

\begin{equation}\label{rela-dis-tQ-H}\sup\limits_{v\in \tilde{Q}_z\cap S}\{d_{usual}(v, H\cap S)\}<\tilde{\varepsilon} \Rightarrow \sup\limits_{v\in H\cap S}\{d_{usual}(v, \tilde{Q}_z\cap S)\}<2(24\delta_{\Pi_Q^{P_K}}+13)\cdot\tilde{\varepsilon}.
\end{equation} Here, we only prove (\ref{rela-dis-H-tQ}), and the proof for (\ref{rela-dis-tQ-H}) is similiar. Recall that $\tilde{K}\times (\tilde{P}_x)$ is continuous. Then, there is a $\tilde{\delta}^{\prime}_{\tilde{\varepsilon}}\in (0,\tilde{\delta}_{\tilde{\varepsilon}})$ such that $d(\tilde{P}_{\tilde{z}},\tilde{P}_z)\leq \frac{1}{12\delta_{\Pi_{Q}^{P_K}}+6}$ for any $\tilde{z}\in \tilde{K}\cap B_{\tilde{\delta}^{\prime}_{\tilde{\varepsilon}}}(z)$. It then follows from (\ref{first-layer-focu}), (\ref{N-pri-denot}) and (\ref{kd-uni-bundle-invariant}) that for any $v\in \tilde{P}_{z}\cap S$ and $k$-codimensional subspace $H\subset X\setminus C_{\tilde{z}}$ with $\tilde{z}\in \tilde{K}\cap B_{\tilde{\delta}^{\prime}_{\tilde{\varepsilon}}}(z)$, one has 

\begin{equation}\label{dist-vz-Lt} d_{usual}(v, H)\geq  \kappa_{\tilde{F}^{-N^{\prime}}(\tilde{z})}-d(\tilde{P}_{\tilde{z}},\tilde{P}_z)\geq \kappa_{\tilde{F}^{-N^{\prime}}(\tilde{z})}- \frac{1}{12\delta_{\Pi_{Q}^{P_K}}+6}\geq \frac{1}{12\delta_{\Pi_{Q}^{P_K}}+6}.\end{equation} Thus, $X=\tilde{P}_z\oplus H$. Denoted by $\Pi^{\tilde{P}_z\cdot H}$ the projection onto $\tilde{P}_z$ along $H$. By virtue of (\ref{dist-vz-Lt}), one has that 

\begin{equation}\label{Pi-H-Pz-norm}\norm{\Pi^{\tilde{P}_z\cdot H}}_{L(X)}\leq 12\delta_{\Pi_Q^{P_K}}+6\end{equation} for any $k$-codimensional subspace $H\subset X\setminus C_{\tilde{z}}$ with $\tilde{z}\in \tilde{K}\cap B_{\tilde{\delta}^{\prime}_{\tilde{\varepsilon}}}(z)$.

Given such a $k$-codimensional subspace $H\subset X\setminus C_{\tilde{z}}$ with $\tilde{z}\in \tilde{K}\cap B_{\tilde{\delta}^{\prime}_{\tilde{\varepsilon}}}(z)$. If $w_1,\,w_2\in H$ such that $\Pi_{\tilde{P}}^{\tilde{Q}_z}(w_1-w_2)=0$, then $w_1-w_2=\Pi_{\tilde{Q}}^{\tilde{P}_z}(w_1-w_2)\in \tilde{P}_z$. Clearly, $w_1-w_2\in H$. Then, $X=\tilde{P}_z\oplus H$ implies that $w_1=w_2$ and hence, $\Pi_{\tilde{P}}^{\tilde{Q}_z}\mid_{H}$ is injective. On the other hand, for any $w_z\in\tilde{Q}_z$, there is a unique pair of points $w_{z_{P}}\in\tilde{P}_z,\,w_{z_H}\in H$ such that $w_z=w_{z_{P}}+w_{z_H}$, and hence, $w_z=\Pi_{\tilde{P}}^{\tilde{Q}_z}w_{z_H}$. It implies that $\Pi_{\tilde{P}}^{\tilde{Q}_z}\mid_{H}$ is bijective. So, for any $w\in \tilde{Q}_z\cap S$, there is a unique point $w_H\in H$ such that 

$$\Pi_{\tilde{P}}^{\tilde{Q}_z}w_{H}=w.$$ Together with $\sup\limits_{v\in H\cap S}\{d_{usual}(v, \tilde{Q}_z\cap S)\}<\tilde{\varepsilon}$, there is a $\lambda\in(0,1)$ such that $\sup\limits_{v\in H\cap S}\{d_{usual}(v, \tilde{Q}_z\cap S)\}<\lambda\cdot\tilde{\varepsilon} $, and then there is a $w^{\prime}\in\tilde{Q}_z\cap S$ such that $\norm{\frac{w_H}{\norm{w_H}}-w^{\prime}}<\lambda\cdot\tilde{\varepsilon}$. Then, 

\begin{equation}\label{pH-Q}\begin{aligned}&\norm{\frac{w}{\norm{w_H}}-w^{\prime} }=\norm{\Pi_{\tilde{P}}^{\tilde{Q}_z}(\frac{w_H}{\norm{w_H}}-w^{\prime})}<\norm{\Pi_{\tilde{P}}^{\tilde{Q}_z}}_{L(X)}\cdot\lambda\cdot\tilde{\varepsilon}\\
&\overset{(\ref{Pi-H-Pz-norm})\,{\rm and}\,\tilde{\varepsilon}\in(0,\frac{1}{24\delta_{\Pi_Q^{P_K}}+14})}{\leq} \frac{12\delta_{\Pi_Q^{P_K}}+7}{24\delta_{\Pi_Q^{P_K}}+14}=\frac{1}{2},\end{aligned}\end{equation} and 

\begin{equation}\label{pH-P}\begin{aligned}\norm{\frac{\Pi_{\tilde{Q} }^{\tilde{P}_z} w_H}{\norm{w_H}}}=\norm{\Pi_{\tilde{Q} }^{\tilde{P}_z}(\frac{w_H}{\norm{w_H}}-w^{\prime})}<\norm{\Pi_{\tilde{Q}}^{\tilde{P}_z}}_{L(X)}\cdot\lambda\cdot\tilde{\varepsilon}.\end{aligned}\end{equation} Consequently,

\begin{equation}\label{wH-norm}\norm{w_H} \in [ \frac{1}{1+\norm{\Pi_{\tilde{P}}^{\tilde{Q}_z}  }_{L(X)}\cdot\lambda\cdot \tilde{\varepsilon} }, \,\,\frac{1}{ 1-\norm{\Pi_{\tilde{P}}^{\tilde{Q}_z}  }_{L(X)}\cdot \lambda\cdot\tilde{\varepsilon} }]\end{equation} and

$$\begin{aligned}&\norm{\frac{w_H}{\norm{w_H}}-w}\leq\norm{\frac{w_H}{\norm{w_H}}-w_H}+\norm{w_H-w}\\
&\quad\quad\quad\quad \quad\leq \abs{\norm{w_H}-1}+\norm{\Pi_{\tilde{Q} }^{\tilde{P}_z} w_H}\\
&\quad\quad\quad\quad \quad\overset{(\ref{pH-P})\,{\rm and}\,(\ref{wH-norm})}{<}(\frac{\norm{ \Pi_{\tilde{P}}^{\tilde{Q}_z}}_{L(X)}+\norm{\Pi_{\tilde{Q} }^{\tilde{P}_z}}_{L(X)}}{1-\norm{ \Pi_{\tilde{P}}^{\tilde{Q}_z}}_{L(X)}\cdot\lambda\cdot \tilde{\varepsilon}}) \cdot \lambda\cdot\tilde{\varepsilon}\\
&\quad\quad\quad\quad \quad\overset{(\ref{Pi-H-Pz-norm})\,{\rm and}\,(\ref{pH-Q})}{\leq}2(24\delta_{\Pi_Q^{P_K}}+13)\cdot\lambda\cdot\tilde{\varepsilon}.\end{aligned}$$ Then, one has that $d_{usual}(w, H\cap S)<2(24\delta_{\Pi_Q^{P_K}}+13)\cdot\lambda\cdot\tilde{\varepsilon}$ for any $w\in \tilde{Q}_z\cap S$ and hence, $\sup\limits_{v\in \tilde{Q}_z\cap S}\{d_{usual}(v, H\cap S)\}<2(24\delta_{\Pi_Q^{P_K}}+13)\cdot\tilde{\varepsilon}$. Therefore, we have proved (As-2).

\vskip 5mm
{\bf (As-3)} We assert that for any $\tilde{\varepsilon}\in (0, \min\{\frac{1}{4},\,\frac{1}{3}\cdot(\frac{3}{7})^{k-1}, \frac{1}{24\delta_{\Pi_Q^{P_K}}+14}\})$, ones have that for any $\tilde{z}\in \tilde{K}\cap B_{\tilde{\delta}^{\prime}_{\tilde{\varepsilon}}}(z)$,

\begin{equation}\label{rela-dis-dual-Qt-L}\sup\limits_{v\in \tilde{Q}_{\tilde{z}}\cap S}\{d_{usual}(v, \tilde{Q}_z\cap S)\}<\tilde{\varepsilon} \Rightarrow \sup\limits_{l_{\tilde{z}}\in L_{\tilde{z}}\cap S^*}\{d^*_{usual}(l_{\tilde{z}}, L_{z}\cap S^*)\}<M_{10}\cdot\tilde{\varepsilon},
\end{equation} and

\begin{equation}\label{rela-dis-dual-Q-Lt}\sup\limits_{v\in \tilde{Q}_z\cap S}\{d_{usual}(v, \tilde{Q}_{\tilde{z}}\cap S)\}<\tilde{\varepsilon} \Rightarrow \sup\limits_{l_z\in L_z\cap S^*}\{d^*_{usual}(l_z, L_{\tilde{z}}\cap S^*)\}<M_{10}\cdot\tilde{\varepsilon},
\end{equation} where $M_{10}=[6\delta_{\Pi_{Q}^{P_K}}+4+(9\delta_{\Pi_Q^{P_K}}+4.5)\cdot (k+12\delta_{\Pi_Q^{P_K}}+9)\cdot (\frac{7}{3})^{k-1}]\cdot (18\delta_{\Pi_Q^{P_K}}+9)\cdot k\cdot (\frac{7}{3})^{k-1}$, $S^*=\{l\in X^*:\,\norm{l}_*=1\}$ with the norm $\norm{l}_{*}=\sup\limits_{v\in S}\abs{l(v)}$, and $d^*_{usual}(l, X_0^*)=\inf\limits_{l^{\prime}\in X_{0}^*}\norm{l-l^{\prime}}_{*}$ for any given $l\in X^*$ and nonempty subset $X_0^*\subset X^*$ of $X^*$. Here, we only prove (\ref{rela-dis-dual-Qt-L}), and the proof for (\ref{rela-dis-dual-Q-Lt}) is similiar.  By utilizing (\ref{Def-t-ttCxN})-(\ref{first-layer-focu}), (\ref{kd-uni-bundle-invariant}), (\ref{loc-tilide-Q}) and (\ref{dzi-lower-b}), one has that for any $v\in X\cap S$ and $\tilde{z}\in \tilde{K}\cap B_{\delta^{\prime}_{\tilde{\varepsilon}}}(z)$ with $\tilde{\varepsilon}\in (0, \min\{\frac{1}{4},\,\frac{1}{3}\cdot(\frac{3}{7})^{k-1}, \frac{1}{24\delta_{\Pi_Q^{P_K}}+14}\})$, there are $\tilde{\lambda}_i\in\mathbb{R},\,\, i\in\{1,2,\cdots,k\}$ and $\tilde{w}_{\tilde{z}}\in\tilde{Q}_{\tilde{z}}$ such that 

\begin{equation}\label{norm-v-decom}v=\tilde{w}_{\tilde{z}}+\sum\limits_{i=1}^k \tilde{\lambda}_i v_{\tilde{z}}^i, \,\,{\rm with}\,\,\abs{\tilde{\lambda}_i}\leq (9\delta_{\Pi_Q^{P_K}}+4.5)\cdot(\frac{7}{3})^{k-1}\,\,{\rm and}\,\, \norm{\tilde{w}_{\tilde{z}}}\leq 6\delta_{\Pi_{Q}^{P_K}}+4\end{equation} Then,

\begin{equation}\label{l-pertu}\norm{l_{\tilde{z}}^i-l_z^i}_*=\sup\limits_{v\in X\cap S}\abs{l^i_{\tilde{z}}(v)-l_z^i(v)}\leq [6\delta_{\Pi_{Q}^{P_K}}+4+(9\delta_{\Pi_Q^{P_K}}+4.5)\cdot (k+12\delta_{\Pi_Q^{P_K}}+9)\cdot (\frac{7}{3})^{k-1}]\cdot \tilde{\varepsilon}
\end{equation} for any $i\in\{1,2,\cdots,k\}$ and $\tilde{z}\in \tilde{K}\cap B_{\tilde{\delta}^{\prime}_{\tilde{\varepsilon}}}(z)$ with $\tilde{\varepsilon}\in (0, \min\{\frac{1}{4},\,\frac{1}{3}\cdot(\frac{3}{7})^{k-1}, \frac{1}{24\delta_{\Pi_Q^{P_K}}+14}\})$. It then follows from (\ref{dzi-lower-b}), (\ref{lr}) and (\ref{l-pertu}) that for any $l_{\tilde{z}}\in L_{\tilde{z}}\cap S^*$ and $\tilde{z}\in \tilde{K}\cap B_{\tilde{\delta}^{\prime}_{\tilde{\varepsilon}}}(z)$ with $\tilde{\varepsilon}\in (0, \min\{\frac{1}{4},\,\frac{1}{3}\cdot(\frac{3}{7})^{k-1}, \frac{1}{24\delta_{\Pi_Q^{P_K}}+14}\})$,

\begin{equation}\label{l-dist-pert}\begin{aligned} &l_{\tilde{z}}=\sum\limits_{i=1}^{k}\frac{l_{\tilde{z}}(v_{\tilde{z}}^i) }{d_{\tilde{z}}^i}l_{\tilde{z}}^i\\
&{\rm and}\\
&\norm{l_{\tilde{z}}- \sum\limits_{i=1}^{k}\frac{l_{\tilde{z}}(v_{\tilde{z}}^i) }{d_{\tilde{z}}^i}l_{z}^i}_*\leq M^{\prime}_{10}\cdot\tilde{\varepsilon},
\end{aligned}
\end{equation} where $M^{\prime}_{10}=[6\delta_{\Pi_{Q}^{P_K}}+4+(9\delta_{\Pi_Q^{P_K}}+4.5)\cdot (k+12\delta_{\Pi_Q^{P_K}}+9)\cdot (\frac{7}{3})^{k-1}]\cdot (9\delta_{\Pi_Q^{P_K}}+4.5)\cdot k\cdot (\frac{7}{3})^{k-1}$. Let $l^{\prime}_z=\frac{\sum\limits_{i=1}^{k}\frac{l_{\tilde{z}}(v_{\tilde{z}}^i) }{d_{\tilde{z}}^i}l_{z}^i}{\norm{\sum\limits_{i=1}^{k}\frac{l_{\tilde{z}}(v_{\tilde{z}}^i) }{d_{\tilde{z}}^i}l_{z}^i}_*}$ and it is clear that $l^{\prime}_z\in L_z\cap S^*$. It entails that 

\begin{equation}\begin{aligned}&d^*_{usual}(l_{\tilde{z}},L_z\cap S^*)\leq \norm{l_{\tilde{z}}-l^{\prime}_z}_*\\
&\leq \norm{l_{\tilde{z}}- \sum\limits_{i=1}^{k}\frac{l_{\tilde{z}}(v_{\tilde{z}}^i) }{d_{\tilde{z}}^i}l_{z}^i}_* +\abs{\norm{ \sum\limits_{i=1}^{k}\frac{l_{\tilde{z}}(v_{\tilde{z}}^i) }{d_{\tilde{z}}^i}l_{z}^i}_*-1}\\
&\leq 2M_{10}^{\prime}\cdot\tilde{\varepsilon}=M_{10}\cdot\tilde{\varepsilon}\end{aligned}\end{equation} for any $l_{\tilde{z}}\in L_{\tilde{z}}\cap S^*$ and $\tilde{z}\in \tilde{K}\cap B_{\tilde{\delta}^{\prime}_{\tilde{\varepsilon}}}(z)$ with $\tilde{\varepsilon}\in (0, \min\{\frac{1}{4},\,\frac{1}{3}\cdot(\frac{3}{7})^{k-1}, \frac{1}{24\delta_{\Pi_Q^{P_K}}+14}\})$. Hence, (\ref{rela-dis-dual-Qt-L}) holds. Note that $z\in B_{\tilde{\delta}^{\prime}_{\tilde{\varepsilon}}}(z)$. We emphasize that arguments to prove (\ref{rela-dis-dual-Q-Lt}) is the same with the same canstant $M_{10}$. Therefore, we have proved this assertion (As-3).

Now, we prove the continuity of $\tilde{K}\times (\tilde{Q}_{z})$ by contrary. Suppose that there are a certain $z\in \tilde{K}$, a constant $\tilde{\tilde{\varepsilon}}>0$ and a sequence $\{z_n\}_{n=1}^{+\infty}\subset \tilde{K}\setminus\{z\}$ such that 

$$\lim\limits_{n\rightarrow +\infty}\norm{z_n-z}=0 \quad {\rm and}\quad d^*(L_{z_n},L_{z})>\tilde{\tilde{\varepsilon}}$$  where $d^*$ is the gap metric in the space $G(k, X^*)$. Then, there is a subsequnce $\{n_i\}_{i=1}^{\infty}\subset \mathbb{N}^+$ such that $\lim\limits_{i\rightarrow+\infty}n_i=\infty$ and one of the following two cases holds:
 (case I) $\sup\limits_{l_{z_{n_i}}\in L_{z_{n_i}}\cap S }d^*_{usual}(l_{z_{n_i}}, L_z\cap S)>\tilde{\tilde{\varepsilon}}$ for all $i\in\mathbb{N}^+$; and (case II) $\sup\limits_{l_{z}\in L_{z}\cap S }d^*_{usual}(l_z, L_{z_{n_i}}\cap S)>\tilde{\tilde{\varepsilon}}$ for all $i\in\mathbb{N}^+$. By virtue of assertion (As-2) and (As-3), there are $\varepsilon^{\prime}\in (0, \tilde{\tilde{\varepsilon}})$ and $v_{z_{n_i}}\in\tilde{Q}_{z_{n_i}}\cap S$ such that 
 
$$d_{usual}(v_{z_{n_i}}, \tilde{Q}_z\cap S)>\varepsilon^{\prime} $$ for all $i\in\mathbb{N}^+$. It then follows that 

$$\norm{\Pi_{\tilde{Q}}^{\tilde{P}_z}v_{z_{n_i}}}>\frac{\varepsilon^{\prime}}{3}$$ for all $i\in\mathbb{N}^+$. By virtue of (\ref{Def-t-ttCxN})-(\ref{first-layer-focu}) and (\ref{kd-uni-bundle-invariant}), one has $1=\norm{v_{z_{n_i}}}\geq\frac{\norm{\Pi^{\tilde{P}_z}_{\tilde{Q}}v_{z_{n_i}}}}{6\delta_{\Pi_Q^{P_K}}+3}$ and hence, $\norm{\Pi_{\tilde{P}}^{\tilde{Q}_z}v_{z_{n_i}}}\leq 6\delta_{\Pi_Q^{P_K}}+4$ for all $i\in\mathbb{N}^+$. Together with (\ref{inverse-G-upb}), the number

$$\delta_{Gnormp}=\inf\limits_{x\in \tilde{K}}\{m(G_x^{N^{\prime}}\mid_{\tilde{P}_x})\}\,\,{\rm is}\,\,\frac{1}{\delta_{G^{-1}_{M.}}},$$ where $m(G_x^{N^{\prime}}\mid_{\tilde{P}_x})=\min\limits_{v\in\tilde{P}_x\cap S}\{\norm{G^{N^{\prime}}_x v}\}$. Together with (\ref{kd-uni-bundle-invariant}), (\ref{Q-invariance}) and (\ref{Separation-G-N}), there is a $\tilde{N}\in\mathbb{N}^+$ such that 

\begin{equation}\label{enter-Pz} \frac{\norm{\Pi_{\tilde{P}}^{\tilde{Q}_{\tilde{F}^{N^{\prime}n}(z)}}G_z^{N^{\prime}n}v_{z_{n_i}}}}{\norm{\Pi_{\tilde{Q}}^{\tilde{P}_{\tilde{F}^{N^{\prime}n}(z)}}G_z^{N^{\prime}n}v_{z_{n_i}}}}<\frac{1}{2}\quad{\rm and}\quad \norm{\Pi_{\tilde{Q}}^{\tilde{P}_{\tilde{F}^{N^{\prime}n}(z)}}G_z^{N^{\prime}n}v_{z_{n_i}}}>\delta_{Gnormp}^n\cdot\frac{\varepsilon^{\prime}}{3}
\end{equation} for any integer $n\geq \tilde{N}$ and $i\in\mathbb{N}^+$. Thus,

\begin{equation}\label{Gz-Cz}G_z^{N^{\prime}n}v_{z_{n_i}}\in C^{\frac{1}{2}}_{\tilde{F}^{N^{\prime}n}(z)}\setminus\{0\} \end{equation} for any integer $n\geq \tilde{N}$ and $i\in\mathbb{N}^+$. By virtue of (\ref{Q-invariance}) and (\ref{loc-tilide-Q}), one has that 

\begin{equation}\label{C-out}G^{N^{\prime}\tilde{N}}_{z_{n_i}} v_{z_{n_i}}\in \tilde{Q}_{\tilde{F}^{N^{\prime}\tilde{N}}(z_{n_i}) }\subset (X\setminus C_{\tilde{F}^{N^{\prime}\tilde{N}}(z_{n_i})})\cup \{0\}\end{equation} for any $i\in \mathbb{N}^+$. It then follows from (\ref{Is-cf-cndelta}) and $\lim\limits_{i\rightarrow +\infty}z_{n_i}=z$ that  there is a $I\in\mathbb{N}^+$ such that 

\begin{equation}\label{Cprime-C}C^{\frac{1}{2}}_{\tilde{F}^{N^{\prime}\tilde{N}}(z)}\setminus\{0\} \subset C^{\frac{13}{17}}_{\tilde{F}^{N^{\prime}\tilde{N}}(z_{n_i})}\setminus\{0\}  \subset {\rm Int} C_{\tilde{F}^{N^{\prime}\tilde{N}}(z_{n_i})}\end{equation} for any $i\in\mathbb{N}^+$ with $i\geq I$. Together with (\ref{Gz-Cz}), (\ref{Cprime-C}) and the continuity of $\mathcal{G}$ on $\tilde{K}$, one has that $G^{N^{\prime}\tilde{N}}_{z_{n_i}}v_{z_{n_i}} \in {\rm Int}C_{\tilde{F}^{N^{\prime}\tilde{N}}(z_{n_i})}$ for any sufficiently large $i\in\mathbb{N}^+$, a contradiction to (\ref{C-out}). 

Thus, we have completed the proof of the continuity of $\tilde{K}\times (\tilde{Q}_x)$.

\vskip 5mm
\noindent{\bf Step 8. The unique $k$-exponential separation of $(\tilde{F},\mathcal{G})$ when $\tilde{K}$ is comapct.}
\vskip 5mm

See \cite[Lemma 4.3 and Remark 4.4]{F-Wang}. We should point out that the condition $\mathcal{G}(x), \,\,x\in\tilde{K}$ being compact, has not been used in the proof of \cite[Lemma 4.3]{F-Wang}. Thus, $(\tilde{F},\mathcal{G})$ on $\tilde{K}\times X$ has the unique $k$-exponential separation  if $\tilde{K}$ is comapct in additional.

\vskip 5mm
Therefore, we have completed the proof of the main theorem.
\end{proof}

\begin{rem} By the proof for the main theorem,  we can actually obtain the conclusion that the linear cocycle $(\tilde{F},\mathcal{G})$ on $\tilde{K}\times X$ has a $k$-exponential separation type property
if we replace the assumption $\tilde{F}$ and $\mathcal{G}$ are continous on $B_1(K)$ such that $\tilde{F}$ is a homoemorphism on $\tilde{K}$, as the weaken one that $\tilde{F}$ and $\mathcal{G}$ are maps on $\tilde{K}$ such that $\tilde{F}$ is a bijection on $\tilde{K}$.
\end{rem}

\begin{rem} Many difficulties in our proof for the the main theorem are caused by the base space being perturbated from $K$ to $\tilde{K}$. If we restrict the base space $K$ and force-map $F$ being not perturbated and the fibre-map-value map $\mathcal{T}$ is slightly perturbated as $\mathcal{G}$, then the linear cocycle $(F,\mathcal{G})$ on $K\times X$ having a $k$-exponential separation type property can be proved under the weaken condition that $K$ is a nonempty subset of $X$, $F$ is a map such that $FK=K$, $\mathcal{T}$ and $\mathcal{G}$ are just bounded maps on $K$. 
\end{rem}

\section{Proof of the corollary.}
\vskip 5mm
\begin{proof}(i) Note that the continuity of $\tilde{F},\,\tilde{F}^{-1},\,\mathcal{G}$ are not used in the first five steps of proof of the main theorem, and $\tilde{F}_{nnh}$ is a bijection on $O_{f,(\tilde{F}_{nnh},\,\tilde{K})}(x)_{\eta}$ for any $x\in\tilde{K}$ with a $\eta\in\mathcal{A}_x$. So, the first five steps of proof of the main theorem imply that $(\tilde{F}_{nnh},\tilde{\mathcal{G}})$ on $O_{f,(\tilde{F}_{nnh},\,\tilde{\tilde{K}})}(x)_{\eta} \times X$ has a $k$-exponential separation type property for each $x\in\tilde{K}$ and $\eta\in\mathcal{A}_x$.

(ii) Since $\tilde{F}_{nnh}$ possesses (${\bf H}^{\prime}$) property on $O_{f,(\tilde{F}_{nnh},\tilde{K})}(x)_{\eta}$ for the certain $x\in \tilde{K}$ and $\eta\in\mathcal{A}_x$, the inverse function thereom implies that $\tilde{F}_{nnh}$ is a homeomorphism on $O_{f,(\tilde{F}_{nnh},\tilde{K})}(x)_{\eta}$. Then, the first seven steps of proof of the main theorem imply that $(\tilde{F}_{nnh}, \tilde{\mathcal{G}})$ on $O_{f,(\tilde{F}_{nnh},\,\tilde{\tilde{K}})}(x)_{\eta} \times X$ admits a $k$-exponential separation for the certain $x\in \tilde{K}$ and $\eta\in\mathcal{A}_x$.

(iii)  Since $\tilde{F}_{nnh}$ possesses (${\bf H}^{\prime}$) property on the compact minimal set $\tilde{\tilde{K}}$, $\tilde{F}_{nnh}$ is a homeomorphism on the nonempty compact set $\tilde{\tilde{K}}$. It then follows from the main theorem that $(\tilde{F}_{nnh}, \tilde{\mathcal{G}})$ on $\tilde{\tilde{K}}\times X$ has the unique $k$-exponential separation.

Therefore, we have completed the proof.
\end{proof}

\section{Applications for dissipative systems under a small perturbation.}

In this section, we give a special array $K$, $F$, $\mathcal{T}$ and $\tilde{K}$, $\tilde{F}_{nnh}$, $\tilde{\mathcal{G}}$ as an example in applications.

Let  $F$ be a completely continuous and $C^1$-smooth self-map on the Banach space $X$. Assume that $F$ is point dissipative. It then follows from \cite[Theorem 2.4.7]{Hale} that there is a connected global attractor, denoted by $K$. By the definition of a global attractor under $F$ (see  \cite[P.17]{Hale}), $K$ is a maximal compact invariant set which attracts each bounded set $B\subset X$ under $F$. Let $\mathcal{T}(x)=D_xF$ for any $x\in X$. It is clear that $\mathcal{T}:X\mapsto L(X)$ is continuous. We remind that the assumption {\bf (H)} also holds for the special $K$, $F$ and $\mathcal{T}$ in this section.
\vskip 3mm
(${\bf H}^{\prime\prime}$) Assume that $\tilde{F}_{nnh}$ is of $C^1$-smooth on $X$ such that $\norm{\tilde{F}_{nnh}-F}_{C^1(X)}\leq \hat{\delta}$ for a number $\hat{\delta}\in\mathbb{R}^+$ small enough.
\vskip 3mm
Here, $\norm{\tilde{F}_{nnh}-F}_{C^1(X)}=\sup\limits_{x\in X}\{\norm{\tilde{F}_{nnh}(x)-F(x)}+\norm{D_x\tilde{F}_{nnh}-D_xF}_{L(X)}\}$. Let $\tilde{\mathcal{G}}: X\mapsto L(X)$ be defined as 

$$\tilde{\mathcal{G}}(x)= D_x\tilde{F}_{nnh}$$ for any $x\in X$. Let $\tilde{B}=B_1(K)$. Clearly, $\tilde{B}$ is bounded because of $K$ being compact. Let $\gamma$ be the Kuratowski's measure of noncompactness for the space $X$ (see \cite[Definition 7.1 in P.41]{Dei} and also \cite[P.13-14]{Hale} for related knowledge).

\begin{theorem}\label{Fnnh-ga-Lip}For any $\hat{\delta}\geq 0$, $\tilde{F}_{nnh}$ is a map such that $\gamma(\tilde{F}_{nnh} B)\leq \hat{\delta} \gamma( B)$ for any bounded set $B\subset X$.
\end{theorem}

\begin{proof} For any given bounded set $B\subset X$ and $\varepsilon_0>0$ small enough, there are finite sets $B_i,\,i\in\{1,\cdots,m\}$ such that $m\in\mathbb{N}^+$, the diameter $d(B_i)<\gamma(B)+\varepsilon_0$ and $B\subset \mathop{\cup}\limits_{i=1}^{m}B_i$. Cleary, $\gamma(F B_i)=0$ for any $i\in\{1,\cdots,m\}$ because of $F$ being completely continuous. Then, for any $\varepsilon_i>0$ small enough with $i\in\{1,\cdots,m\}$, there are finite sets $C_{i}^j,\,j\in\{1,\cdots,m_i\}$ such that $m_i\in\mathbb{N}^+$, the diameter $d(C_i^j)<\varepsilon_i$ and $F B_i\subset \mathop{\cup}\limits_{j=1}^{m_i}C_i^j$. Let $B_{i}^j=F^{-1}C_i^j\cap B_i$, where $F^{-1}C_i^j$ is to represent the preimage of $C_i^j$ under $F$. It then follows that 

$$B\subset \mathop{\cup}\limits_{i=1}^{m}\mathop{\cup}\limits_{j=1}^{m_i}B_{i}^j$$ with the diameter $d(B_i^j)<\gamma(B)+\varepsilon_0$ for any $j\in\{1,\cdots,m_i\}$ with $i\in\{1,\cdots,m\}$. Furthermore, ones have that the diameter

$$\begin{aligned}d(\tilde{F}_{nnh}B_i^j)&\leq \sup\limits_{x,y\in B_i^j}\{\norm{F(x)-F(y)}+\norm{\int_0^1D_{y+s(x-y)}F-D_{y+s(x-y)}\tilde{F}_{nnh}ds}_{L(X)}\cdot\norm{x-y}\}\\
&\overset{({\bf H}^{\prime\prime})}{\leq}\varepsilon_i+\hat{\delta}\cdot(\gamma(B)+\varepsilon_0).\end{aligned}$$ Together with the arbitrariness of choice of bounded set $B$, numbers $\varepsilon_0>0$ and $\varepsilon_i>0,\,i\in\{1,\cdots,m\}$, one has that 

$$\gamma(\tilde{F}_{nnh} B)\leq \hat{\delta} \gamma( B)$$ for any bounded set $B\subset X$. Therefore, we have completed the proof.

\end{proof}

\begin{theorem}\label{exis-attractor} For any sufficiently small $\delta\in (0,1)$, there is a local attractor $\tilde{K}\subset B_{\delta}(K)$ under $\tilde{F}_{nnh}$ attracting all bounded sets in $B_1(K)$ under $\tilde{F}_{nnh}$ as $\hat{\delta}\geq 0$ small enough.
\end{theorem}

\begin{proof} Since $K$ is a global atrractor under $F$, there is a $n_{\delta}\in\mathbb{N}^+\setminus\{1\}$ such that $F^n B_1(K)\subset B_{\frac{\delta}{2}}(K)$ for any intager $n\geq n_{\delta}$. Notice that $F$ is completely continuous. One has the convex hull $B_{cok}=\overline{co}\mathop{\cup}\limits_{i=1}^{n_{\delta}}F^i B_1(K)$ of $\mathop{\cup}\limits_{i=1}^{n_{\delta}}F^i B_1(K)$ is compact. It then follows the $C^1$-smoothness of $F$ that there is a $\delta^{\prime\prime}>0$ such that 
 
$$M_{11}=\sup\limits_{x\in B_{\delta^{\prime\prime}}(B_{cok})}\{\norm{F(x)}+\norm{D_xF}_{L(X)},1.01\}\in (1,+\infty).$$ It is clear that for any integer $2\leq i\leq n_{\delta}$, one has 

$$\begin{aligned}\sup\limits_{x\in B_1(K)}\{\norm{F^i(x)-\tilde{F}_{nnh}^i(x)} \}&\leq \sup\limits_{x\in B_1(K)}\{\norm{F\circ F^{i-1}(x)-F\circ\tilde{F}^{i-1}_{nnh}(x)}+\norm{F\circ\tilde{F}^{i-1}_{nnh}(x)-\tilde{F}^{i}_{nnh}(x)}\}\\
&\overset{({\bf H}^{\prime\prime})}{\leq}\sup\limits_{x\in B_1(K)}\{\norm{F\circ F^{i-1}(x)-F\circ\tilde{F}^{i-1}_{nnh}(x)}\}+\hat{\delta}.\end{aligned}$$ Furthermore, by mathematical induction, one has that for any integer $i\in [1,n_{\delta}]$,

\begin{equation}\label{tube-pertu}\sup\limits_{x\in B_1(K)}\{\norm{F^i(x)-\tilde{F}_{nnh}^i(x)} \}\leq\hat{\delta}\cdot(\mathop{\Sigma}\limits_{j=0}^{i-1}M_{11}^j)\end{equation} if $\hat{\delta}\in[0,\frac{M_{11}-1}{M_{11}^{n_{\delta}-1}-1}\cdot\delta^{\prime\prime}]$. Take 

$$\hat{\delta}\in[0, \min\{\frac{M_{11}-1}{M_{11}^{n_{\delta}-1}-1}\cdot\delta^{\prime\prime},\,\frac{M_{11}-1}{M_{11}^{n_{\delta}}-1}\cdot\frac{\delta}{2},\,1\}).$$ Then, 

\begin{equation}\label{condensing}\tilde{F}^{n_{\delta}}_{nnh}B_1(K)\subset B_{\delta}(K)\subset B_1(K).\end{equation} As a consequence, 

$$\begin{aligned}\tilde{F}_{nnh} \mathop{\cup}\limits_{i=1}^{n_{\delta}}\tilde{F}^{i}_{nnh}B_1(K)\subset\mathop{\cup}\limits_{i=1}^{n_{\delta}}\tilde{F}^{i}_{nnh}B_1(K)
\end{aligned}$$
By the comapctness of $F$ and (\ref{tube-pertu}), $\mathop{\cup}\limits_{i=1}^{n_{\delta}}\tilde{F}^{i}_{nnh}B_1(K)$ is a nonempty bounded set. It then follows that 

$$\overline{\mathop{\cup}\limits_{i=1}^{n_{\delta}}\tilde{F}^{i}_{nnh}B_1(K)}\supset\overline{\tilde{F}_{nnh} \mathop{\cup}\limits_{i=1}^{n_{\delta}}\tilde{F}^{i}_{nnh}B_1(K)}\supset\cdots\supset \overline{\tilde{F}^n_{nnh} \mathop{\cup}\limits_{i=1}^{n_{\delta}}\tilde{F}^{i}_{nnh}B_1(K)  }\supset\cdots,$$ where $\overline{\cdot}$ is used to represent taking the closure of a set in $X$. Together with theorem \ref{Fnnh-ga-Lip}, one has that 

$$ \tilde{K}=\mathop{\cap}\limits_{n=0}^{+\infty}\overline{\tilde{F}^n_{nnh}\mathop{\cup}\limits_{i=1}^{n_{\delta}}\tilde{F}^{i}_{nnh}B_1(K)}\subset B_{\delta}(K)$$ is nonempty and compact such that 

$$\tilde{F}_{nnh} \tilde{K}\subset \tilde{K}.$$ By the Definition of the $\omega$-limit set of a set under $\tilde{F}_{nnh}$ (see \cite[P.8]{Hale}), $\tilde{K}$ is the $\omega$-limit set $\omega(B_{1}(K))$ of $B_1(K)$ under $\tilde{F}_{nnh}$. 

Now, we prove that $\tilde{K}$ attracts $B_1(K)$ under $\tilde{F}_{nnh}$. It implies that $\tilde{K}$ attracts any subset of $B_1(K)$ under $\tilde{F}_{nnh}$. Prove by contrary. Suppose that there are a nondecreasing sequence of integers $\{n_j\}_{j=1}^{+\infty}\subset \mathbb{N}^+$, a sequence of points $\{x_j\}_{j=1}^{+\infty}\subset B_{1}(K)$ and a positive number $\varepsilon_0$ such that 

$$\lim\limits_{j\rightarrow+\infty}n_j=+\infty\quad {\rm and} \quad\tilde{F}_{nnh}^{n_j}x_j\notin B_{\varepsilon_0}(\tilde{K}).$$ By repeating the arguments in the proof of \cite[Lemma 2.3.5]{Hale}, we know that $\{\tilde{F}_{nnh}^{n_j}x_j\}_{j=1}^{+\infty}$ is precompact. Then, there is a its subsequence, denoted also by  $\{\tilde{F}_{nnh}^{n_j}x_j\}_{j=1}^{+\infty}$, such that $\lim\limits_{j\rightarrow+\infty}\tilde{F}_{nnh}^{n_j}x_j=x$ for some point $x\in \omega(B_{1}(K))=\tilde{K}$, a contradiction.

Finally, by utilizing \cite[Lemma 2.1.1]{Hale}, one has that $$\tilde{F}_{nnh}\tilde{K}=\tilde{K}.$$ Therefore, we have completed the proof.
\end{proof}

\begin{theorem} For any sufficiently small $\hat{\delta}\geq 0$ in {\rm(${\bf H^{\prime\prime}}$)}, 

{\rm (i)} The statements in the main Corollary hold for $(\tilde{F}_{nnh},\,\tilde{\mathcal{G}})$ on $\tilde{K}\times X$ in this section;

{\rm (ii)} $(\tilde{F}_{nnh},\,\tilde{\mathcal{G}})$ on $\tilde{K}\times X$ has the unique $k$-exponential separation when $\tilde{F}_{nnh}$ is injective on $\tilde{K}$. 
\end{theorem}

\begin{proof} By virtue of Theorem \ref{exis-attractor}, for sufficiently small $\delta\geq 0$, one can take sufficiently small $\hat{\delta}$ such that $\hat{\delta}\leq\delta$ and the local attractor $\tilde{K}$ under $\tilde{F}_{nnh}$ with $\tilde{K}\subset B_{\delta}(K)$. It is clear that for any $x\in \tilde{K}$, there is a $h(x)\in K$ such that $\norm{h(x)-x}\leq \delta$. This is the way to define a map $h:\tilde{K}\mapsto K$ in the assumption (\ref{Pertu-K-F}). One has that $\tilde{\varepsilon}=\norm{\mathcal{T}-\tilde{\mathcal{G}}}_{C_{KL(X)}}\leq \hat{\delta}\leq\delta$ by the chioce of $\hat{\delta}$. 

{\rm (i)}  It is clear that all statements in the main Corollary hold here.

{\rm (ii)} Since $\tilde{K}$ is a local attractor under $\tilde{F}_{nnh}$, ones have that $\tilde{K}$ is compact such that $\tilde{F}_{nnh}\tilde{K}=\tilde{K}$. Recall that $\tilde{F}_{nnh}$ is of $C^1$-smooth on $X$.  By virtue of the main theorem, it suffices to prove that $\tilde{F}_{nnh}$ is a homeomorphism on $\tilde{K}$ with the additional assumption $\tilde{F}_{nnh}$ being injective on $\tilde{K}$. Clearly, the inverse map of $\tilde{F}_{nnh}$ on $\tilde{K}$ exists, denoted by $\tilde{F}_{nnh}^{-1}$. It then suffices to prove that $\tilde{F}_{nnh}^{-1}$ is continuous on $\tilde{K}$. Prove by contrary. Suppose that there is a point $z\in\tilde{K}$, a sequence $\{z_n\}_{n\in\mathbb{N}^+}\subset \tilde{K}$ and a $\hat{\varepsilon}>0$ such that 

$$\lim\limits_{n\rightarrow +\infty}\norm{z_n-z}=0\quad{\rm and}\quad \tilde{F}^{-1}_{nnh}z_n\notin B_{\hat{\varepsilon}}(\tilde{F}_{nnh}^{-1}z),$$ where $B_{\hat{\varepsilon}}(\tilde{F}_{nnh}^{-1}z)=\{x\in X:\,\,\norm{x-\tilde{F}_{nnh}^{-1}z}\leq \hat{\varepsilon}\} $. By the compactness of $\tilde{K}$, there is a strictly increasing subsequence $\{n_i\}_{i\in\mathbb{N}^+}\subset \mathbb{N}^+$ and $z_{-1}\in\tilde{K}$ such that  $\lim\limits_{i\rightarrow +\infty}\tilde{F}^{-1}_{nnh}z_{n_i}=z_{-1}$ and hence, $z_{-1}\in\notin {\rm Int} B_{\hat{\varepsilon}}(\tilde{F}_{nnh}^{-1}z)$, where ${\rm Int} B_{\hat{\varepsilon}}(\tilde{F}_{nnh}^{-1}z)$ is the interior of $B_{\hat{\varepsilon}}(\tilde{F}_{nnh}^{-1}z)$ in $X$. It then follows from the $C^1$-smoothness of $\tilde{F}_{nnh}$ that $\tilde{F}_{nnh}z_{-1}=z$ and hence, $z_{-1}=\tilde{F}^{-1}_{nnh} z$, a contradiction to $z_{-1}\in\notin {\rm Int} B_{\hat{\varepsilon}}(\tilde{F}_{nnh}^{-1}z)$. Thus, $\tilde{F}_{nnh}$ is a homoemorphism on $\tilde{K}$ when $\tilde{F}_{nnh}$ is injective on $\tilde{K}$.

Therefore, we have completed the proof.
\end{proof}

\begin{rem} In this section, if the $\tilde{F}_{nnh}$ is the time-one map of the $C^1$-smooth semiflow which admits a flow extension (see \cite{F-1,F-2,F-W-W-1,F-W-W-2,F-W-W-3, S-Y}) on its local attractor $\tilde{K}$ in additional, then $\tilde{F}_{nnh}$ is a homeomorphism on $\tilde{K}$, that is also a scene for the main theorem.
\end{rem}

\end{document}